\documentclass[10pt, a4paper]{amsart}
\usepackage{lmodern}
\usepackage[utf8]{inputenc}
\usepackage{amsmath}
\usepackage{graphicx}
\usepackage{amssymb}
\usepackage{esint}
\usepackage[dvipsnames]{xcolor}
\usepackage{tikz}
\usepackage{xxcolor}
\usepackage{floatrow}
\usepackage{mathrsfs}
\usepackage{subcaption}
\usepackage{color}
\usepackage{amsthm}
\usepackage{epsfig}
\usepackage{pifont}
\usepackage[english]{babel}
\usepackage{hyperref}
\usepackage[normalem]{ulem}
\usepackage{tikz}
\usepackage{pgfplots}
\usepackage{subcaption}
\usepackage{caption}
\usepackage{bbm}

\usepackage{stmaryrd}

\usepackage{float}

\hypersetup{
	colorlinks,
	linkcolor={red!80!black},
	citecolor={blue!50!black},
	urlcolor={blue!80!black}
}
\usepackage{diagbox}

\usepackage[a4paper,top=1 cm,bottom=3 cm,left=1 cm,right=1 cm]{geometry}

\theoremstyle{plain}
\newtheorem{lemma}{Lemma}[section]
\newtheorem{proposition}[lemma]{Proposition}
\newtheorem{theorem}[lemma]{Theorem}

\theoremstyle{definition}

\theoremstyle{remark}
\newtheorem{remark}[lemma]{Remark}

\makeatletter
\renewenvironment{proof}[1][\proofname]{\par
  \pushQED{\qed}%
  \normalfont \topsep6\p@\@plus6\p@\relax
  \trivlist
  \item[\hskip\labelsep
        \itshape
    #1\ \ref{#1}.]\ignorespaces
}{%
  \popQED\endtrivlist\@endpefalse
}
\makeatother

\numberwithin{equation}{section}

\def\to{\rightarrow}

\usepackage{dsfont}

\newcommand{\I}{\mathbb{I}}

\newcommand{\K}{\mathbb{K}}

\renewcommand{\P}{\mathbb{P}}

\usepackage{mathtools}

\allowdisplaybreaks
\usepackage[colorinlistoftodos]{todonotes}
\usepackage{url}

\usepackage{graphicx,tikz} 

\author{Marius Bargo}
 \author{Yacouba Simporé}

\title[Analysis and stabilization]{Stabilization by Harvesting in Age-Structured Trophic Networks}
\date{}
\address[Marius BARGO]{Laboratoire LANIBIO, Université Joseph Ki-Zerbo (UJKZ), 01 BP 7021, Ouaga 01, Ouagadougou, Burkina Faso}
\email{bargomarius@gmail.com}

\address[Yacouba SIMPORE]{Chair for Dynamics, Control, Machine Learning and Numerics, Alexander Von Humboldt- Professorship, Department of Mathematics, Friedrich-Alexander-Universit\"{a}t at Erlangen-N\"{u}rnberg, Cauerstrasse 11, 91058 Erlangen, Germany, Université Yembila Abdoulaye TOGUYENI, Burkina Faso, Laboratoire LaST,  Laboratoire LANIBIO (UJKZ)}
\email{yacouba.simpore@fau.de}
\usepackage{thmbox}
\usepackage{dsfont}
\usepackage{float}
\usepackage{appendix}
\pgfplotsset{compat=1.18}
\usepackage{tikz}
\usetikzlibrary{arrows.meta}
\usetikzlibrary{arrows.meta, positioning} 
\usetikzlibrary{positioning} 
\usepackage{subcaption} 

\usepackage{float} 
\begin{document}
\maketitle

{\bf Abstract.} We propose and analyze a nonlinear age-structured multi-species model that serves as a unifying framework for ecological and biotechnological systems in complex environments (microbial communities, bioreactors, etc.). The formulation incorporates nonlocal intra- and interspecific interactions modulated by environmental covariates; under general assumptions on mortality, reproduction rates and interaction kernels, we establish existence, uniqueness and positivity of solutions. We illustrate the model’s practical relevance along two lines:  multi-species examples, notably a non-transitive (cyclic) competition model, for which we show that, under the model assumptions, a control applied to a single species can achieve global stabilization of the system; 
and  the population dynamics of malaria-vector mosquitoes, for which we develop two control strategies (biological and genetic) and, in the biological-control scenario, prove global asymptotic stability of the aquatic compartment by constructing an explicit Lyapunov function. Numerical simulations validate the theoretical results and compare the effectiveness of the proposed strategies in reducing vector density and malaria transmission.

\section{\bf Introduction}\label{s1}
\subsection{Introduction and  motivation}\label{s11}

Animal and plant species are organized into networks of intra- and interspecific interactions that shape ecological balance. Human activities can disrupt these interactions (mutual, competitive, or asymmetric) and alter community structure. Trophic networks (food webs) formalize these feeding relationships (“who eats whom”) and provide a framework for analysing energy transfer, population regulation, and ecosystem stability \cite{ref55,ref56,ref30}.

The first systematic mathematical formulation of predator–prey dynamics dates back to the independent works of Lotka and Volterra (1920s). Combining the Malthusian idea of exponential growth for the prey with the intraspecific competition concept later introduced by Verhulst, the Lotka–Volterra model reproduces cyclic oscillations and characteristic phase shifts between prey and predators. This elementary model provided the basis for countless subsequent generalizations \cite{ref3,ref5}. The standard system is written

$$
\begin{cases}
\dfrac{dX}{dt}=\;l\,X-p\,X\,Y,\\[6pt]
\dfrac{dY}{dt}=\;q\,X\,Y-m\,Y,
\end{cases}
$$

with $X(t)$ (prey density), $Y(t)$ (predator density) and $l,p,q,m>0$ biologically interpretable parameters. Despite its simplicity, this system exhibits neutral limit cycles and serves as a pedagogical and heuristic tool.

There are several ways to generalize the models to make them more realistic; three principal examples are given below :
\begin{itemize}
\item  Prey self-limitation. The exponential growth term $lX$ is often replaced by logistic growth $rX(1-X/K)$ to account for a carrying capacity $K$ and intraspecific competition. This modification prevents unbounded prey growth and can stabilise the dynamics (stable equilibria, damped cycles) \cite{ref3}.

\item  Functional response (consumption per predator). The predation term $pXY$ assumes that the per-predator ingestion rate increases without bound with prey density, ignoring handling time and satiation. One therefore replaces $pX$ by a functional response $f(X)$ and writes

$\frac{dX}{dt}=l\,X - f(X)\,Y.$

Holling systematised these forms and identified three major types (Type I, II and III) which correspond respectively to a linear response, a hyperbolic saturating response (handling time limitation), and a sigmoidal curve related to predator search behaviour or prey switching. These developments are essential to capture satiation and the reduced predation pressure at low prey densities \cite{ref29,ref8,ref35}.

\item  Dependence on predator density and alternative formulations. More complex formulations include dependence on $Y$ (predator-dependent models, e.g. Beddington–DeAngelis) or on the ratio $XY$ (ratio-dependent models, e.g. Arditi–Ginzburg):

Beddington–DeAngelis: $f(X,Y)=\dfrac{aX}{1+ahX+cY}$ (the $cY$ term models interference among predators) \cite{ref39}.

Arditi–Ginzburg (ratio-dependent): $f(X,Y)=\dfrac{aX}{bY+X}$, where the capture rate depends on the prey-to-predator ratio.
\end{itemize}
These variants express different ecological assumptions (interference, territoriality, allocation of search time) and profoundly affect species stability and coexistence.
A general two-species formulation is

$\frac{\partial X}{\partial t}=M(X,Y)\,X,\qquad \frac{\partial Y}{\partial t}=N(X,Y)\,Y.$

The signs of $\partial_Y M$ and $\partial_X N$ encode the nature of the interaction: Competition ($\partial_Y M<0$ and $\partial_X N<0$), Predation (or parasitism) ($\partial_Y M<0$ and $\partial_X N>0$), Mutualism ($\partial_Y M>0$ and $\partial_X N>0$).

This local view of the growth functions clarifies how each species directly influences the other’s instantaneous growth rate and serves to classify mathematically the interactions observed in ecology.

Predator–prey models form a flexible conceptual toolbox: the Lotka–Volterra model provides an analytic foundation, while generalizations (logistic growth, Holling functional responses, predator-dependent or ratio formulations) bring models closer to empirical observations and qualitatively alter dynamics (stability, cycles, coexistence). The explicit choice of a functional form is not neutral, it encodes testable ecological hypotheses and guides data interpretation and management decisions \cite{ref5,ref55,ref56,ref39,ref29,ref3,ref30,ref8,ref9,ref35}.

Competitive interactions, for their part, represent another fundamental class of biotic relations. Unlike predation (an asymmetric effect), competition induces mutually negative impacts on instantaneous growth rates: each species reduces the amount of resources or space available to the other. Theoretical competition models thus make it possible to explore three typical outcomes (competitive exclusion, stable coexistence, or history-dependence (priority effects)) depending on the relative intensity of intra- versus interspecific interactions. Complementary mechanisms (niche partitioning, differences in resource use, spatial structure, environmental fluctuations, colonization–competition trade-offs) promote coexistence in empirical systems, even where simple models predict exclusion. In spatially homogeneous models, one often distinguishes transitive competition (a strict hierarchical ranking of competitive abilities) from non-transitive competition (cyclic dominance, e.g. rock–paper–scissors). This distinction can qualitatively modify coexistence outcomes and stability properties. For instance, biological and biotechnological systems clearly illustrate these phenomena, as they often exhibit competitive dynamics cyclic or non-cyclic)  where the age of individuals plays a crucial role in interspecific interactions, frequently modeled through the law of mass action. Observable at multiple scales, from microbial communities to complex ecosystems, these dynamics call for mathematical modeling that accounts simultaneously for the demographic structure of populations and the non-local interactions between species.

Beyond species pairs, trophic networks (developed since  Charles Elton-1927) describe multispecies communities in which each node can be both predator and prey depending on trophic scale. These networks account for the flow of energy and energy loss (thermal dissipation) that limits food-chain length and shapes community structure. Multispecies trophic models allow evaluation of cascade effects, robustness to perturbations, and the consequences of anthropogenic change for biodiversity \cite{ref55,ref56,ref30}.

\begin{figure}[H]
\centering \includegraphics[width=0.21\textwidth]{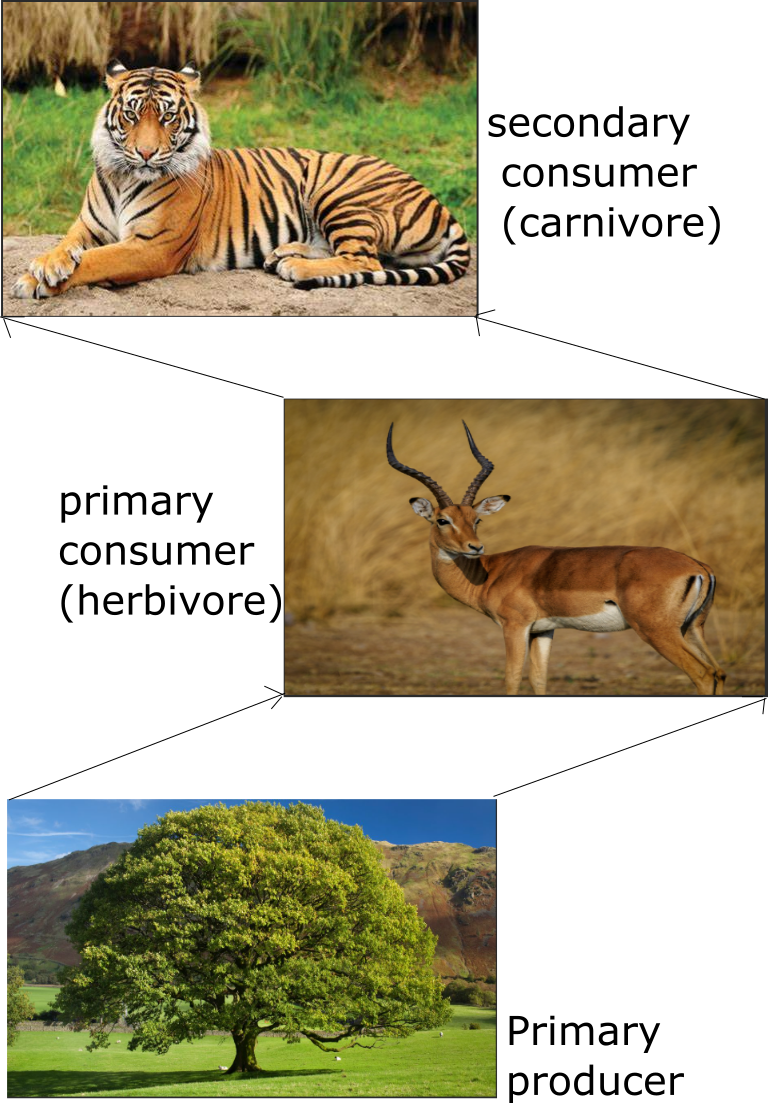}
\caption{A simple forest food chain illustrates the basic flow of energy: primary producers are consumed by herbivores, which in turn are eaten by carnivores. This represents a simplified version of nature’s food cycle, focusing on the main roles of producers and consumers.}
\end{figure}


A general framework for interacting populations is provided by Kolmogorov systems of the form
\begin{equation}\label{eq1.1}
\partial_t y_i= y_i f_i(y_1,\dots,y_N), \qquad i=1,\dots,N,
\end{equation}
so that each component $y_i f_i(Y)$ represents the net growth rate of population $i$ given the state $Y$ of all populations. This Kolmogorov framework generalizes the classical Lotka–Volterra predator–prey equations by allowing for more realistic interaction structures. Within this setting, several well-known models arise as particular cases, including the Rosenzweig–MacArthur predator–prey system with saturation, Gause’s competition models, the logistic (Verhulst) equation, and the exponential (Malthus) model \cite{ref3}.

To capture how species affect each other, we may express the influence of all other populations on population $i$ by a term of the form

\begin{equation}\label{eq1.2}
\partial_t y_i(a,t) + \partial_a y_i(a,t)
= -\mu_i(a)y_i(a,t) - \sum_{j=1}^N \int_0^{A_j} g_{ij}(a,\alpha)\,y_j(\alpha,t)\,d\alpha \; y_i(a,t),
\end{equation}
with renewal boundary conditions
\begin{equation}\label{}
y_i(0,t) = \int_0^{A_i} \beta_i(a)\,y_i(a,t)\,da, \qquad i=1,\dots,N,
\end{equation}
and initial data $x_i(a,0)=x_{i,0}(a)$. Here $\mu_i(a)$ is the age-dependent mortality of species $i$, $\beta_i(a)$ is its fertility kernel, and $g_{ij}(a,\alpha)$ encodes the effect of individuals of age $\alpha$ in species $j$ on those of age $a$ in species $i$. This formulation naturally extends both the classical McKendrick single-species model and Kolmogorov multispecies ODEs \eqref{eq1.1}.

To capture not only predation but also competition, parasitism, mutualism, and related interaction types, we consider model (\ref{eq1.3}), which represents a multispecies trophic network with diverse interaction structures and intensities. Unlike most classical formulations, this model explicitly incorporates age (or stage) structure, thereby linking demographic processes to interspecific interactions. Age-structured approaches are crucial, as survival, fecundity, and vulnerability often vary strongly across life stages, juveniles typically suffer higher mortality and lower reproduction, while adults contribute disproportionately to recruitment. For further details, the reader is invited to consult \cite{ref79} as an example.

In a multispecies context, these demographic details strongly shape community dynamics. Predators may preferentially consume specific age classes, and competition for resources can differ across stages (e.g. seedlings vs. adult trees). By embedding such effects, the framework generalizes classical models (such as Lotka–Volterra, Rosenzweig–MacArthur, or Gause’s competition systems) and recovers scenarios like non-transitive competition, shared predation, or multiple prey–predator interactions.

Structured multispecies models not only improve realism but also inform conservation and management. They help assess persistence, extinction risks, or outbreaks, and suggest targeted strategies, for example by focusing control on vulnerable life stages or strengthening key demographic groups in endangered species. Overall, the age-structured multispecies framework provides a versatile tool to analyze coexistence, resilience, and long-term ecosystem stability.

\subsection{Problem setting}\label{s12}

Let \(Y(a,t) = (y_1(a,t), \dots, y_N(a,t))^T\) be a nonnegative solution of the following age‐structured trophic network model:
\begin{equation}\label{eq1.3}
\begin{cases}
\partial_t Y(a,t) + \partial_a Y(a,t) + D(a)\,Y(a,t) + f\bigl(Y(a,t)\bigr)=0,
& \text{in}\;Q=(0,A)\times(0,T),\\
Y(0,t)=B\,Y(\cdot,t),   & \text{in}\;Q_T=(0,T),\\
Y(a,0)=Y_0(a),         & \text{in}\;Q_A=(0,A),
\end{cases}
\end{equation}
where
\begin{align}\label{eq1.4}
    Y_0(a) = \bigl(y_{01}(a),\dots,y_{0N}(a)\bigr)^T
\in
\mathcal{H} :=
\begin{cases}
\mathcal{H}_2=\bigl(L^2(0,A_i)\bigr)_{i=1}^N,\\
\mathcal{H}_1=\bigl(L^1(0,A_i)\bigr)_{i=1}^N,
\end{cases}
\quad Y_0(a)\ge0\ \text{a.e.}
\end{align}
The choice of the Banach space \(\mathcal H\) in \(L^1\) or \(L^2\) is standard in age‐structured population dynamics.

Notation and components :
\begin{itemize} 
\item \(y_i(a,t)\) denotes the density of species \(i\) at age \(a\) and time \(t\).  
\item \(B\) is the  renewal operator defining the boundary condition at \(a=0\).  
\item \(D(a)=\mathrm{diag}\bigl(\mu_1(a),\dots,\mu_N(a)\bigr)\) is the mortality matrix, where \(\mu_i(a)\) is the mortality rate of species \(i\).  
\item Each \(\beta_i(a)\) (implicitly contained in \(B\)) represents the fertility rate of species \(i\).  
\item \(A=\max_{1\le i\le N}A_i\) is the maximal life span among the \(N\) species.  
\item \(f:\mathcal H\to\mathcal H\) is a (possibly nonlinear) interaction vector field modeling trophic effects.
\end{itemize}

An interest in the non‑linear term $f$ in this model is that it allows the emergence of chaotic dynamics in ecology. Such chaos hampers long‑term forecasting, complicating resource management and species conservation. Hence, understanding these behaviors is essential to design robust strategies that account for unpredictable fluctuations and enhance ecosystem resilience.

In this paper, we adopt the following standing hypotheses (unless otherwise stated) :

$$
\textbf{(H1)}\begin{cases}
\mu_i(a)\geq 0\ \text{a.e.\ on }(0,A_i),\\
\displaystyle\mu_i\in L^1_{\rm loc}(0,A_i),\int_0^{A_i}\mu_i(a)\,da=+\infty,\quad \quad\textbf{(H2)}\begin{cases}
\beta_i\in L^\infty(0,A_i),\quad \beta_i(a)\ge0\ \text{a.e.\ on } (0,A_i),\\
\beta_i(a)=0\ \text{a.e.\ on }(A_i,A).
\end{cases}
\end{cases}
$$

that is, the mortality rates $\mu_i$ are positive, locally integrable, and strong enough that nobody lives past age $A_i$.

$$
\textbf{(H3)}\begin{cases}
f:\mathcal H\to\mathcal H\ \text{is globally Lipschitz and satisfies }f(0)=0,
\\
\exists\,L>0:\quad
\|f(Y_1)-f(Y_2)\|\le L\,\|Y_1-Y_2\|
\quad\forall\,Y_1,Y_2\in\mathcal H.\quad\textbf{(H4)}\begin{cases}
f
\text{ is Fréchet‐differentiable, and there exists}\\
M>0\text{ such that} \displaystyle
\sup_{Y\in\mathcal{H}_2}\,\bigl\|Df(Y)\bigr\|_{\mathcal{L}(\mathcal{H}_2,\mathcal{H}_2)}
\;\le\;M.
\end{cases}
\end{cases}
$$

This Lipschitz condition models saturation effects and ensures well‑posedness.

$$
\textbf{(H5)}\begin{cases}
B:\mathcal H\to \mathcal H\ \text{is (non)linear and globally Lipschitz, with the two typical cases:}
\\(i)\;  \text{Linear renewal}\; 
    BY=\int_0^A \beta(a)\,Y(a,t)\,da,
   \quad
   \beta(a)=\mathrm{diag}(\beta_1(a),\dots,\beta_N(a)).
\\(ii)\; \text{Nonlinear renewal}\;
     BY=\displaystyle\int_0^A \bar\beta\bigl(a,P(t)\bigr)\,Y(a,t)\,da,
   \quad
   P(t)=\Big(\!\!\int_0^{A_i}y_i(a,t)\,da\Big)_{i=1}^N
   \\\text{where}\; \bar\beta(a,p)\ge0 \;\text{is bounded and Lipschitz in}\; p:
   \|\bar\beta(a,p_1)-\bar\beta(a,p_2)\|\le K\,\|p_1-p_2\|.
\end{cases}
$$

\begin{remark}
Assumption (H5) provides a general formulation of the non-local birth term in age‑structured models, without resorting to a specific explicit form, and includes the particular cases treated in \cite{ref71,ref58,ref57, ref53}.
\end{remark}

The age-specific fertility function $\beta_i(a)$ encodes intrinsic biological factors such as maturation and senescence. In the absence of external constraints, $\beta_i(a)$ suffices to describe reproduction. When social interactions, environmental variability, or density-dependent effects are significant, one introduces an adjusted fertility $\bar\beta_i(a,P_1(t))$, where $P_1(t)$ measures population density. This extension captures phenomena such as resource competition, reproductive interference, and Allee effects at low densities. Accordingly, $\beta_i(a)$ represents the baseline physiological fertility, while $\bar\beta_i(a,P_1)$ incorporates external regulatory influences, providing a realistic, nonlinear description of reproductive dynamics.

For each of the $N$ independently evolving populations, define the  survival operator
\begin{align}\label{equation1.6}
\Pi(a) = \operatorname{diag}\bigl(\pi_1(a),\dots,\pi_N(a)\bigr),
\quad
\pi_i(a) = \exp\!\Bigl(-\!\int_0^a\!\mu_i(s)\,ds\Bigr),
\end{align}
and the net reproductive output
\begin{align}
    R = \operatorname{diag}\bigl(R_1,\dots,R_N\bigr),
\quad
R_i = \int_0^{A_i}\!\beta_i(a)\,\pi_i(a)\,da,
\end{align}
where \(A_i\) is the maximal age of population \(i\). Here, \(\pi_i(a)\) is the probability of surviving to age \(a\), and \(R_i\) is the expected number of offspring an individual of species \(i\) produces over its lifetime.

\subsection*{\bf Interpretation of the model}
System \eqref{eq1.3} captures the full spectrum of canonical biotic interactions by encoding species’ roles and pairwise (or higher-order) effects in the nonlinearity $f$: consumer–resource (prey–predator), interspecific competition, mutualism and non-transitive dynamics. To increase realism the model also includes intraspecific density-dependence (logistic-type regulation of fertility and mortality to prevent unbounded growth),  specialist versus generalist predation (specialists collapse with prey loss; generalists persist on alternative resources), and  collective effects such as pack hunting and group defence (per-capita predation rates depend on predator and prey group structure and size). By choosing appropriate functional forms for $f$ (mass-action, Holling II-III, saturating responses, etc.) one recovers standard modules: Lotka–Volterra pairs, two predators-one prey or one predator-two prey configurations, food-chain and food-web motifs (including competition and mutualism), and models incorporating prey defence.

\begin{remark}
The age-structured multi-population framework, via appropriate choices of $D(a)$ and $f$, provides a unified and flexible formalism to model and analyze stability, persistence, and control strategies in systems where demographic structure and non-local interactions are crucial. Although $\eqref{eq1.3}$ was originally devised for animal ecology, it admits a broader interpretation when the notions of “birth,” “death,” and “interaction” are redefined: “species” can then represent agents, firms, or states, and the same framework captures complex dynamics in epidemiology, forest ecosystems, economics, geopolitics, or chemistry \cite{ref74,ref65,ref66, ref67,ref53}.

\end{remark}
\section*{\bf Main results}
In this section we briefly describe our strategy and the principal results : 
\subsubsection*{\bf Global stability result - general $N$-species case}
We establish a global asymptotic stability result for a class of non-transitive age-structured competition models. We first establish the result for three and four species, and then generalize it, by construction, to any $N$-species system via a constructive backstepping design.

\subsubsection*{\bf Global stabilization result – mosquito control}
We prove global stabilization of a non-autonomous logistic mosquito population model using a reduced control strategy acting exclusively on the aquatic stage.

\section*{\bf Related works and novelty}
\subsection*{Let us mention some related works from the literature}
To introduce this literature review, note that although predator–prey models have a long history, age‐structured and especially multispecies formulations remain underexplored.  The classical Lotka–Volterra model dates back to the 1920s, and surveys of non‐age‐structured predator–prey dynamics can be found in \cite{ref5,ref3}.  An apparently novel age‐structured predator–prey model was proposed in \cite{ref45}, combining a mass‐action functional response with a Leslie‐type numerical response (cf.\ Leslie 1948 in \cite{ref3}).  However, that model assumes a specialist predator, an unrealistic assumption in most ecosystems, where generalist feeding strategies prevail.  Accordingly, Holling’s functional responses (and their generalizations) provide more biologically realistic formulations.

Many studies employ Lyapunov functions to demonstrate the stability of predator–prey models. Yet the complexity of the equations often forces one to work with simplified (and sometimes unrealistic) versions of the model in order to construct these functions analytically. Despite these simplifications, these works shed light on the so‑called “species harvesting” technique, which guarantees the positivity of solutions and their global asymptotic stability around a nonzero trajectory. This approach offers promising avenues for the development of conservation strategies. Structured models of competitive or predator-prey type have been extensively studied in the literature; for example, see the recent works of Carina Veil et al. \cite{ref45,ref77}. In \cite{ref45}, stabilization of the two-species predator--prey model is achieved by introducing an additional control term that specifically strengthens the predator dynamics. In \cite{ref77}, which addresses a two-species competition model, stabilization via the backstepping method is obtained without recourse to fictitious controls, owing to the small number of species considered and the particular structure of the model. Moreover, the issues of well‑posedness and stability analysis for an age‑structured predator–prey model are examined in detail in \cite{ref51, ref50}.

In addition to work on age-structured predator–prey systems, the literature includes contributions on single-species age-structured models, focusing on the analysis of asymptotic stability and the design of robust control laws, as in \cite{ref70,ref69,ref78}.

In the context of vector control, several age‐structured mosquito models incorporate biological control methods, such as the sterile insect technique and cytoplasmic incompatibility, to account for larval and adult stages \cite{ref38, ref42, ref44, ref43, ref41}.  A limited number of works (e.g.\ \cite{ref52}) also include predator–mosquito interactions, though it remains challenging to model a single predator that effectively regulates both aquatic and adult mosquito populations.
\subsection*{Contributions}
\begin{itemize}
\item  Under suitable assumptions on fertility, mortality, and the interaction function, we establish well‑posedness results for multiple system configurations. In this section, we have established a threshold for resilience and/or the maintenance of biodiversity. Under {\bf assumption (H5-i)}, we prove existence, uniqueness, continuity and positivity of solutions for $f\le0$ (Theorems \ref{th2.5}–\ref{th2.6}).  Moreover, if $f\ge0$ but the initial datum satisfies

\begin{align}
Y_0(a)\;\ge\;\Gamma(a+t;a)\;>\;0
      \quad\text{a.e. }(a,t)\in(0,A)\times(0,T),
\end{align}
then the same results hold (see Remark \ref{re2.8} for the definition of $\Gamma$).
\item We study asymptotic and exponential stability under hypotheses on a reproduction number that includes local interactions and age structure, and we analyze symmetric versus asymmetric cases in multi-species models, accounting for pressures (competition, predation) tied to $R_0$.
\item We examine stability under symmetric versus asymmetric reproduction numbers in multi-species models and show that stability depends more on the arrangement and strength of interspecific interactions and external drivers than on species richness. Thus, representing interaction topology is essential for realistic prediction and resilience assessment. 
\item  Non-transitive competition, analogous to the “rock–paper–scissors” game, refers to a cyclic interaction network in which no species is globally dominant. Such dynamics tend to promote cyclic coexistence and complicate the natural stabilization of the system. Our work demonstrates that a single multiplicative control applied to one species is sufficient to globally stabilize the multi-species model under non-transitive competition. This control acts directly on the targeted species and indirectly influences the others through the cyclic network, ensuring convergence of the entire system toward a stable equilibrium.

The methodology relies on a constructive design: fictitious controls are successively defined to partially stabilize subsystems of the model until the final global control is synthesized. Each fictitious control targets a specific state or subsystem while preserving the cyclic structure of the system. This procedure allows for the explicit construction of a common Lyapunov function, guaranteeing the asymptotic convergence of all species while preserving essential biological properties (positivity and boundedness of solutions).

In this work, we go beyond the two-species framework to address multi-species stabilization, focusing specifically on non-transitive competition models with $N$ species. Existing methods for two species exploit low dimensionality and particular structural features that do not generalize easily. Increasing the number of species and introducing non-transitive interactions lead to cross-couplings and nonlinear terms, whose control requires a design that exploits the topological properties of the model and careful estimation of the interaction terms.

{\bf The novelty of this work lies in several aspects:}

\begin{enumerate}
\item[*] Single control for $N$ species: we show that a multiplicative control applied to one species ensures the global stabilization of the system for any number of non-transitive species. The validity is first established for three and then four species, and extended to the general case through a constructive backstepping design.

\item[*]  Incremental backstepping procedure: the design relies on fictitious controls that gradually stabilize each subsystem while preserving the model’s topology. This backstepping scheme leads to the stepwise synthesis of a global Lyapunov function guaranteeing asymptotic stability.

\item[*] This strategy is novel in the literature. Indeed, our approach preserves the non-transitive structure of species interactions: the applied control does not alter the cyclic dominance relations, while it qualitatively modifies the system dynamics so that trajectories converge to a globally stable equilibrium. Unlike existing studies that focus either on sustaining cycles or on directly stabilizing an equilibrium, our method achieves both simultaneously. Moreover, it guarantees positivity and boundedness of populations, thereby respecting ecological constraints and preserving the model’s cyclic topology.

\item[*] Novel extension beyond two species: to our knowledge, very few results in the literature address the global stabilization of an arbitrary number of age-structured non-transitive species. Our framework fills this gap by providing a minimalist and cost-effective solution  (a single control applied to one species, sufficient to stabilize the entire system) suitable for biological or management applications.
\end{enumerate}
The combination of these elements yields a robust global stabilization strategy, ensuring convergence of any multi-species non-transitive system to equilibrium without compromising its cyclic dynamics. This innovative backstepping approach offers a minimalist and feasible control, opening new perspectives for the regulation of complex biological communities.

\item We establish global asymptotic stability for an age-structured, non-autonomous logistic model of mosquito populations (aquatic and adult stages) via a bounded multiplicative control acting on the aquatic compartment. Since the adult population arises from the emergence of individuals from the aquatic stage, the proposed control is shown to be necessary and sufficient to stabilize the vector dynamics. The control $P(t)$, referenced to the static total aquatic population $k_I$, combines feedforward and feedback components to compensate for parameter variability and ensure stability of the aquatic subsystem. The reference parameter $k_I$ governs the trade-off between responsiveness and the magnitude of the correction.

\item We perform numerical simulations that corroborate the theoretical results obtained in our work on the stabilization of multi-species models (in particular, non-transitive competition models), as well as on multi-phase transition models describing mosquito dynamics.
\end{itemize}
\section*{\bf Organization}
Our work is organized as follows. {\bf Section \ref{s2}} is devoted to the well-posedness of the model. In particular, we establish well-posedness in three settings: the demographic case ($f\equiv 0$), the linear nonlocal case, and the nonlinear nonlocal case. A special instance is discussed in Remark \ref{re2.8}, which further clarifies its relation to Theorems \ref{th2.5} and \ref{th2.6}. {\bf Section \ref{s3}}, devoted to applications, addresses design and stabilization. After a general introduction to multi-species models, we propose a multi-species backstepping control strategy for a non-transitive competition model. We then present a control strategy for a four-compartment mosquito dynamics model (aquatic stages, juvenile females, mature females and wild males) including the release of genetically modified individuals. Numerical simulations illustrate the evolution of each population with respect to age and time. Section \ref{s4} is devoted to the conclusion and outlines directions for future research.

\section{\bf Mathematical Analysis}\label{s2}
Studying the well‑posedness of model \eqref{eq1.3} is essential to guarantee it faithfully reflects the inherent properties of the underlying phenomena. Following Hadamard’s definition for Cauchy problems, we require existence, uniqueness, and continuous dependence on the initial data. In our age‑structured, multi‑species framework with nonlinear functional response $f$, we further demand that solutions remain nonnegative and extend globally in time.

To cast system \eqref{eq1.3} into an abstract Cauchy problem, we work in the Hilbert space $\mathcal H_2$. Under hypotheses $(H1)$–$(H4)$ and for any fixed initial datum $Y_0\in\mathcal H_2$ (or $\mathcal H_1$), define

\begin{align}\label{op1}
 \mathcal A : D(\mathcal A)\subset\mathcal H_2\longrightarrow\mathcal H_2,
   \quad \mathcal A\varphi=-\partial_a\varphi - D(a)\varphi,
\end{align}

with

\begin{align}
D(\mathcal A)
   = \Bigl\{\varphi\in\mathcal H_2 :
     \varphi\text{ a.c. on }[0,A),\;
    \varphi(0)=\!\int_0^A\beta(a)\varphi(a)\,da,\;
     -\partial_a\varphi - D(a)\varphi\in\mathcal H_2
   \Bigr\}.
\end{align}

In block‑diagonal form,

\begin{align}
 \mathcal A
   = \mathrm{diag}\bigl(-\partial_a-\mu_1(a),\dots,-\partial_a-\mu_N(a)\bigr),
   \quad a\in(0,A).
\end{align}

We recall a fundamental notion concerning the semigroups associated with this operator.
\begin{lemma}\label{le2.1}
The operator $(\mathcal{A}, D(\mathcal{A}))$ is the infinitesimal generator of a strongly continuous semigroup $\mathcal{T}= (\mathcal{T}_t)_{t\geq 0}$ on $\mathcal{H}_2.$
\end{lemma}
\begin{proof}[of Lemma \ref{le2.1}]
It is well known (e.g., \cite{ref60, ref59}) that $\mathcal{A}_i$ is the infinitesimal generator of a strongly continuous semigroup on $L^2(0,A_i).$ Since each of its elements is an infinitesimal generator of a semigroup, $\mathcal{A}$ is itself an infinitesimal generator of a semigroup. 
\end{proof}

For similar operators, the reader may refer to the examples in \cite{ref28, ref48}. We now introduce the adjoint operator $\mathcal A^*$ of $\mathcal A$.  Define

$$
   \mathcal A^*: D(\mathcal A^*)\subset\mathcal H_2\longrightarrow\mathcal H_2,
   \qquad
   \mathcal A^*\eta = \partial_a\eta - D(a)\,\eta + \beta(a)\,\eta(0),
 $$

 with

$$
   D(\mathcal A^*) = \Bigl\{\eta\in\mathcal H_2 : 
     \eta\text{ is a.c. on }[0,A),\;
     \lim_{a\to A}\eta(a)=0,\;
     \partial_a\eta - D(a)\eta\in\mathcal H_2
   \Bigr\}.
 $$

By integration by parts one checks that

$$
   \langle \mathcal A Y,\eta\rangle 
   = \langle Y,\mathcal A^*\eta\rangle
   \quad\text{for all }Y\in D(\mathcal A),\;\eta\in D(\mathcal A^*).
$$
This semigroup framework thus provides a natural dynamical‐systems interpretation of the population model in the state space $\mathcal H_2$.

\subsection{Model without species interaction}
Consider the following age‑structured system with $f\equiv0$:
 \begin{equation}\label{eq2.4}
\left\lbrace 
\begin{array}{ll}
\partial_tY(a,t)+\partial_aY(a,t)+D(a)Y(a,t)=0 &\text{ in }Q,\\  
Y(0,t)=BY, &\text{ in } Q_T,\\
Y(a,0)=Y_0 &\text{ in } Q_A,
\end{array}
\right.
\end{equation} 
 where $Y_0\in\mathcal H_2$. This is the classical Lotka–McKendrick (or Lotka–von Foerster) model for $N$ non‑interacting species.  Although it omits ecological feedbacks, it is mathematically well‑posed and serves as the prototype for age‑structured dynamics (see \cite{ref25, ref28, ref26, ref69, ref30}).  In particular, Inaba \cite[Proposition 2.4]{ref28} introduced exactly this “multi‑state” version as a stable‑population process.
\begin{remark}
In this stable population model, interactions can be accounted for by introducing couplings either in the operator $D(a)$ (which therefore ceases to be diagonal) or in the nonlocal term $B$, thus endowing system \eqref{eq2.4} with a nonlinear character.    
\end{remark}
Equivalently, one rewrites \eqref{eq2.4} as the abstract Cauchy problem
\begin{equation}\label{eq2.5}
\left\lbrace 
\begin{array}{ll}
\partial_tY(a,t)=\mathcal{A}Y(a,t) &\text{ in }Q,\\  
Y(a,0)=Y_0 &\text{ in } Q_A.
\end{array}
\right.
\end{equation}  
 in the state space $\mathcal H_2$.  

 This vector‑valued Lotka–von Foerster system appears in many applications, including multi‑regional demography, two‑sex linear population models \cite{ref54}, and budding‑yeast population dynamics \cite{ref28}.  It provides the basic framework for describing the evolution of an age‐structured population under prescribed boundary and initial conditions.

 We now state the following fundamental proposition on existence and uniqueness for system \eqref{eq2.4}:
\begin{proposition}\label{pr2.2}
Let us $Y_0\in D(\mathcal{A}), Y_0\geq 0,$ assumptions  $(H1)-(H2)$ satisfied, the system \eqref{eq2.4} is well-posedness (with $\mathcal{A}$ linear). Moreover, the unique solution  is given by $Y(.,t)=\mathcal{T}(t)Y_0(.),$ for all $Y_0\in D(\mathcal{A}).$
\end{proposition}
\begin{proof}[of Proposition \ref{pr2.2}]
   The operator $\mathcal{A}$ is the infinitesimal generator of a $C_0$ semigroup of contraction on $\mathcal{H}_2$ \cite{ref28}. Then, from \cite[Theorem I]{ref32}, \cite{ref33} the proof is immediate.
\end{proof}
 Through the characteristic lines, the solution $Y$ to \eqref{eq2.4} satisfies
 \begin{equation}\label{eq2.6}
Y(a,t)=\left\lbrace
\begin{array}{ll}
\dfrac{\Pi(a)}{\Pi(a-t)}Y_{0}(a-t)\quad t\leq a,\\
\\\Pi(a)b(t-a)\quad t> a,
\end{array}
\right.
\end{equation}
and
 \begin{equation}\label{eq2.7}
b(t)=\displaystyle\int_0^t\beta(a)\Pi(a)b(t-a)da+\displaystyle\int_t^A\beta(a)\dfrac{\Pi(a)}{\Pi(a-t)}Y_{0}(a-t)da
\end{equation}
is a linear Volterra integral equation.
\begin{proposition}\label{pr2.3}
    Let us $Y_0\in \mathcal{H}_1, Y_0\geq 0,$ assumptions $(H1)-(H2)$ satisfied, the system \eqref{eq2.4} admits a unique solution in $C((0,T) ;\mathcal{H}_1)$ such that
    \begin{itemize}
        \item $\Vert Y(.,t)\Vert_{\mathcal{H}_1}\leq Me^{(\Vert\beta\Vert_{L^\infty}(M+w)t}\Vert Y_0\Vert_{\mathcal{H}_1},\qquad t \in (0,T).$
    \end{itemize}
 \end{proposition}
 \begin{proof}[of Proposition \ref{pr2.3}]
  As demonstrated in \cite[Theorem 4]{ref48},  \cite[Theorem 4.3]{ref31}, the same technique can be effectively utilized to establish this proposition.
 \end{proof}
 \begin{remark}
     The constants \(M\) and \(w\) arise from the estimation of the semi group in \cite[Lemma p.19]{ref33}, \cite[Lemma B]{ref28}.
 \end{remark}

The linear case is valuable: one can fully describe its solutions, yet important questions remain—such as finite‑time blow‑up \cite{ref22}. If we cast model \eqref{eq1.3} in linear form (i.e.\ with a linear function $f$), it is certainly well‑posed mathematically, although it lacks biological realism. This simplification treats multiple species as non‑interacting, which contradicts the fundamental role of species interactions in population regulation, species distribution, and overall ecosystem structure. Ignoring these interdependencies can therefore yield inaccurate predictions, since real‑world dynamics are driven by numerous coupled factors.

Nevertheless, in a resource‑management context, such a linear model can serve as a prototype for harvesting a single species. It can be used to explore goals such as preventing extinction, maintaining population stability, or optimizing sustainable yields. In this setting, one would choose $f$ to capture harvesting effects, thereby enabling the analysis of different management or exploitation strategies.

\subsection{Case of a linear nonlocal term}
In this section, only the function $f$ is nonlinear and the renewal equation satisfies the condition (H5-i). The following  result  assures the existence and uniqueness of mild solutions of (\ref{eq1.3}) for Lipschitz continuous functions $f.$ An example of a typical age-structured model can be found in \cite{ref48}.
\begin{theorem}\label{th2.5}
Assume that hypotheses \((H1)-(H2)-(H3)-(H5-i)\) hold. The  operator $\mathcal{A}$  is the infinitesimal generator of a $C_0$ semigroup $\mathcal{T}(t),\quad t\geq 0,$ on $\mathcal{H}_2,$ then for every $Y_0\in \mathcal{H}_2,$ the system \eqref{eq1.3} has a unique mild solution $Y\in C((0,T) ;  \mathcal{H}_2)^N.$
    Moreover, the mapping $Y_0\longrightarrow Y$  is Lipschitz continuous from $ \mathcal{H}_2$ into $C((0,T) ; \mathcal{H}_2)^N,$ and if $Y_0$ satisfies condition \eqref{eq2.49}, then the solution remains non-negative for any  function $f.$
\end{theorem}
The map of mild solution to \eqref{eq1.3} is 
\begin{align}\label{eq2.8}
Y(.,t)=\mathcal{T}(t)\bigl(y_{01}(.),\dots,y_{0N}(.)\bigr)^T+\displaystyle\int_0^t\mathcal{T}(t-s)\lfloor -f(Y(.,s))\rfloor ds,
\end{align}
    where $\mathcal{T}$ is the semi group generates by $\mathcal{A},\; f$  the nonlinear function composed of the $f_i.$
    
    \begin{proof}[of Theorem \ref{th2.5}]
{\bf Existence} :    For every $Y\in \mathcal{X}$, the function  $F$ defined by 
\begin{align}\label{eq2.9}
(FY)(t)=\mathcal{T}(t)\bigl(y_{01}(.),\dots,y_{0N}(.)\bigr)^T+\displaystyle\int_0^t\mathcal{T}(t-s)\lfloor -f(Y(.,s))\rfloor ds,\qquad t\in (0,T),
\end{align}
belongs to $$\mathcal{X}=C((0,T);\mathcal{H}_2)$$ with the same norm as defined below in the proof of Theorem \ref{th2.6}. Let us $Y_1,Y_2$ solution to \eqref{eq1.3} we get from \eqref{eq2.9} the  following  
\begin{align}
(FY_1)(t)-(FY_2)(t)=\displaystyle\int_0^t\mathcal{T}(t-s)\lfloor -f(Y_1(.,s))+f(Y_2(.,s))\rfloor ds,\qquad t\in (0,T),
\end{align}
and thanks to $(H3),$
\begin{align}
\Vert (FY_1)-(FY_2)\Vert_{\mathcal{X}}\leq MLT\Vert Y_1-Y_2\Vert_{\mathcal{X}}
\end{align}
and induction on $n$ iterations  it follows easily that 
\begin{align}
\Vert (F^nY_1)-(F^nY_2)\Vert_{\mathcal{X}}\leq\dfrac{( MLT)^n}{n!}\Vert Y_1-Y_2\Vert_{\mathcal{X}}.
\end{align}
Hence, for $n$ large enough such that $\dfrac{( MLT)^n}{n!}<1,$ it follows that   $F$ is a contraction and  by the Banach fixed point, $F$ admits a fixed point $Y\in C((0,T) ; \mathcal{H}_2).$ 
\\{\bf Uniqueness and continuity} : let $Y$ and $\bar{Y}$ be two solutions of \eqref{eq1.3} and from \eqref{eq2.9} we deduce that
\begin{align}
    \Vert Y(t)-\bar{Y}(t)\Vert_{\mathcal{H}_2}\leq Me^{wt}\Vert Y_0-\bar{Y}_0\Vert_{\mathcal{H}_2}+MLe^{wt}\displaystyle\int_0^t e^{-ws}\Vert Y(s)-\bar{Y}(s)\Vert_{\mathcal{H}_2}ds\quad\text{for every}\quad t\in (0,T)
\end{align}
with $Me^{wt}$ the bound of $\Vert\mathcal{T}(t)\Vert$ and from Gronwall's inequality 
\begin{align}\label{eq2.14}
    \Vert Y-\bar{Y}\Vert_{\mathcal{X}}\leq Me^{MLT}\Vert Y_0-\bar{Y}_0\Vert_{\mathcal{H}_2}.
\end{align}
The uniqueness and continuity follows from \eqref{eq2.14}.
\end{proof}
In \cite[Theorem 2.1.1]{ref21}, the well-posedness of the Lotka-McKendrick model without diffusion has been established, ensuring the existence of a non-negative solution. This is due to the positivity of the initial data and the kernel. We have the following theorem. 
\begin{theorem}\label{th2.6}
Under assumptions $(H1)-(H2)-(H3)-(H5-i)$, for every $Y_0\in \mathcal{H}_1,$ with $Y_0\geq 0$ and $f\leq 0$ globally Lipschitzian, system \eqref{eq1.3} admits a unique solution non-negative  $Y$ in $C((0, T), \mathcal{H}_1)^ N.$ Furthermore, if $Y_0$ satisfies condition \eqref{eq2.49}, then the solution remains non-negative for any  function $f.$
\end{theorem}
Integrating the system \eqref{eq1.3} along the characteristic curves $a-t=t_0,$ we obtain implicit formulas for its solutions stated below
\begin{equation}
Y(a,t)=\left\lbrace
\begin{array}{ll}
\dfrac{\Pi(a)}{\Pi(a-t)}Y_{0}(a-t)+\displaystyle\int_{a-t}^{a}\dfrac{\Pi(a)}{\Pi(z)}\lfloor -f(Y(z,z-t_0))\rfloor dz\quad t\leq a,\\
\\\Pi(a)b(t-a)+\displaystyle\int_{0}^{a}\dfrac{\Pi(a)}{\Pi(z)}\lfloor -f(Y(z,z+t_0))\rfloor dz\quad t> a.
\end{array}
\right.
\end{equation}
Here, $b(t)=Y(0,t)$ plays the role of the renewal term, and by Remark 2.8(ii) the corresponding renewal equation admits a unique solution. The operator $\Pi(\cdot)$ and the initial datum $Y_{0}$ are defined in \eqref{eq1.4} and \eqref{equation1.6}, respectively. 
\begin{proof}[of Theorem \ref{th2.6}]
Let us fixed $\bar{Y}\in \mathcal{H}_1$ and define the mapping $\delta$ such that $\delta(\bar{Y})=Y(a,t).$ Consider in $\mathcal{X}=C((0;T); \mathcal{H}_1)$ the norm 
\begin{align*}
\Vert \bar{Y}\Vert_{\mathcal{X}}=\sup_{t\in(0,T)}e^{-\lambda t}\Vert \bar{Y}\Vert_{\mathcal{H}_1},\quad\;\text{for any}\;\bar{Y}(.,t)\in \mathcal{H}_1,
\end{align*}
which is equivalent to the usual norm in $\mathcal{H}_1$ with $\lambda,$ a positive constant that will
be made precise later. 

 On the one hand, for all $(a, t)\in(0,A)\times(0,T)$ such that $t\leq a,$ we have 
\begin{align}
    \displaystyle\int_t^A\delta(\bar{Y})(a,t)da=\displaystyle\int_0^{A-t}\dfrac{\Pi(s+t)}{\Pi(s)}Y_{0}(s)ds+\displaystyle\int_t^A\displaystyle\int_{a-t}^{a}\dfrac{\Pi(a)}{\Pi(z)}\lfloor -f(\bar{Y}(z,z-a+t))\rfloor dzda
\end{align}
thanks to $(H3),$ we obtain 
\begin{align}\label{eq2.17}
    e^{-\lambda t}\displaystyle\int_t^A\delta(\bar{Y})(a,t)da\leq\Vert Y_0\Vert_{\mathcal{H}_1}+\frac{L}{\lambda}\Vert\bar{Y}\Vert_{\mathcal{X}}
\end{align}

 On the other hand,  for $t>a,$
 \begin{align}
    \displaystyle\int_0^t\delta(\bar{Y})(a,t)da=\displaystyle\int_0^{t}\Pi(a)b(t-a)da+\displaystyle\int_0^t\displaystyle\int_{0}^{a}\dfrac{\Pi(a)}{\Pi(z)}\lfloor -f(\bar{Y}(z,z-a+t))\rfloor dzda,
\end{align}
and assumption $(H3)$ allows to obtain this estimate
\begin{align}\label{eq2.19}
    e^{-\lambda t}\displaystyle\int_0^t\delta(\bar{Y})(a,t)da\leq \frac{\Vert b\Vert_{L^{\infty}}}{\lambda}+\frac{L}{\lambda}\Vert\bar{Y}\Vert_{\mathcal{X}}.
\end{align}
The estimate \eqref{eq2.17}-\eqref{eq2.19} allows us to obtain
\begin{align}
\Vert\delta(\bar{Y})\Vert_{\mathcal{X}}\leq\dfrac{\Vert b\Vert_{L^{\infty}}+2\Vert\bar{Y}\Vert_{\mathcal{X}} +\lambda\Vert Y_0\Vert_{\mathcal{H}_1}}{\lambda}
\end{align}
Then, $\delta\in \mathcal{X}$ for $\lambda>0.$  For every $\bar{Y}_1,\; \bar{Y}_2\in \mathcal{X},$ and $t\leq a$
\begin{align}
e^{-\lambda t}\displaystyle\int_t^A\vert\delta(\bar{Y}_1)-\delta(\bar{Y}_1)\vert(a,t)da\leq\frac{L}{\lambda}\Vert\bar{Y}_1-\bar{Y}_2\Vert_\mathcal{X}
\end{align}
and for $t>a$ we similarly obtain as in \eqref{eq2.19} and 
\begin{align}
    \Vert\delta(\bar{Y}_1)-\delta(\bar{Y}_1)\Vert_{\mathcal{X}}\leq\frac{L}{\lambda}\Vert\bar{Y}_1-\bar{Y}_2\Vert_\mathcal{X}.
\end{align}
For $\lambda$ large enough, we clearly prove that $\delta$ is a contraction in $\mathcal{X}.$
\end{proof}
 We now consider the function $f$ and the non-local term $B$, nonlinear, and globally Lipschitzian.
\subsection{Case of a nonlinear nonlocal term}
\begin{theorem}\label{th2.7}
 Under assumptions $(H1)-(H2)-(H3)-(H5-ii)$ and for all $Y_0\in \mathcal{H}_2$ with $Y_0\geq 0,$ system \eqref{eq1.3} admits a unique nonnegative solution $Y.$ This solution belongs to $Y\in  C((0,T) ; \mathcal{H}_2)\cap L^2((0,A)\times(0,T))^N.$ 
\end{theorem}
Introduce the time‑dependent interaction operator

\begin{align}
M(t) =\operatorname{diag}\bigl(f_1(Y(\cdot,t)),\,\dots,\,f_N(Y(\cdot,t))\bigr),
\end{align}

 where each $f_i(Y(\cdot,t))$ is a given nonlinear functional of the age‑density $Y(\cdot,t)$.  If the system is of Kolmogorov type, then model \eqref{eq1.3} takes the form
\begin{equation}\label{eq2.24}
\left\lbrace 
\begin{array}{ll}
\partial_tY(a,t)+\partial_aY(a,t)+D(a)Y(a,t)+M(t).Y(a,t)=0 &\text{ in }Q,\\  
Y(0,t)=\displaystyle\int_0^A\bar{\beta}(a,P_1(t))Y(a,t)da, &\text{ in } Q_T,\\
Y(a,0)=Y_0 &\text{ in } Q_A.
\end{array}
\right.
\end{equation} 

\begin{proof}[of Theorem \ref{th2.7}]
Using the method of characteristics, the solution to system \eqref{eq2.24} can be expressed as follows:

\begin{equation}
Y=\left\lbrace\begin{array}{ll}
Y_{0}(a-t)e^{-\displaystyle\int_{a-t}^a(D(s)+M(t))ds}\qquad\text{for}\quad t\leq a,\\
b(t-a)e^{-\displaystyle\int_0^a(D(s)+M(t))ds}\qquad\text{for}\quad t> a,
\end{array}
\right.
\end{equation} 

where the function $b(t)$ satisfies the renewal condition
\begin{align}\label{eq2.26}
b(t)=\displaystyle\int_0^A\bar{\beta}(a,P_1(t))Y(a,t)da.
\end{align}
 Substituting the expression of $Y(a,t)$ into \eqref{eq2.26}, we obtain the following Volterra integral equation:

\begin{align}
b(t)=F(t)+\displaystyle\int_0^t k(t-s,t)b(s)ds  
\end{align}

where the terms are given by

\begin{align}
 F(t)=\displaystyle\int_t^A \bar{\beta}(a,P_1)Y_{0}(a-t)e^{-\displaystyle\int_{a-t}^a(D(s)+M(s))ds}da,\qquad k(a,t)=\bar{\beta}(a,P_1(t))e^{-\displaystyle\int_0^a(D(s)+M(s))ds}
\end{align}

Define
\begin{align}
    \bar{\mu}=\inf\lbrace D(s),\quad s\in (0,A)\rbrace.
\end{align}
 Then the following upper bounds hold:
 \begin{align}
    \vert F(t)\vert \leq \Vert\bar{\beta}\Vert_{L^{\infty}}e^{-t(\bar{\mu+M)}}\Vert Y_0\Vert_{\mathcal{H}_2},\quad\text{and}\;\vert k(a,t)\vert\leq \Vert\bar{\beta}\Vert e^{-a(\bar{\mu}+M)}.
\end{align}
 To establish the existence and uniqueness of a solution to the integral equation \eqref{eq2.26}, we apply Banach’s fixed point theorem. To this end, define a weighted norm on $L^\infty(0,T)^N$ by
 \begin{align}
\Vert b\Vert=\sup_{t\in(0,T)}\lbrace e^{-\lambda t}b(t)\rbrace
\end{align}
 for any $b \in L^\infty(0,T)^N$. Then consider the mapping $\mathcal{F} : L^\infty(0,T)^N \to L^\infty(0,T)^N$ defined by
 \begin{align}
\mathcal{F}(b)(t)=F(t)+\displaystyle\int_0^t k(t-s,t)b(s)ds
\end{align}
 For any $b_1, b_2 \in L^\infty(0,T)^N$, we estimate:
 \begin{align}
     |\mathcal{F}(b_1) - \mathcal{F}(b_2)|
 &\leq \sup_{t \in (0,T)} \left\lbrace e^{-\lambda t} \int_0^t k(t-s, t)|b_1(s) - b_2(s)|ds \right\rbrace 
 \end{align}
\begin{align*}
|\mathcal{F}(b_1) - \mathcal{F}(b_2)|
 &\leq |\bar{\beta}| |b_1 - b_2|  \sup_{t \in (0,T)} \left\lbrace e^{-\lambda t} \int_0^t e^{-(t-s)(\bar{\mu}+M)}e^{\lambda s}ds \right\rbrace \\
 \end{align*}
\begin{align}
   |\mathcal{F}(b_1) - \mathcal{F}(b_2)|
 & \leq  \frac{|\bar{\beta}|}{\bar{\mu}+M+\lambda} |b_1 - b_2|.
\end{align}
Choosing $\lambda$ sufficiently large such that  $\frac{\|\bar{\beta}\|}{\bar{\mu}+M+\lambda}<1$ ensures that $\mathcal{F}$ is a contraction on $L^\infty(0,T)^N$, and Banach's fixed point theorem yields the existence and uniqueness of the solution $b(t)$.
\end{proof}
\begin{remark}
We may allow a more general nonlocal term by assuming that the operator $B$ satisfies hypothesis (H5) without prescribing an explicit kernel. Moreover, if the nonlinearity $f$ has the required structure, then the solution can be expressed in the integral form as in \eqref{eq2.8}.
\end{remark}
\begin{remark} 
The operator $\mathcal{A}$ generates a $C_0$-semigroup, and the perturbation $M(Y(t))$ is bounded in $\mathcal{H}_2$. By the perturbation theorem, $\mathcal{A}-M(Y(t))$ is still the generator of a $C_0$-semigroup with $D(\mathcal{A}-M(Y(t)))=D(\mathcal{A})$. It follows that there exists a mild solution $Y(\cdot,t)=\mathcal{T}(t)Y_0\in\mathcal{H}_2$ for the first equation of \eqref{eq2.24}, and the system is well-posed as a nonlinear system describing the network dynamics. The transition matrix $M$ may depend on $t$ or $a$, may be non-diagonal, or may be replaced by a quadratic form $f(Y),$ in all cases the existence of a solution can be established analogously. Finally, existence and uniqueness for the full system \eqref{eq1.3} follow from showing that the mapping associated with \eqref{eq2.26} is a contraction in $C([0,T])^N$, even in the presence of the nonlinear nonlocal term.
\end{remark}
\begin{remark}\label{re2.8}
{\bf Non-Kolmogorov case.} Suppose the model is not of Kolmogorov type (as in G. F. Webb \cite{ref48}), i.e., the function $f$ does not take the form \eqref{eq1.2}, which corresponds to a particular case. Then, the existence of solutions can be established by reformulating the problem as a Volterra integral equation. In particular, the function $b(t) = Y(0,t)$ satisfies

  \begin{align}\label{eq2.41}
b(t)=F(t)+\displaystyle\int_0^tK(a,t)b(t-a)da\quad\text{a.e.}\quad t\in (0,T).
\end{align}
with
\begin{align}\label{eq2.42}
F(t)=\displaystyle\int_0^t\bar{\beta}(a,P_1)\int_{0}^a\dfrac{\Pi(a)}{\Pi(s)} [-f(Y(s,s+a-t))]dsda+\displaystyle\int_t^A\bar{\beta}(a,P_1)\dfrac{\Pi(a)}{\Pi(a-t)}Y_0(a-t)da
\end{align}
\begin{align*}
+\displaystyle\int_t^A\bar{\beta}(a,P_1)\int_{a-t}^a\dfrac{\Pi(a)}{\Pi(s)} [-f(Y(s,s-a+t))]dsda \quad\text{a.e.}\quad t\in(0,\min{\lbrace T,A\rbrace})
\end{align*}
\begin{align}
    F(t)=-\displaystyle\int_0^A\bar{\beta}(a,P_1)\int_0^a\dfrac{\Pi(a)}{\Pi(s)}f(Y(s,s+a-t))dsda\quad\text{a.e.}\quad \min{\lbrace T,A\rbrace}<t<T,
\end{align}
and  the maternity function
\begin{equation}
K(a,t)=\left\lbrace 
\begin{array}{ll}
\bar{\beta}(a,P_1(t))\Pi(a)\quad\text{a.e.}\quad (a,t)\in Q\quad a<t,\\
\\0 \qquad\text{elsewhere}. 
\end{array}
\right.
\end{equation} 
Equation \eqref{eq2.41}, known as the renewal equation and also as Volterra's equation, admits a solution. Let us therefore prove the uniqueness of the solution using Banach's fixed-point theorem. To this purpose, we define an operator on  $C(0,T)$ by
\begin{align}\label{eq2.45}
    \mathcal{F}(b)(t)=F(t)+\displaystyle\int_0^tK(t-s,t)b(s)ds\quad\text{a.e.}\quad t\in (0,T).
\end{align}
and we shown that the operator defined in \eqref{eq2.45} has a unique fixed point \cite{ref21} by applying the same technique in the Theorem \ref{th2.6}.  An addition, by constructing the renewal  equation as   \( b \) as a sequence, we obtain
\begin{equation}\label{eq2.46}
\left\lbrace 
\begin{array}{ll}
b_0(t)=F(t), & t\in (0,T),\\
\\b_{n+1}(t)= F(t)+\displaystyle\int_0^tK(t-s,t)b_n(s)ds,& t\in (0,T).
\end{array}
\right.
\end{equation}
If we take $f\leq 0,$ a.e. $(a,t)\in(0,A)\times(0,T)$  then, from \eqref{eq2.46}, we get $F\geq 0.$ And we deduice that $b_n(t)\geq 0,\quad t\in (0,T).$ 
\\However, if \( f \geq 0 \), it is necessary for \( Y_0 \) to take a sufficiently high value to compensate for the negative effects of the other terms. Thus, a sufficient condition for \( F(t) \geq 0 \; t\in (0,T)\;\text{a.e.}\) would be 
\begin{align}\label{eq2.49}
     Y_0(s) \geq \int_s^{s+t} \frac{\Pi(s)}{\Pi(z)} f(Y(z, z-s)) \, dz>0, \;\text{a.e.}\;\quad (s,t) \in (0, A)\times (0,T).
\end{align}

We may conclude that $b$, the solution of \eqref{eq2.46}, is nonnegative on $(0,T)$. However, the positivity of the solution may be compromised for $t \in (\min\{T,A\},T)$, particularly when $A < T$. To address this, we redefine $F(t)$ as in 
\begin{equation}
F(t)=\left\lbrace 
\begin{array}{ll}
\displaystyle\int_t^A\bar{\beta}(a,P_1)\dfrac{\Pi(a)}{\Pi(a-t)}Y_0(a-t)da+\displaystyle\int_t^A\bar{\beta}(a,P_1)\int_{a-t}^a\dfrac{\Pi(a)}{\Pi(s)} [-f(Y(s,s-a+t))]dsda \;\text{a.e.}\;(t,a)\in(0,\min{\lbrace T,A\rbrace})\times (t,A),\\
\\0 \; t\in\;(\min{\lbrace T,A\rbrace},T),
\end{array}
\right.
\end{equation} 
to ensure positivity. Consequently, for any function $f$ such that $Y_0$ satisfies condition \eqref{eq2.49}, the solution of system \eqref{eq1.3} remains positive.
\end{remark}
This Remark \ref{re2.8} examines the case where the solution of \eqref{eq1.3} is written in the form of \eqref{eq2.8}. This occurs when the interaction function groups together response functions of type II or III, or when it depends solely on the producers (prey). However, in the exponential case (the case of a Kolmogorov-type model) \eqref{eq2.24}, one may encounter, for example, Holling-Tanner functions or less realistic models, as mentioned in \cite{ref45}.  

The condition \eqref{eq2.49} states that the initial population density \( Y_0(s) \) must be greater than an integral involving the interaction function \( f \) and the survival probability \( \Pi \) of the species at different times. This condition is crucial to ensure the positivity of solutions in an interactive multi-species model.

\subsection*{\bf Discussion}
Condition $\eqref{eq2.49}$ is necessary and sufficient because it imposes a minimal initial density that ensures population persistence: it prevents extinction due to Allee effects and guarantees that ecological interactions (predation, resource availability, competition) remain sufficiently strong to sustain positive dynamics. Predator dietary flexibility and biodiversity enhance resilience by providing alternative food pathways and functional redundancy. Conversely, strict niche competition (Gause’s principle) may lead to exclusion. Importantly, ecosystem survival and stability (see \cite{ref55}) depend on a combination of environmental and anthropogenic factors (land-use changes, human pressures, population management), which must be considered when interpreting or applying condition $\eqref{eq2.49}$. Historically, Ronald Ross demonstrated the practical significance of such thresholds by showing that reducing mosquito density below a critical level can interrupt malaria transmission, an illustration of how threshold effects translate into actionable control strategies \cite{ref5}.

Mathematically, $\eqref{eq2.49}$ reads as a constraint on the initial datum $Y_0(s)$ that guarantees existence, positivity and persistence of solutions to integro-differential equations; empirical cases of species recovery illustrate that well-functioning and properly managed ecosystems enable population rebound.

\subsection{Steady states}\label{s2.4}
In this section, we determine the steady states of system. Any steady state of model \eqref{eq1.3} satisfies the
following system 
\begin{equation}\label{eq2.51}
\left\lbrace 
\begin{array}{ll}
\partial_aY(a)+D(a)Y(a)+f(Y(a))=0 &\text{ in }Q_A,\\  
Y(0)=BY.
\end{array}
\right.
\end{equation} 
The system \eqref{eq2.51} can be formulated either as a Kolmogorov-type system (when the functional responses used are of the Holling-Tanner type) or as another type of system (when using Holling Types II or III), and the analysis varies depending on this choice.  In the case of a Kolmogorov-type system, we obtain a solution of the form \( Y(a) = Y(0)\Pi_1(a) \) where $\Pi_1$ is defined in \eqref{eq2.58}. We simplify the calculation by assuming that the newborns satisfy $(H5-i)$.  

We analyze the local asymptotic and exponential stability of the equilibrium points using the linearization technique. Let us $Y$ and $\overline{Y}$ be, respectively, the solutions of the time‐dependent and the stationary systems.  Define the perturbation from the stationary state by $
\widetilde{y}(a,t) \;=\; Y(a,t)\;-\;\overline{Y}(a).$ The substitution of $\widetilde{y}$ into \eqref{eq1.3} yields, under {\bf assumption (H4)}, the following linearized problem:
\begin{equation}\label{eq2.52}
\left\lbrace 
\begin{array}{ll}
\partial_t\widetilde{y}+\partial_a\widetilde{y}+D(a)\,\widetilde{y}
\;+\;\Tilde{y}f'\bigl(\overline{Y}\bigr)
\;=\;0 
\quad&\text{in }Q,\\
\widetilde{y}(0,t) 
\;=\;\displaystyle\int_{0}^{A}\beta(a)\,\widetilde{y}(a,t)\,da\,\quad&\text{in }Q_T. 
\end{array}
\right.
\end{equation}  

To study the asymptotic behavior of \eqref{eq2.52}, we look for exponential in time solutions of the form $
\widetilde{y}(a,t) \;=\; e^{\lambda t}\,w(a),
$ where $w(a)$ satisfies
\begin{equation}\label{eq2.53}
\left\lbrace 
\begin{array}{ll}
\partial_a\,w \;+\;\bigl(\lambda I + D(a)\;+\;f'\bigl(\overline{Y}(a)\bigr)\bigr)\,w\;=\;0 
\quad&\text{in }\;(0,A),\\
w(0) 
\;=\;\displaystyle\int_{0}^{A}\beta(a)\,w(a)\,da\,.
\end{array}
\right.
\end{equation}

By solving \eqref{eq2.53}, we obtain for every $a\in[0,A]$ :

\begin{align}\label{eq2.54}
    w(a)
  \;=\;
  w(0)\,\exp\Bigl(-\int_{0}^{a}\bigl[\lambda I + D(s) + f'\bigl(\overline{Y}(s)\bigr)\bigr]\,ds\Bigr)
\end{align}

where the constant $w(0)$ satisfies the nonlocal boundary condition :

\begin{align}\label{eq2.55}
     w(0)
  \;=\;
  \int_{0}^{A}\beta(a)\,w(a)\,da
\end{align}

By substituting expression \eqref{eq2.54} into \eqref{eq2.55}, one obtains the following dispersion equation:

\begin{align}\label{eq2.56}
   \Lambda(\lambda)= \int_{0}^{A}
    \beta(a)\,
    \exp\Bigl(-\int_{0}^{a}\bigl[\lambda I + D(s) + f'\bigl(\overline{Y}(s)\bigr)\bigr]\,ds\Bigr)
  \;da
  \;=\;I_{\mathcal{H}^2}.
\end{align}
We call $\Lambda(\lambda)=I$ the characteristic equation, and its roots  are called characteristic roots. 
By convention, we then define the basic reproduction number $R_{0}$ as the value of the right-hand side of \eqref{eq2.56} when $\lambda = 0$. In other words,

\begin{align}\label{eq2.57}
    R_{0}
    \;=\;
    \int_{0}^{A}
      \beta(a)\,
      \exp\Bigl(-\int_{0}^{a}\bigl[D(s) + f'\bigl(\overline{Y}(s)\bigr)\bigr]\,ds\Bigr)
    \;da
\end{align}

The factor

\begin{align}\label{eq2.58}
   \Pi_1(a)= \exp\Bigl(-\int_{0}^{a}\bigl[D(s) + f'(\overline{Y}(s))\bigr]\,ds\Bigr)
\end{align}

represents the (linearized) probability of survival of an individual from birth up to age $a$.    The characteristic equation \eqref{eq2.56} allows one to impose conditions on the basic reproduction number $R_0$ that ensure stability, which differ from those derived in \cite{ref45}.

\begin{theorem}\label{th2.9}
Let us consider assumption (H4),  $\bar Y$  be the stationary solution of \eqref{eq1.3}. If $R_0<I_{\mathcal{H}^2},$ then the steady state of system\eqref{eq1.3} is locally asymptotically stable. Otherwise, if $R_0>I_{\mathcal{H}^2}$ then the steady state is unstable.
\end{theorem}
\begin{proof}[of Theorem \ref{th2.9}]
The proof is based on classical methods : 
\begin{itemize}
\item  If $R_{0} < I_{\mathcal{H}^2}$, then $\Lambda(0) = R_{0} < I_{\mathcal{H}^2}$. Since $\Lambda(\lambda)$ is strictly decreasing, the only solution of $\Lambda(\lambda)=I_{\mathcal{H}^2}$ is $\lambda^{*}<0$. Consequently, $\overline{Y}$ is locally asymptotically stable.

\item  If $R_{0} > I_{\mathcal{H}^2}$, then $\Lambda(0) = R_{0} > I_{\mathcal{H}^2}$, so the equation $\Lambda(\lambda)=I_{\mathcal{H}^2}$ admits a root $\lambda^{*}>0$. Hence $\overline{Y}$ is unstable.

\item  If $R_{0} = I_{\mathcal{H}^2}$, we have $\Lambda(0) = I_{\mathcal{H}^2}$, hence $\lambda^{*}=0$. In this critical case, the linear perturbation remains constant in time, corresponding to the bifurcation threshold. One must then examine higher-order (nonlinear) terms to determine whether the equilibrium becomes stable or unstable.
\end{itemize}
\end{proof}

The analysis of (exponential) stability can be carried out using semigroups, as shown in  \cite{ref63}. The following theorem ensures the local exponential stability of system \eqref{eq1.3} around $\overline{Y}$.

\begin{theorem} \label{th2.10}
Let us consider assumption (H4), $R_0<I_{\mathcal{H}^2}.$ Then, the system \eqref{eq1.3} is locally exponentially stable.
\end{theorem}

\begin{proof}[of Theorem \ref{th2.10}]
Since $\mathcal{A}$ generates an exponentially stable $C_{0}$‐semigroup on $\mathcal{H}_{2}$ and $f'(\overline{Y})$ is bounded on $\mathcal{H}_{2}$, Phillips’ theorem \cite{ref37} implies that the operator
 $$
 M \;=\; \mathcal{A} \;-\; f'\bigl(\overline{Y}\bigr),
 \quad D(M)=D(\mathcal{A}),
 $$
 also generates a (exponentially stable) $C_{0}$‐semigroup on $\mathcal{H}_{2}$.
\end{proof}
\subsection*{\bf Discussion}
The qualitative characterization of dynamical systems based on the basic reproduction number $R_0$ becomes substantially more complex in coupled multi-species models. Two cases arise.
\begin{itemize}
\item  Multi-phase transition models. When a population progresses through successive phases (e.g. aquatic stage then adult stage with male/female subpopulations), it is often possible, despite the coupling between distinct phase equations, to reduce the analysis to a single global $R_0$ for the entire model.

\item  Interacting multi-species models. If several species interact via nonlocal terms or coupled source terms, each species has its own $R_{0_i}$. One then sets the diagonal matrix

\begin{align}
R_{0}=\operatorname{diag}(R_{0_1},\dots,R_{0_N}),
\end{align}

and the matrix condition

\begin{align}\label{e2.54}
R_{0}<I_{\mathcal{H}^2} \ \Longleftrightarrow\ R_{0_i}<1\ \forall i,\qquad
R_{0}>I_{\mathcal{H}^2} \ \Longleftrightarrow\ R_{0_i}>1\ \forall i,
\end{align}

which provides a sufficient (but not necessary) criterion for joint extinction or joint persistence of all species.
\end{itemize}
{\bf Ecological interpretation and limits of the $R_0$ criterion}. Comparing each $R_{0_i}$ to 1 does not capture all possible dynamics: intraspecific regulation, asymmetric feedback loops (competition, mutualism, cross-predation) and network topology can compensate for heterogeneity in the $R_{0_i}$. Thus, systems with some $R_{0_i}<1$ may nonetheless allow partial persistence or stable coexistence thanks to compensatory interactions.

{\bf Role of interaction structure \cite{ref75}} : Theory and empirical evidence show that stability depends more on statistical and structural properties of the interaction matrix—its moments (mean, variance) and correlations of coefficients—than on the mere count of links. The dispersion of interaction strengths often matters more than raw connectance. Moreover, the presence of weak interactions and functional redundancy enhances resilience: many alternative pathways and species able to fulfil similar roles mitigate the impact of perturbations.

{\bf Historical and empirical nuances.} Although some authors have observed that greater trophic complexity tends to damp large oscillations seen in very simple systems, this is not a universal rule: simple systems can be stable and rich networks can collapse if crucial links (keystone species) are disrupted. Trophic cascades show that removing a keystone species can trigger the rapid collapse of an otherwise diverse network. MacArthur even proposed that community stability increases roughly with the logarithm of the number of trophic links, a heuristic observation based on an analogy to information theory rather than on a formal mathematical proof \cite{ref76}. 

{\bf For example}, in a two-species predator–prey system one may have $R_{0,1}<1$ and $R_{0,2}<1$ without guaranteeing predator extinction or unrestrained prey outbreaks: intraspecific regulatory mechanisms, community functional redundancy, or predator generalism (the ability to switch to alternative prey) can stabilize the dynamics. In networks of five or more species, feedback loops (asymmetric competition, mutualism, cross-predation), reinforced by redundancy and trophic plasticity, often play a compensatory role: they preserve global stability despite heterogeneity in the $R_{0,i}$, allowing partial persistence, stable coexistence, or selective extinction. This illustrates that it is the topology and nature of interactions, rather than strict compliance with $\eqref{e2.54}$, that underlie ecosystem stability.

{\bf In conclusion,} ecosystem stability depends less on species richness per se and more on the arrangement, intensity and variability of interspecific interactions, together with environmental and anthropogenic drivers. Therefore, accounting for multi-species interactions and their topology is essential for realistic modelling, prediction and the preservation of ecological resilience.

\begin{remark}
Theorem \ref{th2.9} provides a local stability result based on the basic reproduction number $R_0$. To extend beyond this local analysis and establish a more global type of stability, one must resort to more powerful tools, most notably the direct Lyapunov method. Rather than solving the differential equations explicitly, this approach constructs a scalar “energy” (Lyapunov) function whose time‑derivative along system trajectories is nonpositive. If that function is strictly positive everywhere except at the equilibrium and its derivative along the flow is strictly negative off the equilibrium, then the equilibrium is asymptotically stable. In other words, every trajectory starting sufficiently close will converge to the equilibrium.
\end{remark}
\section{\bf Applications : Examples of models and stability analysis}\label{s3}
Biological and biotechnological systems often exhibit competitive dynamics, cyclic or non-cyclic, where the age of individuals plays a crucial role in interspecific interactions, frequently modeled using the law of mass action. These phenomena, observable across scales from microbial communities to complex ecosystems, require mathematical modeling that captures both the demographic structure of populations and the non-local interactions between species. The mathematical study of such age-structured systems is therefore a fundamental challenge for biotechnological optimization, ecosystem management, and therapeutic interventions. Studying these models enables analysis of stability, optimization of control strategies, and prediction of ecological transitions.

We consider $N$ interacting age-structured populations distributed over the age interval $(0,A)$ and evolving over the time horizon $(0,T)$.  Let the state vector be:
\begin{align}
X(t) = \big( x_1(\cdot,t), x_2(\cdot,t), \dots, x_N(\cdot,t) \big)^\top,
\end{align}
where each $x_i(a,t)$ represents the density of individuals of species (or group) $i$ of age $a$ at time $t$.  

For $1 \leq i,j \leq N$ we define the nonlocal interaction coefficients:
\begin{align}
\gamma_{ij}(t) = \int_0^A g_{i,j}(a) \, x_j(a,t) \, da,
\end{align}
where $g_{i,j}(a) \geq 0$ represents the interaction kernel describing how individuals of group $j$ affect the mortality of group $i$ across all ages and $g_{ij}\in L^2(0,A)$.\paragraph{Definition of the interaction operator.}
Let 
\begin{align}\label{equation3.3}
A = (a_{ij})_{1 \le i,j \le N}, \quad a_{ij} \in \{0,1\},
\end{align}
be the \emph{adjacency matrix} describing the static structure of the interaction network 
(\(a_{ij} = 1\) if node \(i\) interacts with node \(j\), and \(a_{ij} = 0\) otherwise).

Let 
\begin{align}
\Gamma(t) = \big( \gamma_{ij}(t) \big)_{1 \le i,j \le N}
\end{align}
be the matrix of \emph{time-dependent nonlocal interaction intensities}.

The \emph{Hadamard product} of two matrices \(M=(m_{ij})\) and \(N=(n_{ij})\) is denoted by
\begin{align}
 M \circ N := \big( m_{ij} \, n_{ij} \big)_{1 \le i,j \le N}.
\end{align}

For a matrix \(M\) and a vector \(X = (x_1,\dots,x_N)^\top\), we define the \emph{row-wise scalar product} 
\begin{align}
[M \bullet_r X]_i := \sum_{j=1}^N M_{ij} \, x_i, \quad i=1,\dots,N.
\end{align}

We then define the interaction term
\begin{align}\label{a3.7}
\mathcal{B}(X(t)) := \big( A \circ \Gamma(t) \big) \bullet_r X(t),
\end{align}
Let:
\begin{align}
    M(a) = \big(\mu_1(a), \dots, \mu_N(a)\big)^\top,
\end{align}
where $\mu_i(a) \geq 0$ is the natural mortality rate of group $i$ at age $a$. The birth process is described by the \emph{birth kernel matrix}:
\begin{align}
    K(a) = \big(k_{ij}(a)\big)_{1 \leq i,j \leq N},
\end{align}
where $k_{ij}(a)$ represents the fertility rate at age $a$ of group $j$ producing newborns in group $i$.

We introduce an external control input vector:
\begin{align}
    U(t) = \big( u_1(t), u_2(t), \dots, u_N(t) \big)^\top,
\end{align}
where $u_i(t)$ represents the control effort applied to group $i$ at time $t$ (e.g., harvesting rate, medical intervention, removal effort). Let:
\begin{align}\label{a3.12}
    B = (a_1, a_2, \dots, a_N)^T,
\end{align}
where $a_i \in \lbrace 0,1\rbrace$ is the efficiency coefficient of the control on population $i$. The control term is then given by:
\begin{align}
    B\circ U(t) = \big( a_1 u_1(t), a_2 u_2(t), \dots, a_N u_N(t) \big)^\top,
\end{align}
and acts pointwise in age through:
\begin{align}
    B\circ U(t) \, \bullet_r  \, X(a,t).
\end{align}

Thus, the full age-structured model is given by :

\begin{align}\label{a3.9}
    \begin{cases}
\partial_t X(a,t) + \partial_a X(a,t) 
= -\mathcal{B}(X(t))-\big( M(a) + B\circ U(t)  \big)\bullet_r  \; X(a,t),
& \text{in}\; Q, \\[6pt]
X(0,t) = \displaystyle \int_0^A K(a) \, X(a,t) \, da,
& \text{in}\; Q_T, \\[6pt]
X(a,0) = X_0(a),
& \text{in}\; Q_A.
\end{cases}
\end{align}

Within this general model \eqref{a3.9} lies a particular case: non-transitive competition, which is an ecological interaction in an ecosystem where the relationships are cyclic. In this context, the matrix $ \mathcal{B}(X(t))$ in \eqref{a3.7} may take the following form
\begin{align}
A &= 
\begin{bmatrix}
0 & 1 & 0 & \cdots & 0 \\
\vdots & 0 & \ddots & \cdots & 0 \\
\vdots & \vdots & \ddots & \vdots & 1 \\
1 & 0 & \cdots & \cdots & 0
\end{bmatrix},
&
\mathcal{B}(X(t)) &=(A \circ \Gamma(t))\bullet_r X(t)
=
\begin{bmatrix}
0 & \gamma_{1,2}(t) & 0 & \cdots & 0 \\
\vdots & 0 & \ddots & \cdots & 0 \\
\vdots & \vdots & \ddots & \vdots & \gamma_{N-1,N}(t) \\
\gamma_{N,1}(t) & 0 & \cdots & \cdots & 0
\end{bmatrix}\bullet_r X(t),
\end{align}
and the system became 
\begin{align}\label{ee3.11}
\begin{cases}
\partial_t X(a,t) + \partial_a X(a,t) 
= -\mathcal{B}(X(t))-\big( M(a) + B\circ U(t)  \big)\;\bullet_r   \; X(a,t),
& \text{in}\; Q, \\[6pt]
X(0,t) = \displaystyle \int_0^A K(a) \, X(a,t) \, da,
& \text{in}\; Q_T, \\[6pt]
X(a,0) = X_0(a),
& \text{in}\; Q_A.
\end{cases}
\end{align}
We adopt the following standing hypotheses (unless otherwise stated):
\begin{align*}
\textbf{(H11)}\begin{cases}
\mu_i(a)\geq 0 \quad \text{a.e. on }(0,A), \\[0.3em]
\mu_i \in L^1_{\rm loc}(0,A), \quad \displaystyle\int_0^{A}\mu_i(a)\,da = +\infty,
\end{cases}
\qquad
\textbf{(H22)}\begin{cases}
k_i \in L^\infty(0,A),  \\[0.3em]
k_i(a)\ge 0 \ \text{a.e. on } (0,A).
\end{cases}
\end{align*}
\subsection*{\bf Well-posedness}
\begin{proposition}\label{p3.1}
Under {\bf assumptions (H11)-(H22)} and for all $X_0\in \mathcal{H}_2$ with $X_0\geq 0,$  the system  system \eqref{ee3.11} admits a unique nonnegative solution. 
\end{proposition}
\begin{proof}[of Proposition \ref{p3.1}]
Thanks to Lemma \ref{le2.1}, the operator
\begin{align}
 \mathcal A_l : D(\mathcal A_l)\subset\mathcal H_2\longrightarrow\mathcal H_2,
   \quad \mathcal A_l\varphi=-\partial_a\varphi - M(a)\varphi,
\end{align}
with
\begin{align}
D(\mathcal A_l)
= \Bigl\{\varphi\in\mathcal H_2 :
\varphi\text{ a.c. on }[0,A),\;
\varphi(0)=\!\int_0^AK(a)\varphi(a)\,da,\;
-\partial_a\varphi - M(a)\varphi\in\mathcal H_2
   \Bigr\}.
\end{align}
is the infinitesimal generator of a strongly continuous semi-group on $\mathcal{H}_2.$ Since the control $U$ is bounded, the result follows from Theorem \ref{th2.7}.
\end{proof}
\begin{remark}
In our framework, the matrix A represents the interactions between species and should be interpreted as a connectivity matrix associated with a directed graph.
A control applied to a given species influences the whole system if, from the corresponding node, there exists a directed path to every other node.
This property, known as the strong connectivity of the interaction graph, ensures that the effect of the control propagates throughout the system.
Thus, rather than directly invoking the Kalman condition (which is primarily suited to linear systems), we adopt an interpretation in terms of graph connectivity, which is more appropriate for the nonlinear and nonlocal structure of our model.

\end{remark}

After presenting the general framework, we focus on a specific example of a non-transitive competition model involving {\bf three-species, four-species, up to a generalization}.\\
\subsubsection*{\bf A sketch of the proof}
In the study of stability for dynamic competition models, particularly in non-transitive settings, we establish a general stability result for non-transitive competition models by induction, using a reduced control localized on a single species. The approach treats first the three-species case, then the four-species case, to derive the inductive step that extends to the general $N$ species case.
The single control, applied to one species, acts indirectly on the other system components.
To synthesize the global feedback, we construct successive fictitious controls that partially stabilize each species; these controls are implemented step by step and recursively until the global control is obtained.
\subsection{\bf Three-species non-transitive competition}\label{se3.1}
 An age-structured system with three interacting populations can model cyclic competition observed in microbial communities, where each species produces a toxin that inhibits another. For example, Kerr et al. (2002) experimentally demonstrated a rock–paper–scissors dynamic using genetically modified Escherichia coli strains: species A kills B, B kills C, and C kills A. In the model, the nonlocal terms $\gamma_{ij}(t)$ describe the age-distributed inhibitory effects of one species on another. A control applied to species 1 may represent a targeted antibiotic treatment. This framework captures both experimentally validated and theoretically analyzed dynamics of cyclic dominance (Durrett and Levin, 1998; Frean et al., 2008; Nowak and Sigmund, 2004). A representative example of a non-transitive competition model is discussed in \cite{ref3}.

In this section, we focus on the non-transitive competition model \eqref{e3.1}. Non-transitive competition is an ecological relationship in which species interactions do not follow a linear hierarchy but rather form a cycle. Concretely, species $x_{1}$ dominates $x_{3}$, $x_{3}$ dominates $x_{2}$, and $x_{2}$ in turn dominates $x_{1}$. This organization tends to promote stable diversity in the ecosystem by preventing any single species from achieving exclusive dominance. Moreover, such cycles foster the emergence of ecological niches and complex dynamics (oscillations, cyclic coexistence, etc.), phenomena commonly observed in natural communities.
\begin{align}\label{e3.1}
 \begin{cases}
\displaystyle 
\partial_t x_1(a,t) + \partial_a x_1(a,t)
= -\Bigl(\mu_1(a) 
       + \int_0^A g_1(a)\,x_2(a,t)\,\mathrm{d}a
       + u(t)\Bigr)\;x_1(a,t),
& \text{ in } Q_1=(0,A)\times\mathbb{R}_+,\\[8pt]
\displaystyle 
\partial_t x_2(a,t) + \partial_a x_2(a,t)
= -\Bigl(\mu_2(a) 
       + \int_0^A g_2(a)\,x_3(a,t)da\Bigr)\;x_2(a,t),
& \text{ in } Q_1=(0,A)\times\mathbb{R}_+,\\[8pt]
\partial_t x_3(a,t) + \partial_a x_3(a,t)
= -\Bigl(\mu_3(a) 
       + \displaystyle\int_0^A g_3(a)\,x_1(a,t)\,\mathrm{d}a
       \Bigr)\;x_3(a,t),
& \text{ in } Q_1=(0,A)\times\mathbb{R}_+,\\[8pt]
\displaystyle 
x_i(0,t)
= \int_0^Ak_i(a)x_i(a,t)da,\quad i=1,2,3 & \text{ in } Q_+=\mathbb{R}_+,\\[8pt]
x_i(a,0) = x_{i,0}(a),\quad i=1,2,3
& \text{ in } Q_A=(0,A).
\end{cases}
\end{align}

Here:
\begin{itemize}
\item  $g_i(a)\in L^2(0,A)$ describes the interaction kernel (how individuals of age $a$ of one species affect the other).
\item  $u(t)$ (a bounded control law) is the common (harvesting or management) control applied to both populations.
\item  $A$ is the maximal age, and $T$ the final time horizon.
\end{itemize}
\begin{remark}

In the non-transitive competition model \eqref{e3.1}, the connectivity of the interaction graph ensures that an input applied to species $x_1$ influences the remaining species through the network of interactions, which makes state-feedback stabilization feasible. Nevertheless, proving global stability necessitates a more refined nonlinear analysis.
\end{remark}
\subsubsection{\bf Stability analysis} \label{s3.1.2.2}
The steady-state form of equation \eqref{e3.1} is as follows
\begin{align}\label{e3.2}
\begin{cases}
\partial_ax^*_1(a)=-\left(\mu_1(a)+\zeta_1\right)x^*_1(a), & \text{ in } Q_A,\\[8pt]
\\\partial_ax_2(a)=-\left(\mu_2(a)+\zeta_2\right)x^*_2(a), & \text{ in } Q_A,\\[8pt]
\\\partial_ax_3(a)=-\left(\mu_3(a)+\zeta_3\right)x^*_2(a),& \text{ in } Q_A,\\[8pt]
x^*_i(0)=\displaystyle\int_0^Ak_i(a)x_i^*(a)da,\quad i=1,2,3
\end{cases}\qquad\text{with}\; \begin{cases}
\zeta_1=\lambda_2+u^*,\\ \zeta_2=\lambda_3,\\
\zeta_3=\lambda_1,\\
\lambda_i=\displaystyle\int_0^Ag_j(a)x^*_i(a)da
\end{cases}
\end{align}

and  the solution takes the form
\begin{align}\label{e3.3}
x_i^*(a)=x^*_i(0)\underbrace{e^{-\int_0^a(\mu_i(s)+\zeta_i)ds}.}_{\Tilde{x}^*_i(a)}
\end{align}

$\zeta_i$ is the unique solution to the characteristic equation 
\begin{align}\label{e3.4}
\begin{cases}
\displaystyle\int_0^A\underbrace{k_1(a)e^{-\int_0^a(\mu_1(s)+\zeta_1)ds}}_{\tilde{k}_1(a)}da=1,\\
      \displaystyle\int_0^A\underbrace{k_2(a)e^{-\int_0^a(\mu_2(s)+\zeta_2)ds}}_{\tilde{k}_2(a)}da=1,\\
       \displaystyle\int_0^A\underbrace{k_3(a)e^{-\int_0^a(\mu_3(s)+\zeta_3)ds}}_{\tilde{k}_3(a)}da=1.
\end{cases}
\end{align}
It follows  that
\begin{align}\label{e3.5}
    u^*=\zeta_1-\lambda_2\in (0;\zeta_1).
\end{align}
The newborn population then takes the form
\begin{align}\label{e3.6}
 \begin{cases}
x_1^*(0)=\dfrac{\zeta_3}{\displaystyle\int_0^Ag_3(a)\Tilde{x}^*_1(a)da}>0,\\
x_2^*(0)=\dfrac{\zeta_1-u^*}{\displaystyle\int_0^Ag_1(a)\Tilde{x}^*_2(a)da}>0,\\
x_3^*(0)=\dfrac{\zeta_2}{\displaystyle\int_0^Ag_2(a)\Tilde{x}^*_3(a)da}>0.
 \end{cases}
\end{align}

\begin{lemma}\label{lee3.1}
Consider the following transformation
\begin{align}\label{e3.7}
 \left[\begin{array}{c}
\eta_i(t) \\ 
\psi_i(t-a)
\end{array}  \right]=\left[\begin{array}{c}
\ln[\Pi_i(x_i(t))] \\ 
\dfrac{x_i(a,t)}{x^*(a)\Pi_i(x_i(t))}-1
\end{array} \right],
\end{align}
where
\begin{align}\label{e3.8}
\Pi_i(x_i(t))=\dfrac{\langle \pi_{0,i}, x_i(t)\rangle_{L^2(0,A)}}{\langle\pi_{0,i}, x_i^*\rangle_{L^2(0,A)}},
\end{align}

with  $\pi_{0,i}$ is continuous functions of the form
\begin{align}\label{e3.9}
   \pi_{0,i}(a)= \displaystyle\int_a^{A}k_i(a)e^{\int_s^a(\zeta_i+\mu_i(l)dl}ds.
\end{align}
Moreover, the variables $\psi_i$ and $\eta_i$ satisfy:

\begin{align}\label{e3.10}
    \begin{cases}
\dot{\eta}_1(t)
= \zeta_1 - u(t)
  - e^{\eta_2} \displaystyle\int_0^A g_1(a)\,x^*_2(a)\,\bigl(1 + \psi_2(t-a)\bigr)\,\mathrm{d}a,\\
\dot{\eta}_2(t)
= \zeta_2 
  - e^{\eta_3} \displaystyle\int_0^A g_2(a)\,x^*_3(a)\,\bigl(1 + \psi_3(t-a)\bigr)\,\mathrm{d}a,\\
  \dot{\eta}_3(t)
= \zeta_3
  - e^{\eta_1} \displaystyle\int_0^A g_3(a)\,x^*_1(a)\,\bigl(1 + \psi_1(t-a)\bigr)\,\mathrm{d}a,\\
\psi_i(t)= \displaystyle\int_0^A \tilde k_i(a)\,\psi_i(t-a)\,\mathrm{d}a,\\
\eta_i(0)
= \ln\bigl(\Pi[x_{i,0}]\bigr) = \eta_{i,0},\qquad
\psi_i(-a)
= \frac{x_{i,0}(a)}{x_i^*(a)\,\Pi[x_{i,0}]} - 1 = \psi_{i,0}(a).
\end{cases}
\end{align}

The unique solutions are then given by :

\begin{align}\label{e3.11}
x_i(a,t)
={x}_i^{*}(a)\,\bigl(1+\psi_{i}(t-a)\bigr)\,e^{\eta_{i}}.
\end{align}
\end{lemma}
\begin{proof}[of Lemma \ref{lee3.1}]  
Integrating by parts over $(0,A)$ yields the following expressions :

\begin{align}
\left\langle \pi_{0,1}(a), \partial_ tx_1(a,t)\right\rangle=\left\langle\partial_a\pi_{0,1}(a)+\pi_{0,1}(a)k_1(a)-\pi_{0,1}(a)(\mu_1(a)+\zeta_1), x_1(a,t)\right\rangle+\langle\pi_{0,1}(a),(\zeta_1-\int_{0}^A g_1(a)x_2(a,t)da-u(t))x_1(a,t)\rangle,
\end{align}

\begin{align}
     \left\langle \pi_{0,2}(a), \partial_ tx_2(a,t)\right\rangle=\left\langle\partial_a\pi_{0,2}(a)+\pi_{0,2}(a)k_2(a)-\pi_{0,2}(a)(\mu_2(a)+\zeta_2), x_2(a,t)\right\rangle+\langle\pi_{0,2}(a),(\zeta_2-\int_{0}^A g_2(a)x_3(a,t)da)x_2(a,t)\rangle,
\end{align}

\begin{align}
     \left\langle \pi_{0,3}(a), \partial_ tx_3(a,t)\right\rangle=\left\langle\partial_a\pi_{0,3}(a)+\pi_{0,3}(a)k_3(a)-\pi_{0,3}(a)(\mu_3(a)+\zeta_3), x_3(a,t)\right\rangle+\langle\pi_{0,3}(a),(\zeta_3-\int_{0}^A g_3(a)x_1(a,t)da)x_3(a,t)\rangle.
\end{align}

We obtain system \eqref{e3.10} together with 

\begin{align}\label{}
\begin{cases}
\mathcal{D}^*\pi_{0,1}(a)=\partial_a\pi_{0,1}(a)-\pi_{0,1}(a)(\mu_1(a)+\zeta_1)+\pi_{0,1}(0)k_1(a),\quad \pi_{0,1}(A)=0,\\
\mathcal{D}^*\pi_{0,2}(a)=\partial_a\pi_{0,2}(a)-\pi_{0,2}(a)(\mu_2(a)+\zeta_2)+\pi_{0,2}(0) k_2(a),\qquad    \pi_{0,2}(A)=0,\\
\mathcal{D}^*\pi_{0,3}(a)=\partial_a\pi_{0,3}(a)-\pi_{0,3}(a)(\mu_3(a)+\zeta_3)+\pi_{0,3}(0) k_3(a),\qquad \pi_{0,3}(A)=0,
\end{cases}
\end{align}
by following the strategy employed in \cite{ref69}. Applying transformation \eqref{e3.7} yields equation \eqref{e3.11}.
\end{proof}
\paragraph{\bf 3.1.1.1. Stability for $\psi_i\equiv 0$}
Then, from \eqref{e3.10}, where we set  
\begin{align}
\phi_i(\eta_i)=\lambda_i\bigl(e^{\eta_i}-1\bigr),\qquad i=1,2,3
\end{align}
we obtain the following system
 \begin{align}\label{ee3.27}
   \begin{cases}
       \dot\eta_1 
   &= u^* - u - \lambda_2\bigl(e^{\eta_2}-1\bigr) \\
   \dot\eta_2 
   &= -\lambda_3\bigl(e^{\eta_3}-1\bigr)\\
   \dot\eta_3 
   &= -\lambda_1\bigl(e^{\eta_1}-1\bigr)
   \end{cases}\;\Longleftrightarrow\begin{cases}
       \dot\eta_1 = u^* - u - \phi_2(\eta_2),\\
\dot\eta_2 =  - \phi_3(\eta_3),\\
\dot\eta_3 =  - \phi_1(\eta_1)
   \end{cases}
 \end{align}

By employing the Lyapunov function

\begin{align}\label{a3.27}
    \omega(\eta_i)
 = \lambda_i\bigl(e^{\eta_i} - 1 - \eta_i\bigr),
\end{align}

 we derive the following time derivative, 
\begin{align}\label{a3.28}
    \dot \omega(\eta_i)
 = \phi_i(\eta_i)\dot \eta_i.
\end{align}



With the static control $ u(t) = u^*,$ the Jacobian at the point $(\eta_1, \eta_2, \eta_3) = (0, 0, 0)$ is

 $$
 J(0) =
 \begin{pmatrix}
 0 & -\lambda_2 & 0 \\[6pt]
 0 & 0 & -\lambda_3 \\[6pt]
 -\lambda_1 & 0 & 0
 \end{pmatrix},
 \qquad \lambda_i > 0.
 $$
Its eigenvalues are
$$\mu_1 = -(\lambda_1\lambda_2\lambda_3)^{1/3}e^{j\frac{2\pi k}{3}},\quad k\in \lbrace0,1,2\rbrace.$$

With these eigenvalues, the equilibrium at $(0,0,0)$ is unstable: it is not asymptotically stable. Only initial trajectories belonging to the one-dimensional stable subspace will converge to the equilibrium; most perturbations will diverge exponentially along the two unstable directions. Therefore, the static control $u^*$ is not sufficient to achieve asymptotic stability. It is necessary to modify the control law (e.g., state feedback or backstepping) if one wishes to obtain asymptotic stabilization (local, and a fortiori global).
\begin{remark}
To prevent overharvesting of species $x_1$, we design a control law that can assume both negative and positive values. Thus, if species $x_1$ becomes depleted, a positive dilution would drive all populations to extinction. 
\end{remark}

In system \eqref{ee3.27}, $\eta_{2}$ is solely a function of $\eta_{3}$, $\eta_{3}$ depends on $\eta_{1}$, and the latter is the only variable directly controlled by $u$. We propose to apply the backstepping method to achieve stabilization of a system composed of three nested subsystems : 


\begin{align}\label{a3.32}
\begin{cases}
\eta_2,\\
 z_{1} = \eta_{3}-\alpha_1(\eta_2),\\  
 z_{2} = \eta_{1} - \alpha_2(z_1).
\end{cases}
\end{align}
\begin{remark}
In system \eqref{a3.32}, $\eta_3$ is employed as a fictitious control to stabilize $\eta_2$, and the error variable $z_1$ is introduced. Subsequently, $\eta_1$ is regarded as another fictitious control to stabilize $z_1$, which naturally leads to the definition of a second error variable, $z_2$. The purpose of the constructed fictitious controls is to partially stabilize the equations in a recursive manner, until the general control $u$ is obtained.
\end{remark}

\subsubsection*{\bf Notation} For clarity, we will denote by \(z_i^k\) the \(i\)-th error variable of the system with \(k\) species. Let us introduce the following notation:
\[
\phi_i(z^k_j)=\lambda_i\bigl(e^{z^k_j}-1\bigr),\qquad 
\alpha_i(\eta_j)=\eta_j,\qquad 
\alpha_i(z^k_j)=z^k_j,
\]
i.e. without loss of generality each $\alpha_i$ is the identity function. The design coefficients are
\[
c_i,\qquad c_{io}=c_i+1,\qquad \theta,
\]
where $i,j=1,\dots,N$  (replace indices as appropriate). Consider the following Lyapunov function : 
\begin{align}\label{equation3.44}
V_3(\eta_2, z^3_1, z^3_2)= \theta \lambda_{2}\bigl(e^{\eta_{2}}-1-\eta_{2}\bigr) + \theta\lambda_3\big(e^{z^3_1} - 1-z^3_1\big)+\lambda_1(e^{z^3_2}-1-z^3_2).
\end{align}
This function satisfies the Lyapunov conditions:
\begin{itemize}
 \item  $V_{3}(0,0,0)=0$,
 \item   for all $\alpha\neq0,\;V_{3}(\alpha)>0$ and $\lim_{\alpha\to\infty}V_{3}(\alpha,\alpha,\alpha)=+\infty$.
\end{itemize}
\begin{theorem}\label{the3.3}
Under the proposed feedback law the system \eqref{ee3.27} is globally asymptotically stable. The feedback control constructed is uniformly bounded, though not necessarily nonnegative. Moreover, the control satisfies $u(t)>0$ for every $t>0,$ for every $\eta_i(0)$ belonging to the largest
level set of $V_3(\eta_2, z^3_1, z^3_2)$ within the set 
\begin{align}\label{equation3.45}
\mathcal{K} = 
\left\{ \eta\in\mathbb{R}^3\ \middle|\ 
\begin{aligned}
&u^*+c_{3o}\phi_1(z^3_2)-\theta\frac{1}{\lambda_3}\phi^2_3(z^3_1)+(\dfrac{\lambda_1}{\lambda_3}-1-\theta)\phi_3(z^3_1)+\frac{1}{\lambda_3}\phi_3(z^3_1)\phi_1(z^3_2)-\frac{1}{\lambda_2}\phi_3(z^3_1)\phi_2(\eta_2)\\
&-(\dfrac{\lambda_3}{\lambda_2}+1)\phi_2(\eta_2)+\frac{\theta\phi_3(z^3_1)}{\phi_1(z^3_2)}\left(\frac{1}{\lambda_2}\phi_3(z^3_1)\phi_2(\eta_2)-\frac{1}{\lambda_2}\phi^2_2(\eta_2)+ (\dfrac{\lambda_3}{\lambda_2}-1)\phi_2(\eta_2)+\,\phi_3(z^3_1)\right)>0
\end{aligned}
\right\}.
\end{align}
\end{theorem}
\begin{proof}[of Theorem \ref{the3.3}] The proof is carried out in three steps. From \eqref{ee3.27}-\eqref{a3.32}, we have 
\begin{align}\label{a3.33}
\begin{cases}
\dot\eta_2= -\phi_3(\eta_3),\\
   \dot z^3_1=-\phi_1(\eta_1)-\dot\eta_2,\\
   \dot z^3_2=- \phi_2(\eta_2)-\dot z^3_1+u^* - u 
\end{cases}
\end{align}
\begin{itemize}
\item[\bf Step (i)] To stabilize $\eta_{2}$, consider the Lyapunov function

 \begin{align}
     \omega_{1}(\eta_{2})=\lambda_{2}\bigl(e^{\eta_{2}}-1-\eta_{2}\bigr).
 \end{align}

Its time derivative along system trajectories is

 \begin{align}
\dot \omega_{1}(\eta_2)=-\phi_2(\eta_2)\phi_3(\eta_3).
 \end{align}
To ensure decay, we choose the fictitious control $\eta_3$ with $c_1 > 0$ such that
\begin{align}\label{aa3.45}
    \phi_3\big( \eta_2 \big) = c_1\,\phi_2(\eta_2) 
\end{align}

\item[\bf Step (ii)] Thus, by considering the tracking error for $\eta_3$, we write

 \begin{align}\label{a3.37}
     \lambda_3\big(e^{z^3_1+\eta_2} - 1\big) 
 = \phi_3(z^3_1) e^{\eta_2} + \lambda_3\big(e^{\eta_2} - 1\big).
 \end{align}

Thanks to \eqref{aa3.45}, it follows that

\begin{align}
     \lambda_3\big(e^{z^3_1+\eta_2} - 1\big) 
 = \phi_3(z^3_1) e^{\eta_2} + c_1\,\phi_2(\eta_2).
\end{align}
 Substituting this into $\dot \omega_1(\eta_2)$ yields
\begin{align}
\dot \omega_{1}(\eta_2)
 = -c_1\,\phi_2^2(\eta_2) - \lambda_3\,\phi_2(\eta_2)\,e^{\eta_2}\big(e^{z^3_1} - 1\big).
\end{align}
From system \eqref{a3.33}, we have
\begin{align}
 \dot z^3_1 = -\phi_1(\eta_1)
 + \,\lambda_3\big(e^{z^3_1+\eta_2}-1\big).
\end{align}



 We define the composite Lyapunov function for $(\eta_2,z^3_1)$ as

\begin{align}
     \omega_2(\eta_2,z^3_1) =\omega_1(\eta_2) + \lambda_3\big(e^{z^3_1} - 1-z^3_1\big),
\end{align}

 and we have

 \begin{align}
 \dot \omega_2(\eta_2,z^3_1)= -c_1\,\phi_2^2(\eta_2) - \frac{c_1}{\lambda_3}\phi^2_2(\eta_2)\,\phi_3(z^3_1)-\phi_2(\eta_2)\,\phi_3(z^3_1)+ \phi_3(z^3_1)\,\dot z^3_1. 
\end{align}
Then, 
\begin{align}
\dot \omega_2(\eta_2,z^3_1)&= -c_1\,\phi_2^2(\eta_2)-\phi_3(z^3_1)\phi_1(\eta_1) - \frac{c_1}{\lambda_3}\phi^2_2(\eta_2)\,\phi_3(z^3_1)-\phi_2(\eta_2)\,\phi_3(z^3_1)+\frac{c_1}{\lambda_3}\phi^2_3(z^3_1)\phi_2(\eta_2)+\phi^2_3(z^3_1) + c_1\,\phi_3(z^3_1)\,\phi_2(\eta_2).  
\end{align}


\item[\bf Step (iii)] To stabilize the dynamics \(z^3_1\), we choose the constant \(c_2\) such that 

\begin{align}\label{aa3.53}
\phi_1\big(z^3_1\big) = c_2 \, \phi_3(z^3_1),\quad\text{for}\; c_2>0.
\end{align}
This fictitious control is designed to achieve partial stabilization of the state $\eta_3.$ Also, by computing,
\begin{align}\label{a3.46}
\phi_1(z^3_2+z^3_1) = \phi_1(z^3_2) e^{z^3_1}+\phi_1\big(z^3_1\big).
\end{align}
and thanks to \eqref{aa3.53}, we get
\begin{align}
\phi_1(\eta_1)=\phi_1(z_2) e^{z_1}+c_2\,\phi_3(z_1).
\end{align}
Hence,
\begin{align}
\dot \omega_2(\eta_2,z^3_1)&= -c_1\,\phi_2^2(\eta_2)-c_2\,\phi_3^2(z^3_1) -\phi_3(z^3_1)\phi_1(z^3_2)-\frac{c_2}{\lambda_1}\phi^2_3(z^3_1)\phi_1(z^3_2)- \frac{c_1}{\lambda_3}\phi^2_2(\eta_2)\,\phi_3(z^3_1)-\phi_2(\eta_2)\,\phi_3(z^3_1) +\frac{c_1}{\lambda_3}\phi^2_3(z^3_1)\phi_2(\eta_2)
\end{align}
\begin{align*}
+\phi^2_3(z^3_1) + c_1\phi_3(z^3_1)\,\phi_2(\eta_2).
\end{align*}

Let the global Lyapunov function \eqref{equation3.44} as
\begin{align}\label{a3.49}
V_3(\eta_2, z^3_1, z^3_2)= \theta \omega_2(\eta_2,z^3_1)+\lambda_1(e^{z^3_2}-1-z^3_2).
\end{align}
We have its derivative of the form
\begin{align}
\dot V_3(\eta_2, z^3_1, z^3_2)&= -\theta c_1\,\phi_2^2(\eta_2)-\theta c_{2}\,\phi_3^2(z^3_1)+\theta\,\phi_3^2(z^3_1) -\theta \phi_3(z_1)\phi_1(z^3_2)-\theta\frac{c_2}{\lambda_1}\phi^2_3(z^3_1)\phi_1(z^3_2)- \theta \frac{c_1}{\lambda_3}\phi^2_2(\eta_2)\,\phi_3(z^3_1)-\theta \phi_2(\eta_2)\,\phi_3(z^3_1) 
\end{align}
\begin{align*}
+\theta\frac{c_1}{\lambda_3}\phi^2_3(z^3_1)\phi_2(\eta_2)+\theta c_1\phi_3(z^3_1)\,\phi_2(\eta_2)+\phi_1(z^3_2)\left(- \phi_2(\eta_2)-\dot z^3_1+u^* - u \right).
\end{align*}
It follows that 


\begin{align}
 \dot V_3(\eta_2, z^3_1, z^3_2)&= -\theta c_1\,\phi_2^2(\eta_2)-\theta c_{2}\,\phi_3^2(z^3_1)+\theta\,\phi_3^2(z^3_1) -\theta\phi_3(z^3_1)\phi_1(z^3_2)-\theta\frac{c_2}{\lambda_1}\phi^2_3(z^3_1)\phi_1(z^3_2)- \theta\frac{c_1}{\lambda_3}\phi^2_2(\eta_2)\,\phi_3(z^3_1)-\theta\phi_2(\eta_2)\,\phi_3(z^3_1) 
\end{align}
\begin{align*}
+\theta\frac{c_1}{\lambda_3}\phi^2_3(z^3_1)\phi_2(\eta_2)+ \theta c_1\phi_3(z^3_1)\,\phi_2(\eta_2)+\phi_1(z^3_2)\left((c_2-1)\phi_3(z^3_1)+\phi_1(z^3_2)+\frac{c_2}{\lambda_1}\phi_3(z^3_1)\phi_1(z^3_2)-(c_1+1)\phi_2(\eta_2)-\frac{c_1}{\lambda_3}\phi_3(z^3_1)\phi_2(\eta_2)\right)
\end{align*}
\begin{align*}
+\phi_1(z^3_2)\left(u^* - u \right).
\end{align*}

Finally, with the control law of the form
\begin{align}\label{aa3.64}
\begin{aligned}
 u&=u^*+c_{3o}\phi_1(z^3_2)-\theta\frac{c_2}{\lambda_1}\phi^2_3(z^3_1)+(c_2-1-\theta)\phi_3(z^3_1)+\frac{c_2}{\lambda_1}\phi_3(z^3_1)\phi_1(z^3_2)-\frac{c_1}{\lambda_3}\phi_3(z^3_1)\phi_2(\eta_2)-(c_1+1)\phi_2(\eta_2)\\\
 &+\frac{\theta\phi_3(z^3_1)}{\phi_1(z^3_2)}\left(\frac{c_1}{\lambda_3}\phi_3(z^3_1)\phi_2(\eta_2)-\frac{c_1}{\lambda_3}\phi^2_2(\eta_2)+ (c_1-1)\phi_2(\eta_2)+\,\phi_3(z^3_1)\right),
 \end{aligned}%
\end{align}

we obtain the time derivative of the Lyapunov control function 
\begin{align}\label{ae3.63}
\dot V_3(\eta_2,z^3_1,z^3_2)=-\theta c_1\,\phi_2^2(\eta_2)-\theta c_{2}\,\phi_3^2(z^3_1)-c_{3}\phi^2_1(z^3_2).
\end{align}
\end{itemize} 
Therefore, global asymptotic stability follows.
\end{proof}
\begin{remark}
Using \eqref{aa3.45}–\eqref{aa3.53}, we obtain $c_1=\dfrac{\lambda_3}{\lambda_2}$ and $c_2=\dfrac{\lambda_1}{\lambda_3}$. Substituting these values into \eqref{aa3.64} and \eqref{ae3.63} yields 
\begin{align*}
\begin{aligned}
 u&=u^*+c_{3o}\phi_1(z^3_2)-\theta\frac{1}{\lambda_3}\phi^2_3(z^3_1)+(\dfrac{\lambda_1}{\lambda_3}-1-\theta)\phi_3(z^3_1)+\frac{1}{\lambda_3}\phi_3(z^3_1)\phi_1(z^3_2)-\frac{1}{\lambda_2}\phi_3(z^3_1)\phi_2(\eta_2)-(\dfrac{\lambda_3}{\lambda_2}+1)\phi_2(\eta_2)\\\
 &+\frac{\theta\phi_3(z^3_1)}{\phi_1(z^3_2)}\left(\frac{1}{\lambda_2}\phi_3(z^3_1)\phi_2(\eta_2)-\frac{1}{\lambda_2}\phi^2_2(\eta_2)+ (\dfrac{\lambda_3}{\lambda_2}-1)\phi_2(\eta_2)+\,\phi_3(z^3_1)\right)
 \end{aligned}%
\end{align*}

and 
\begin{align*}
\dot V_3(\eta_2,z^3_1,z^3_2)=-\theta \dfrac{\lambda_3}{\lambda_2}\,\phi_2^2(\eta_2)-\theta \dfrac{\lambda_1}{\lambda_3}\,\phi_3^2(z^3_1)-c_{3}\phi^2_1(z^3_2).
\end{align*}
\end{remark}
\begin{remark}
Since the Lyapunov control function $V_3(\eta_2,z^3_1,z^3_2)$ is radially unbounded, its level sets $\mathcal{L}_c=\left\{ \eta\in \mathbb{R}^3 \;\middle|\; V_3(\eta_2,z^3_1,z^3_2)\leq c \right\}$ are bounded for any $c>0$. In particular, one can choose $c$ such that these level sets are contained in the positively invariant set \eqref{equation3.45}. Moreover, this set $\mathcal{K}$ is well-defined.  Let $(\eta_1,\eta_2,\eta_3)=(\varepsilon,0,0)$ with $\varepsilon\neq0$ small enough. Then $z^3_1=0, \quad z^3_2=\varepsilon,$ so that $\phi_3(z^3_1)=0$, $\phi_2(\eta_2)=0$, and $\phi_1(z^3_2)=\phi_1(\epsilon)\neq0$. Consequently, the constraint defining $\mathcal K$ yields $u^*+c_{3o}\phi_1(\epsilon)>0$, hence $(\varepsilon,0,0)\in\mathcal K$ and thus $\mathcal K\neq\varnothing$. Furthermore, $V_3(0,0,\varepsilon)=\lambda_1(e^{\varepsilon}-1-\varepsilon)\xrightarrow[\varepsilon\to0]{}0.$ By continuity, for any $c>0$ there exists $\varepsilon\neq0$ small enough such that $V_3(0,0,\varepsilon)\leq c$. It follows that $(\varepsilon,0,0)\in\mathcal L_c\cap\mathcal K$, and thus $\mathcal L_c\cap\mathcal K\neq\varnothing$.
\end{remark}
\begin{remark}
The control in \eqref{aa3.64} may be rewritten as
\begin{align}
\begin{aligned}\label{equation3.61}
 u&=u^*+\frac{1}{\phi_1(z^3_2)}\left(c_{3}\phi^2_1(z^3_2)+\theta c_1\phi^2_2(\eta_2)+\theta c_2\phi^2_3(z^3_1)\right)\\
 &+\frac{\phi_1(z^3_2)-\theta\phi_3(z^3_1)}{\phi_1(z^3_2)}\phi_1(\eta_1)-\phi_2(\eta_2)+\frac{\theta\phi_3(z^3_1)-\phi_1(z^3_2)-\theta\phi_2(\eta_2)}{\phi_1(z^3_2)}\phi_3(\eta_3),
 \end{aligned}%
\end{align} 
\end{remark}
\begin{remark}
It was proved in \cite{ref69} that the state $\psi_i$ of the internal
dynamics are restricted to the sets 
\begin{align}\label{a3.61}
\mathcal{S}_i=\left\lbrace \psi_i\in\; C^0((-A,0);(-1,\infty)) : P(\psi_i)=0\wedge\psi_i(0)=\displaystyle\int_0^A\tilde{k}_i(a)\psi_i(-a)da\right\rbrace,
\end{align}
where
\begin{align*}
P(\psi_i)=\dfrac{\displaystyle\int_0^A\psi_i(-a)\displaystyle\int_a^A\tilde{k}_i(s)dsda}{\displaystyle\int_a^Aa\tilde{k}_i(a)da},
\end{align*}
and that the state $\psi_i$ is globally
exponentially stable in $L^{\infty}$ norm, which means that there exist $M_i>1,\sigma_i>0$ such that 
\begin{align}\label{aa3.65}
    \Vert\psi_i(t-a) \Vert \leq M_ie^{-\sigma_i t} \Vert\psi_{i,0} \Vert_{\infty}.
\end{align}
\end{remark}
\begin{figure}[H]
    \centering \includegraphics[width=0.7\textwidth]{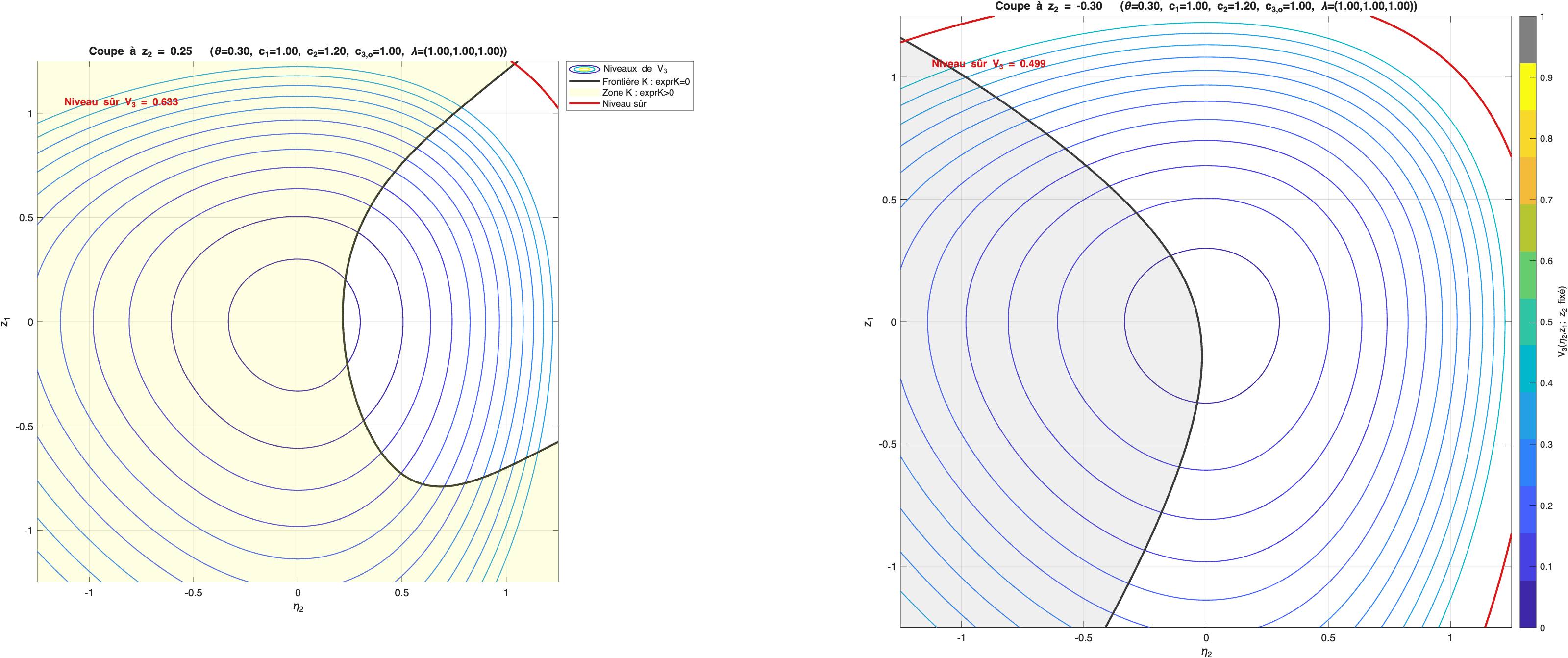}
        \caption{Fig. 2. Level sets of $V_3(\eta_2,z^3_1,z^3_2)$ for two fixed values of $z_2.$ Colored contours represent the values of the Lyapunov function $V_3,$ with hotter colors corresponding to larger values. The gray shaded region indicates the positively invariant set $\mathcal{K}$, defined by the system constraints. For each slice, the largest level set of $V_3$ entirely contained within $\mathcal{K}$ is highlighted in red. This contour provides a practical estimate of the region of attraction of the equilibrium under the imposed constraints (positivity of the control and admissible state bounds). The plots show that the boundary of $\mathcal{K}$ is strongly influenced by the choice of $z^3_2,$ while the geometry of $V_3$ remains convex due to its entropic structure. Increasing $\theta$ or $c_{3o}$ tends to enlarge the invariant domain (more dissipation), whereas increasing $c_2$ reduces it along the $z^3_1-$direction through the quadratic term in $\phi_3(z^3_1)$}
    \end{figure}
\begin{figure}[H]
    \centering \includegraphics[width=0.7\linewidth]{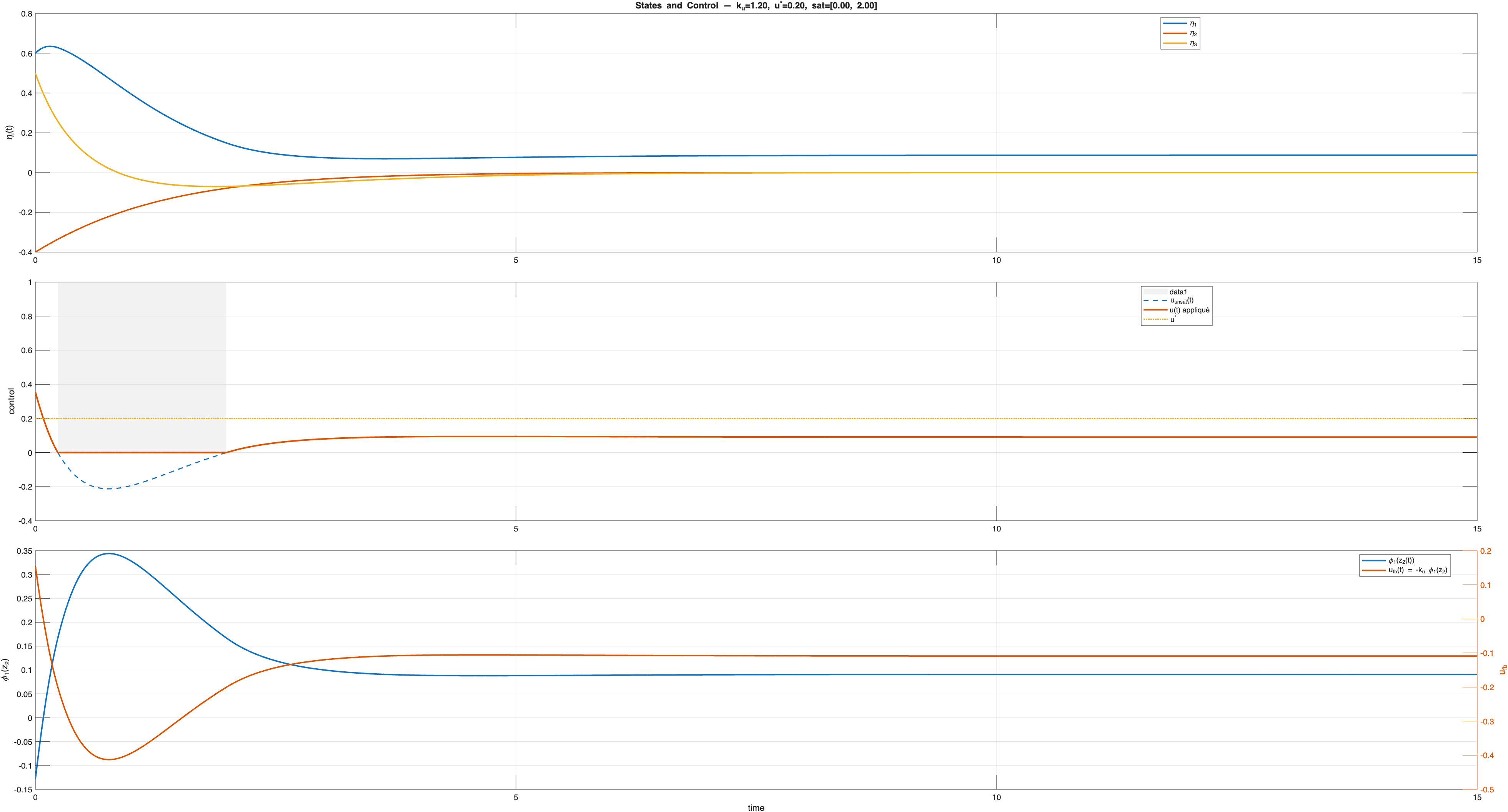}
        \caption{Figure 3. Time evolution of the closed-loop system under the Lyapunov-based feedback law. The top panel shows the states $\eta_1(t),$ $\eta_2(t),$ $\eta_3(t),$ all converging towards the origin. The middle panel compares the applied control $u(t)$ with the unsaturated control signal $u_{\text{unsat}}(t)$ and the baseline term $u^*.$ The gray shaded areas indicate intervals where the control signal reaches the saturation bounds. The bottom panel illustrates the nonlinear feedback component: the function $\phi_1(z^3_2(t))$ (left axis) and its contribution to the feedback $u_{\text{fb}}(t) = -k_u \phi_1(z^3_2(t))$ (right axis). This decomposition highlights the role of the nonlinear terms in shaping the control action. Overall, the plots confirm that the feedback ensures state convergence while maintaining the control within admissible bounds.}
    \label{fig:placeholder}
\end{figure}
\begin{lemma}\label{lem3.11}
Under the model \eqref{e3.1} assumptions, there exists $o_i>0$ such that for all $t\in \mathbb{R}_+$ and all $i$ 
\begin{align}
\vert\eta_i(t)\vert\leq o_i.
\end{align}
\end{lemma}
\begin{proof}[of Lemma \ref{lem3.11}]
We rewrite the system \eqref{e3.10} as
\begin{equation}\label{ae3.67}
\begin{aligned}
\begin{cases}
\dot \eta_1(t) &= u^*-u(t) - \phi_2(\eta_2)\;-\; e^{\eta_2(t)}\,r_2(t),\\
\dot \eta_2(t) &= -\phi_3(\eta_3)\;-\; e^{\eta_3(t)}\,r_3(t),\\
\dot \eta_3(t) &= -\phi_1(\eta_1)\;-\; e^{\eta_1(t)}\,r_1(t),
\end{cases}\quad\text{where}\; \begin{cases}
r_2(t)&=\displaystyle\int_0^A g_1(a)\,x_2^*(a)\,\psi_2(t-a)\,da,\\
r_3(t)&=\displaystyle\int_0^A g_2(a)\,x_3^*(a)\,\psi_3(t-a)\,da,\\
r_1(t)&=\displaystyle\int_0^A g_3(a)\,x_1^*(a)\,\psi_1(t-a)\,da.
\end{cases}
\end{aligned}
\end{equation}
\begin{itemize}
\item When $\psi\equiv 0$ (i.e. $r\equiv 0$), there exists $\kappa>0$ and a Lyapunov function $V_3(\eta_2,z^3_1,z^3_2)$ in \eqref{a3.49} which is radially unbounded in $(\eta_2,z^3_1,z^3_2)$, such that along the closed-loop trajectories 
\begin{align}\label{ae3.68}
\dot V_3(\eta_2, z^3_1, z^3_2)\ \le\ -\,\kappa\Big(\phi_1(z^3_2)^2+\phi_2(\eta_2)^2+\phi_3(z^3_1)^2\Big).
\end{align}
where $\kappa=\min(\theta c_1,\theta c_2,c_3)>0.$

\item Case $\psi\neq 0.$ Since $g_i,\,x_j^*\in L^2(0,A)$, the product $g_i x_j^*\in L^1(0,A)$ by the Cauchy-Schwarz inequality. From \eqref{aa3.65} there exist constants $D_i>0$ and
$\sigma:=\min\{\sigma_1,\sigma_2,\sigma_3\}>0$ such that, for all $t\ge0$ and $i,j=1,2,3$,
\begin{align}\label{ae3.69}
|r_i(t)| \le D_i\,e^{-\sigma t}.
\end{align}
From \eqref{ae3.68}, introducing $\psi_i\neq 0$ produces cross terms of the form
\begin{align}
\mathcal{C}_{ij}(t):=\phi_i(\cdot)\,e^{\eta_j(t)}\,r_j(t),
\end{align}
which originate from the $-e^{\eta_j}r_j$ terms in \eqref{ae3.67} when computing $\dot V_3.$ Fix $\varepsilon\in(0,\kappa/3)$ and apply Young's inequality $|ab|\le \tfrac{\varepsilon}{2}a^2+\tfrac{1}{2\varepsilon}b^2$ with $a=\phi_i(\cdot)$ and $b=e^{\eta_j}r_j$. We get
\begin{align}
|\mathcal{C}_{ij}(t)|\le \frac{\varepsilon}{2}\,\phi_i^2(\cdot)\;+\;\frac{1}{2\varepsilon}\,e^{2\eta_j(t)}\,r_j(t)^2.
\end{align}
Since, 
\begin{align}
e^{2\eta_j}\le 2\bigl(1+\phi_j(\eta_j)^2/\lambda_j^2\bigr),
\end{align}
hence
\begin{align}
|\mathcal{C}_{ij}(t)|\le \frac{\varepsilon}{2}\,\phi_i^2(\cdot)\;+\;\frac{1}{\varepsilon}\,r_j(t)^2\;+\;\frac{1}{\varepsilon\,\lambda_j^2}\,\phi_j(\eta_j)^2\,r_j(t)^2.
\end{align}

Summing these contributions and using $r\in L^\infty$  yields some $C>0$ such that
\begin{align}\label{ae3.74}
\sum_{i,j}|\mathcal{C}_{ij}(t)|
\le \varepsilon\Big(\phi_1(z^3_2)^2+\phi_2(\eta_2)^2+\phi_3(z^3_1)^2\Big)\;+\;C\,\|r(t)\|^2.
\end{align}

Inserting \eqref{ae3.74} into  $\dot V_3$ of \eqref{ae3.68} and by choosing $\varepsilon<\kappa/2$, we obtain

\begin{align}\label{ae3.75}
\dot V_3(\eta_2, z^3_1, z^3_2) \le -\frac{\kappa}{2}\Big(\phi_1(z^3_2)^2+\phi_2(\eta_2)^2+\phi_3(z^3_1)^2\Big)
+ C\,\|r(t)\|^2,
\end{align}

From \eqref{ae3.75} and \eqref{ae3.69} we have

\begin{align}
\dot V_3(\eta_2, z^3_1, z^3_2)(t)\ \le\ -\tfrac{\kappa}{2}W(\eta,z)\ +\ C\,D^2 e^{-2\sigma t}
\ \le\ C\,D^2 e^{-2\sigma t},
\end{align}
where $W:=\phi_1(z^3_2)^2+\phi_2(\eta_2)^2+\phi_3(z^3_1)^2\ge0$.
Integrating on $[0,t]$ yields
\begin{align}
V_3(\eta_2, z^3_1, z^3_2)(t)\le V_3(\eta_2, z^3_1, z^3_2)(0)+\frac{C\,D^2}{2\sigma},
\end{align}

so $V_3(\eta_2, z^3_1, z^3_2)(t)$ is uniformly bounded on $[0,\infty)$.
Since $V_3(\eta_2, z^3_1, z^3_2)$ is radially unbounded in $(\eta_2,z^3_1,z^3_2)$, the trajectories $(\eta_2,z^3_1,z^3_2)$ remain in a compact level set of $V_3$; in particular $\eta_2,z^3_1,z^3_2$ are bounded.
Hence there exists $o_i>0$ such that $|\eta_i(t)|\le o_i$ for all $t\ge0$. 
\end{itemize}
\end{proof}
\paragraph{\bf 3.1.1.2. Stability for $\psi_i\neq 0$}
Firstly, we introduce for $i,j\in\{1,2,3\}$ (with $j\neq i$) the functions: 
$$v_i : \mathcal{S}_i\longrightarrow\mathbb{R}_+$$
by 
\begin{align}
    v_i(\psi_{i,t})
= \ln\Bigl(1 + \int_0^A \bar g_j(a)\,\psi_i(t-a)\,\mathrm{d}a\Bigr),
\end{align}

where

\begin{align}
    \bar g_j(a)
= \frac{g_i(a)\,x^*_j(a)}{\displaystyle\int_0^A g_i(s)\,x^*_j(s)\,\mathrm{d}s},
\quad
\int_0^A \bar g_j(a)\,\mathrm{d}a = 1.
\end{align}

From equations \eqref{e3.10}, it is straightforward to obtain
\begin{align*}
\dot\eta_1(t)
= u^* - u(t)-\lambda_2\bigl(e^{\eta_2(t)}\,e^{v_2(\psi_{2})}-1\bigr),
\end{align*}
after a few transformations, or alternatively
\begin{align*}
\dot{\eta}_1(t)= u^* - u(t) - \phi_2(\eta_2+v_2(\psi_2)).
\end{align*}

By analogy, one obtains for $\eta_2,\;\eta_3$:
\begin{align*}
\dot{\eta}_2(t)
= -\phi_3\bigl(\eta_3(t) + v_3(\psi_{3})\bigr),\\
\dot{\eta}_3(t)
= -\phi_1\bigl(\eta_1(t) + v_1(\psi_{1})\bigr)
\end{align*}
In conclusion,  the closed-loop system is given by

\begin{align}\label{ee3.58}
    \begin{cases}
\dot{\eta}_1(t)=  - \phi_2(\eta_2(t)+v_2(\psi_2))+u^* - u(t),\\[6pt]
\dot{\eta}_2(t)
= -\phi_3\bigl(\eta_3(t) + v_3(\psi_{3})\bigr),\\[6pt]
\dot{\eta}_3(t)
= -\phi_1\bigl(\eta_1(t) + v_1(\psi_{1})\bigr),\\[6pt]
\psi_i(t)= \displaystyle\int_0^A \tilde k_i(a)\,\psi_i(t-a)\,\mathrm{d}a,
\end{cases}
\end{align}
For the remainder of the calculations, we set 
\begin{align}\label{a3.60}
\hat{\phi}_i=\phi_i\bigl(\eta_i(t) + v_i(\psi_{i})\bigr),    
\end{align}
\textit{We make the following assumption (see \cite{ref69}}):
\begin{enumerate}
 \item[] \noindent\textbf{Assumption H6 :} There exist constants $\kappa_i$ such that $\int_{0}^{A}\Bigl|\,
 \title{k}_i(a)\,
 \;-\;z_i\kappa_i\!\int_a^A \title{k}_i(s)\,\,ds\Bigr|\;da
 \;<\;1, $ where; $ z_i = \Bigl(\int_{0}^{A}a\,\title{k}_i(a)\,da\Bigr)^{-1}.$  Let $\sigma_i>0$  be a sufficiently small constant that satisfies the inequality 
 $\int_{0}^{A}\Bigl|\,\title{k}_i(a)\,\;-\;z_i\kappa_i\!\int_a^A \title{k}_i(s)\,ds\Bigr|\;e^{\sigma_i a}da<\;1.$
 \end{enumerate}

Before stating the main result of this section, we define the following functions. Let the functional
\begin{align}\label{}
    G_i(\psi_i)
= \dfrac{\max_{a\in(0,A)}\bigl|\psi_i(t-a)\bigr|\,e^{-a\sigma_i}}{1+\max(0,\min_{a\in(0,A)}\psi_i(t-a))},
\end{align}
whose Dini derivative satisfies (see \cite{ref69})
\begin{align}\label{a3.63}
    D^+\bigl(G_i(\psi_{i,t})\bigr)
\;\le\;
-\sigma_i\,G_i(\psi_{i,t})
\end{align}

We then define the following Lyapunov function
\begin{align}\label{a3.64}
V_G(\eta,\psi)
= V_3(\eta_2,z^3_1,z^3_2)\;+\;\frac{\gamma_1}{\sigma_1}h(G_1(\psi_1))\;+\;\frac{\gamma_2}{\sigma_2}h(G_2(\psi_2))\;+\;\frac{\gamma_3}{\sigma_3}h(G_3(\psi_3))
\end{align}
with the function 
\begin{align}
h(p)=\displaystyle\int_0^p\frac{1}{z}(e^z-1)^2dz.
\end{align}
being positive definite and radially unbounded.  We denote  $\;S^{N}:=S_{1}\times\cdots\times S_{N},\; N\in\mathbb{N}^*.$ 
\begin{theorem}\label{the3.4}
Under {\bf Assumption H6},  system \eqref{ee3.58} is globally asymptotically stable and the control remains uniformly bounded. Moreover, the control satisfies $u(t)>0$ for every $t>0,$ for every $\eta_i(0)$ belonging to the largest
level set of $V_G(\eta,\psi)$ within the set
\begin{equation}
\mathcal{K}_3=
\left\{ (\eta,\psi)\in \mathbb{R}^3\times\mathcal{S}^3\ \middle|\ 
\begin{alignedat}{2}
&\eta_1 \le \ln\left(\frac{\gamma_1}{C_1\lambda_1}\right), \qquad && \gamma_1 > C_1\lambda_1,\\[4pt]
&\eta_2 \le \ln\left(\frac{\gamma_2}{C_2\lambda_2}\right), \qquad && \gamma_2 > C_2\lambda_2,\\[4pt]
&\eta_3 \le \ln\left(\frac{\gamma_3}{C_3\lambda_3}\right), \qquad && \gamma_3 > C_3\lambda_3,\\[6pt]
&u^*+c_{3o}\phi_1(z^3_2)-\theta\frac{c_2}{\lambda_1}\phi^2_3(z^3_1)+(c_2-1-\theta)\phi_3(z^3_1)\\
&+\frac{c_2}{\lambda_1}\phi_3(z^3_1)\phi_1(z^3_2)-\frac{c_1}{\lambda_3}\phi_3(z^3_1)\phi_2(\eta_2)-(c_1+1)\phi_2(\eta_2)\\
&+\frac{\theta\phi_3(z^3_1)}{\phi_1(z^3_2)}\left(\frac{c_1}{\lambda_3}\phi_3(z^3_1)\phi_2(\eta_2)-\frac{c_1}{\lambda_3}\phi^2_2(\eta_2)+ (c_1-1)\phi_2(\eta_2)+\,\phi_3(z^3_1)\right)>0.
\end{alignedat}
\right\}.
\end{equation}
\end{theorem}
\begin{proof}[of Theorem \ref{the3.4}]
By applying the same strategy as in Theorem \ref{the3.3} to system \eqref{ee3.58} and using equation \eqref{a3.49}, we obtain the following  relation 
\begin{align}
V_3(\eta_2,z^3_1,z^3_2)=\theta \omega_2(\eta_2,z^3_1)+ \lambda_1(e^{z^3_2}-1-z^3_2)
\end{align}
The derivative of the Lyapunov function $V_3$ is given as follows

\begin{align}
\dot V_3(\eta_2, z^3_1, z^3_2)= -\theta\phi_2(\eta_2)\hat{\phi}_3 + \theta\phi_3(z^3_1)(-\hat{\phi}_1-\dot\eta_2)+\phi_1(z^3_2)(-\hat{\phi}_2+u^*-u-\dot z^3_1).
\end{align}



Thanks to control \eqref{aa3.64}, we obtain
\begin{align}
\dot V_3(\eta_2, z^3_1, z^3_2)=-\theta c_1\,\phi_2^2(\eta_2)-\theta c_{2}\,\phi_3^2(z^3_1)-c_{30}\phi^2_1(z^3_2)
\end{align}
\begin{align*}
+\phi_3(z^3_1)\left(\theta c_{2}\,\phi_3(z^3_1)+\theta\frac{c_1}{\lambda_3}\phi^2_2(\eta_2)-\theta\frac{c_1}{\lambda_3}\phi_3(z^3_1)\phi_2(\eta_2)-\theta (c_1-1)\phi_2(\eta_2)-\theta\phi_3(z^3_1)\right)+\phi_2(\eta_2)\left(\theta c_1\,\phi_2(\eta_2)+(c_1+1)\phi_1(z^3_2)-\phi_1(z^3_2)\right) 
\end{align*}
\begin{align*}
+\phi_1(z^3_2)\phi_3(z^3_1)\left(\theta\frac{c_2}{\lambda_1}\phi_3(z^3_1)-(c_2-1-\theta)-\frac{c_2}{\lambda_1}\phi_1(z^3_2)+\frac{c_1}{\lambda_3}\phi_2(\eta_2)\right)+\left(\theta\phi_3(z^3_1)-\phi_1(z^3_2)-\theta\phi_2(\eta_2)\right)\phi_3(\eta_3)+ \left(\phi_1(z^3_2)-\theta\phi_3(z^3_1)\right)\phi_1(\eta_1)
\end{align*}

\begin{align*}
+\underbrace{\left(\theta\phi_3(z^3_1)-\phi_1(z^3_2)-\theta\phi_2(\eta_2)\right)}_{A_3}\left(\hat{\phi}_3-\phi_3(\eta_3)\right)+\underbrace{\left(\phi_1(z^3_2)-\theta\phi_3(z^3_1)\right)}_{A_1}\left(\hat{\phi}_1-\phi_1(\eta_1)\right)\underbrace{-\phi_1(z^3_2)}_{A_2}\left(\hat{\phi}_2-\phi_2(\eta_2)\right).
\end{align*}

Furthermore,
\begin{align}
\dot V_3(\eta_2, z^3_1, z^3_2)=-\theta c_1\,\phi_2^2(\eta_2)-\theta c_{2}\,\phi_3^2(z^3_1)-c_{3}\phi^2_1(z^3_2)+A_3\left(\hat{\phi}_3-\phi_3(\eta_3)\right)+A_1\left(\hat{\phi}_1-\phi_1(\eta_1)\right)+A_2\left(\hat{\phi}_2-\phi_2(\eta_2)\right)+\mathcal{R},
\end{align}
with 
\begin{align}
\mathcal{R}=\phi_3(z^3_1)\left(\theta c_{2}\,\phi_3(z^3_1)+\theta\frac{c_1}{\lambda_3}\phi^2_2(\eta_2)-\theta\frac{c_1}{\lambda_3}\phi_3(z^3_1)\phi_2(\eta_2)-\theta (c_1-1)\phi_2(\eta_2)-\theta\phi_3(z^3_1)\right)+\phi_2(\eta_2)\left(\theta c_1\,\phi_2(\eta_2)+c_1\phi_1(z^3_2)\right) 
\end{align}
\begin{align*}
+\phi_1(z^3_2)\phi_3(z^3_1)\left(\theta\frac{c_2}{\lambda_1}\phi_3(z^3_1)-(c_2-1-\theta)-\frac{c_2}{\lambda_1}\phi_1(z^3_2)+\frac{c_1}{\lambda_3}\phi_2(\eta_2)\right)+\left(\theta\phi_3(z^3_1)-\phi_1(z^3_2)-\theta\phi_2(\eta_2)\right)\phi_3(\eta_3)+ \left(\phi_1(z^3_2)-\theta\phi_3(z^3_1)\right)\phi_1(\eta_1).
\end{align*}

We have from \eqref{a3.37}-\eqref{a3.46}
\begin{align}\label{ea3.93}
\begin{cases}\phi_1(\eta_1)=c_2\phi_3(z^3_1)+\phi_1(z^3_2)+\frac{c_2}{\lambda_1}\phi_1(z^3_2)\phi_3(z^3_1),\\
\phi_3(\eta_3)=c_1\phi_2(\eta_2)+\phi_3(z^3_1)+\frac{c_1}{\lambda_3}\phi_3(z^3_1)\phi_2(\eta_2),
\end{cases}
\end{align}
thus
\begin{align}
\mathcal{R}=0
\end{align}
Hence, the derivative of $V_3$ now yields
\begin{align}\label{aa3.81}
\dot V_3(\eta_2, z^3_1, z^3_2)=-\theta c_1\,\phi_2^2(\eta_2)-\theta c_{2}\,\phi_3^2(z^3_1)-c_{3}\phi^2_1(z^3_2)+A_3\left(\hat{\phi}_3-\phi_3(\eta_3)\right)+A_1\left(\hat{\phi}_1-\phi_1(\eta_1)\right)+A_2\left(\hat{\phi}_2-\phi_2(\eta_2)\right)
\end{align}

From \eqref{a3.64} and thanks to \eqref{a3.63}, we get   
\begin{align}\label{}
\dot V_G(\eta,\psi)
\leq  \dot V_3(\eta_2, z^3_1, z^3_2)\;-\;\gamma_1(e^{G_1}-1)\;-\;\gamma_2(e^{G_2}-1)\;-\;\gamma_3(e^{G_3}-1)
\end{align}

\begin{align}
\dot V_G(\eta,\psi)
\leq-\theta c_1\,\phi_2^2(\eta_2)-\theta c_{2}\,\phi_3^2(z^3_1)-c_{3}\phi^2_1(z^3_2)+\vert A_1\vert\vert\hat{\phi}_1-\phi_1\vert+\vert A_2\vert\vert\hat{\phi}_2-\phi_2\vert+\vert A_3\vert\vert\hat{\phi}_3-\phi_3\vert-\;\gamma_1(e^{G_1}-1)-\gamma_2(e^{G_2}-1)-\gamma_3(e^{G_3}-1).
\end{align}
We have also $\hat{\phi}_i-\phi_i=(\phi_i+\lambda_i)(e^{v_i}-1),$ then 
\begin{align}
\dot V_G(\eta,\psi)
\leq-\theta c_1\,\phi_2^2(\eta_2)-\theta c_{2}\,\phi_3^2(z^3_1)-c_{3}\phi^2_1(z^3_2)+\vert A_1\vert\vert\phi_1(\eta_1)+\lambda_1\vert(e^{v_1}-1)+\vert A_2\vert\vert\phi_2(\eta_2)+\lambda_2\vert(e^{v_2}-1)+\vert A_3\vert\vert\phi_3(\eta_3)+\lambda_3\vert(e^{v_3}-1)
\end{align}
\begin{align*}
-\gamma_1(e^{G_1}-1)-\gamma_2(e^{G_2}-1)-\gamma_3(e^{G_3}-1)
\end{align*}
From Lemma \ref{lem3.11}, there exists $C_i>0$ such that $\vert A_i\vert\leq C_i.$ Then, we get 
\begin{align}
\dot V_G(\eta,\psi)
\leq-\theta c_1\,\phi_2^2(\eta_2)-\theta c_{2}\,\phi_3^2(z^3_1)-c_{3}\phi^2_1(z^3_2)+\left(C_1\vert\phi_1(\eta_1)+\lambda_1\vert-\gamma_1\right)(e^{G_1}-1)+\left(C_2\vert\phi_2(\eta_2)+\lambda_2\vert-\gamma_2\right)(e^{G_2}-1)
\end{align}
\begin{align*}
+\left(C_3\vert\phi_3+\lambda_3\vert-\gamma_3\right)(e^{G_3}-1)
\end{align*}
By restricting $\eta$ as in 
\begin{align}
\begin{cases}
\eta_1\leq \ln\left(\frac{\gamma_1}{C_1\lambda_1}\right),\\
\\\eta_2\leq \ln\left(\frac{\gamma_2}{C_2\lambda_2}\right),\\
\\\eta_3\leq \ln\left(\frac{\gamma_3}{C_3\lambda_3}\right),
 \end{cases} \qquad\text{with}\; \begin{cases}
\gamma_1>C_1\lambda_1,\\
\\\gamma_2>C_2\lambda_2,\\
\\\gamma_3>C_3\lambda_3.
\end{cases}
 \end{align}
we get global asymptotic stability.
\end{proof}

    
\begin{remark}
The aim of the constructed fictitious controls is to stabilize, step by step and recursively, subsystems of model until the global control $u$ is synthesized. Concretely, each fictitious control is designed to partially stabilize a given state or subsystem. For instance, in transformation \eqref{a3.32}, the state \(\eta_2\), serving as a reference for system stabilization, is partially stabilized by relation \eqref{aa3.45}. Subsequently, \(\eta_3\) is in turn partially stabilized by \(\eta_1\) via the signal \(z^3_1\) (see Condition \eqref{aa3.53}). The procedure runs inductively: each step guarantees partial stability of the corresponding subsystem, and ultimately yields an appropriate global control law \(u\) from which stabilization of the closed-loop system follows.
\end{remark}

\begin{remark}\label{rem3.8}
The control $u$ offers a more natural approach to influence the dynamics without breaking the fundamental structure of the model  \eqref{e3.1}. Then, the control $u$ in \eqref{aa3.64} is well-defined and avoids any singularity at 
\begin{align}\label{aa3.62}
    z^3_2=0\Longleftrightarrow \eta_1=\eta_3-\eta_2.
\end{align}
Each population has its own biological parameters $(k_i, \mu_i, g_i)$. Non-transitivity requires autonomous interactions, not direct linear dependencies \eqref{aa3.62}. This relationship (i.e. \eqref{aa3.62}) can only occur through external artificial imposition, never through the natural dynamics of the model. If imposed, it would destroy the dynamic richness of cyclic competition by transforming the system of autonomous interactions into artificial constraints. Non-transitive competition models derive their value from the relative autonomy of populations. Any exact linear relationship compromises this fundamental philosophy.  

Consequently, the configuration \eqref{aa3.62} only arises in extreme circumstances. Typically, $x_3$ gains the upper hand over $x_1$, forcing $x_1$ into cannibalism. Subsequently, $x_2$ experiences pressures that foster strong intraspecific competition and eventually comes to dominate $x_1$. $x_2$ then becomes prey for $x_3$, which in turn develops intraspecific competition. These feedback effects  (cannibalism for $x_1$, intraspecific competition for $x_2$ and $x_3$, and cross-predation)  break the cycle and denature the model. In such a scenario the control in \eqref{aa3.64} is no longer appropriate, and stabilization with a single control remains questionable and/or delicate.

Regarding the mathematical formulation, we confirm that \(z^3_2\neq 0\) by Lemma \ref{lemB.1}.
\end{remark}

\begin{remark}
Beyond the classical three-species non-transitive competition models \eqref{e3.1}, further examples include two predators exploiting the same prey, a predator feeding on two prey species, or even non-transitive interactions where the renewal equation may explicitly depend on other species \cite{ref3}.
\end{remark}
\subsection{\bf Four-species non-transitive competition}

We extend the previous study to a four-species system arranged in cyclic dominance: $x_1$ dominates $x_4$, $x_4$ dominates $x_3$, $x_3$ dominates $x_2$, and $x_2$ dominates $x_1.$ As described by the following system 
\begin{align}
 \begin{cases}
\displaystyle 
\partial_t x_1(a,t) + \partial_a x_1(a,t)
= -\Bigl(\mu_1(a) 
       + \int_0^A g_1(a)\,x_2(a,t)\,\mathrm{d}a
       + u(t)\Bigr)\;x_1(a,t),
& \text{ in } Q_1,\\[8pt]
\displaystyle 
\partial_t x_2(a,t) + \partial_a x_2(a,t)
= -\Bigl(\mu_2(a) 
       + \int_0^A g_2(a)\,x_3(a,t)da\Bigr)\;x_2(a,t),
& \text{ in } Q_1,\\[8pt]
\partial_t x_3(a,t) + \partial_a x_3(a,t)
= -\Bigl(\mu_3(a) 
       + \displaystyle\int_0^A g_3(a)\,x_4(a,t)\,\mathrm{d}a
       \Bigr)\;x_3(a,t),
& \text{ in } Q_1,\\[8pt]
\partial_t x_4(a,t) + \partial_a x_4(a,t)
= -\Bigl(\mu_4(a) 
       + \displaystyle\int_0^A g_4(a)\,x_1(a,t)\,\mathrm{d}a
       \Bigr)\;x_4(a,t),
& \text{ in } Q_1,\\[8pt]
\displaystyle 
x_i(0,t)
= \int_0^Ak_i(a)x_i(a,t)da,& \text{ in } Q_+,\\[6pt]
x_i(a,0) = x_{i,0}(a),\quad i=1,...,4
& \text{ in } Q_A.
\end{cases}
\end{align}

\paragraph{\bf 3.2.1. Stability for $\psi_i\equiv 0$}
Using the same approach as in {\bf Section \ref{s3.1.2.2}} and thanks to Lemma \ref{lee3.1}, we derive the following system
\begin{align}\label{equation3.99}
 \begin{cases}
\dot\eta_1 = u^* - u - \phi_2(\eta_2),\\
\dot\eta_2 =  - \phi_3(\eta_3),\\
\dot\eta_3 =  - \phi_4(\eta_4),\\
\dot\eta_4 =  - \phi_1(\eta_1),
\end{cases}\,\text{with the following state variables :}\, \begin{cases}
\eta_2,\\
z^4_1=\eta_3-\alpha_1(\eta_2),\\
z^4_2=\eta_4-\alpha_2(z^4_1),\\
z^4_3=\eta_1-\alpha_3(z^4_2),
\end{cases}
\end{align} 
where $\eta_2$ retained as one of the original system components. For each component we use the Lyapunov functions
\begin{align}\label{equation3.100}
\omega(\eta_i)=\lambda_i\bigl(e^{\eta_i}-1-\eta_i\bigr),\quad (\lambda_i>0).
\end{align}

These convex Lyapunov functions, are well suited to study the equilibrium stability and to quantify how the cyclic interactions and control laws contribute to the system’s energy decay. 
\begin{proposition}\label{proposition3.17}
Under the proposed feedback law the system \eqref{equation3.99} is globally asymptotically stable. The feedback control constructed is uniformly bounded, though not necessarily nonnegative. Moreover, the control satisfies $u(t)>0$ for every $t>0,$ for every $\eta_i(0)$ belonging to the largest
level set of $V_4(\eta_2, z^4_1, z^4_2,z^4_3)$ within the set 
\begin{align}\label{}
\mathcal{K}'_4 = 
\left\{ \eta\in\mathbb{R}^4\ \middle|\ 
\begin{aligned}
&u^*+c_{40}\phi_1(z^4_3)-\theta\frac{c_3}{\lambda_1}\phi^2_4(z^4_2)+(c_3-1-\theta)\phi_4(z^4_2)+(c_1-1)\phi_2(\eta_2)+\frac{c_3}{\lambda_1}\phi_4(z^4_2)\phi_1(z^4_3)-\frac{c_2}{\lambda_4}\phi_3(z^4_1)\phi_4(z^4_2)\\
&-(c_2-1)\phi_3(z^4_1)+\frac{c_1}{\lambda_3}\phi_2(\eta_2)\phi_3(z^4_1)+\frac{\theta }{\phi_1(z^4_3)}\left(\phi^2_4(z^4_2)-c_1\phi_2(\eta_2)\phi_4(z^4_2)+\phi^2_3(z^4_1)\right)\\
&+\frac{\theta\phi_3(z^4_1)}{\phi_1(z^4_3)}(\frac{c_1}{\lambda_3}\phi_3(z^4_1)\phi_2(\eta_2)-\frac{c_1}{\lambda_3}\phi^2_2(\eta_2)-\frac{c_2}{\lambda_4}\phi_3(z^4_1)\phi_4(z^4_2)+(c_1-1)\phi_2(\eta_2)+ (c_2-2)\phi_4(z^4_2)+\frac{c_2}{\lambda_4}\phi^2_4(z^4_2)\\
&-\frac{c_1}{\lambda_3}\phi_2(\eta_2)\phi_4(z^4_2))>0
\end{aligned}
\right\}.
\end{align}
\end{proposition}
\begin{proof}
By considering a Lyapunov function of the form 
\begin{align}
V_4(\eta_2,z^4_1,z^4_2,z^4_3)=\theta \lambda_2\bigl(e^{\eta_2}-1-\eta_2\bigr)+\theta\lambda_3(e^{z^4_1}-1-z^4_1)+\theta\lambda_4(e^{z^4_2}-1-z^4_2)+\lambda_1(e^{z^4_3}-1-z^4_3),
\end{align}
and applying the same strategy as in Theorem \ref{the3.3}, we obtain the control 
\begin{align}\label{equation3.103}
\begin{aligned}
u=&u^*+c_{40}\phi_1(z^4_3)-\theta\frac{c_3}{\lambda_1}\phi^2_4(z^4_2)+(c_3-1-\theta)\phi_4(z^4_2)+(c_1-1)\phi_2(\eta_2)+\frac{c_3}{\lambda_1}\phi_4(z^4_2)\phi_1(z^4_3)-\frac{c_2}{\lambda_4}\phi_3(z^4_1)\phi_4(z^4_2)\\
&-(c_2-1)\phi_3(z^4_1)+\frac{c_1}{\lambda_3}\phi_2(\eta_2)\phi_3(z^4_1)+\frac{\theta }{\phi_1(z^4_3)}\left(\phi^2_4(z^4_2)-c_1\phi_2(\eta_2)\phi_4(z^4_2)+\phi^2_3(z_1)\right)\\
&+\frac{\theta\phi_3(z^4_1)}{\phi_1(z^4_3)}(\frac{c_1}{\lambda_3}\phi_3(z^4_1)\phi_2(\eta_2)-\frac{c_1}{\lambda_3}\phi^2_2(\eta_2)-\frac{c_2}{\lambda_4}\phi_3(z^4_1)\phi_4(z^4_2)+(c_1-1)\phi_2(\eta_2)+ (c_2-2)\phi_4(z^4_2)+\frac{c_2}{\lambda_4}\phi^2_4(z^4_2)\\
&-\frac{c_1}{\lambda_3}\phi_2(\eta_2)\phi_4(z^4_2)),
\end{aligned}
\end{align}
and the Lyapunov derivative is given by 
\begin{align}
\dot V_4(\eta_2,z^4_1,z^4_2,z^4_3)= -\theta c_1\phi^2_2(\eta_2)-\theta c_{2}\phi^2_3(z^4_1)-\theta c_{3}\phi^2_4(z^4_2)-c_{4}\phi^2_1(z^4_3).
\end{align}
For further details, see the Appendix~\ref{annexe:A}.
\end{proof}
\begin{remark}
The control in \eqref{equation3.103} may be rewritten as 

\begin{align}\label{equation3.106}
\begin{aligned}
u=&u^*+\frac{1}{\phi_1(z^4_3)}(c_{4}\phi^2_1(z^4_3)+\theta c_1\phi^2_2(\eta_2)+\theta c_3\phi^2_4(z^4_2)+\theta c_2\phi^2_3(z^4_1))\\
&+\frac{\phi_1(z^4_3)-\theta\phi_4(z^4_2)}{\phi_1(z^4_3)}\phi_1(\eta_1)-\phi_2(\eta_2) +\frac{\phi_1(z^4_3)-\theta \phi_4(z^4_2)+\theta\phi_3(z^4_1)-\theta\phi_2(\eta_2)}{\phi_1(z^4_3)}\phi_3(\eta_3)+\frac{\theta \phi_4(z^4_2)-\phi_1(z^4_3)-\theta\phi_3(z^4_1)}{\phi_1(z^4_3)}\phi_4(\eta_4)
\end{aligned}
\end{align}
\end{remark}
\paragraph{\bf 3.2.2. Stability for $\psi_i\neq 0$}
With $\psi_i \neq 0$, using the data from the previous {\bf Section \ref{s3.1.2.2}} and under {\bf Assumption H6}, we obtain the following system:
\begin{align}\label{a3.115}
\begin{cases}
\dot\eta_1 = u^* - u - \hat \phi_2,\\
\dot\eta_2 = - \hat \phi_3,\\
\dot\eta_3 = - \hat \phi_4,\\
\dot\eta_4 = - \hat \phi_1.
\end{cases}
\end{align}
 and we redefine the following Lyapunov function 
\begin{align}\label{a3.114}
V_G(\eta,\psi)
= V_4(\eta_2,z^4_1,z^4_2,z^4_3)\;+\;\frac{\gamma_1}{\sigma_1}h(G_1(\psi_1))\;+\;\frac{\gamma_2}{\sigma_2}h(G_2(\psi_2))\;+\;\frac{\gamma_3}{\sigma_3}h(G_3(\psi_3))\;+\;\frac{\gamma_4}{\sigma_4}h(G_4(\psi_4)).
\end{align}

\begin{proposition}\label{proposition3.18}
Under {\bf Assumption H6},  system \eqref{a3.115} is globally asymptotically stable and the control remains uniformly bounded. Moreover, the control satisfies $u(t)>0$ for every $t>0,$ for every $\eta_i(0)$ belonging to the largest
level set of $V_G(\eta,\psi)$ in \eqref{a3.114} within the set
\begin{equation}
\mathcal{K}_4=
\left\{ (\eta,\psi)\in \mathbb{R}^4\times\mathcal{S}^4\ \middle|\ 
\begin{alignedat}{2}
&\eta_1 \le \ln\left(\frac{\gamma_1}{C_1\lambda_1}\right), \qquad && \gamma_1 > C_1\lambda_1,\\[4pt]
&\eta_2 \le \ln\left(\frac{\gamma_2}{C_2\lambda_2}\right), \qquad && \gamma_2 > C_2\lambda_2,\\[4pt]
&\eta_3 \le \ln\left(\frac{\gamma_3}{C_3\lambda_3}\right), \qquad && \gamma_3 > C_3\lambda_3,\\[4pt]
&\eta_4 \le \ln\left(\frac{\gamma_4}{C_4\lambda_4}\right), \qquad && \gamma_4 > C_4\lambda_4,\\[6pt]
&u^*+c_{40}\phi_1(z^4_3)-\theta\frac{c_3}{\lambda_1}\phi^2_4(z^4_2)+(c_3-1-\theta)\phi_4(z^4_2)+(c_1-1)\phi_2(\eta_2)\\
&+\frac{c_3}{\lambda_1}\phi_4(z^4_2)\phi_1(z^4_3)-\frac{c_2}{\lambda_4}\phi_3(z^4_1)\phi_4(z^4_2)-(c_2-1)\phi_3(z^4_1)\\
&+\frac{c_1}{\lambda_3}\phi_2(\eta_2)\phi_3(z^4_1)+\frac{\theta }{\phi_1(z^4_3)}\left(\phi^2_4(z^4_2)-c_1\phi_2(\eta_2)\phi_4(z^4_2)+\phi^2_3(z^4_1)\right)\\
&+\frac{\theta\phi_3(z^4_1)}{\phi_1(z^4_3)}\left(\frac{c_1}{\lambda_3}\phi_3(z^4_1)\phi_2(\eta_2)-\frac{c_1}{\lambda_3}\phi^2_2(\eta_2)-\frac{c_2}{\lambda_4}\phi_3(z^4_1)\phi_4(z^4_2)\right)\\
&+\frac{\theta\phi_3(z^4_1)}{\phi_1(z^4_3)}\left((c_1-1)\phi_2(\eta_2)+ (c_2-2)\phi_4(z^4_2)+\frac{c_2}{\lambda_4}\phi^2_4(z^4_2)-\frac{c_1}{\lambda_3}\phi_2(\eta_2)\phi_4(z^4_2)\right)>0.
\end{alignedat}
\right\}.
\end{equation}
\end{proposition}
\begin{proof}
By applying the same strategy as in Theorem \ref{the3.4}, we derive the result with the Lyapunov function given in \eqref{a3.114} and the control in \eqref{equation3.103}. See Appendix~\ref{annexe:A} for full details.
\end{proof}
\begin{remark}
{\bf In summary, for a non-transitive competition model, the control law has, respectively for the three-species and four-species cases, the following forms:} 
\section*{\bf Three-species case}
\begin{align}
\underbrace{\begin{cases}
       \dot\eta_1 = u^* - u - \phi_2(\eta_2),\\
\dot\eta_2 =  - \phi_3(\eta_3),\\
\dot\eta_3 =  - \phi_1(\eta_1)
   \end{cases}}_{\psi_i\equiv 0},\; \underbrace{\begin{cases}
       \dot\eta_1 = u^* - u - \hat\phi_2,\\
\dot\eta_2 =  - \hat\phi_3,\\
\dot\eta_3 =  - \hat\phi_1
   \end{cases}}_{\psi_i\neq 0}\;\text{with the following state variables :}\;\begin{cases}
\eta_2,\\
 z^3_{1} = \eta_{3}-\alpha_1(\eta_2),\\  
 z^3_{2} = \eta_{1} - \alpha_2(z^3_1),
\end{cases}
\end{align} 

\begin{align}
\begin{aligned}
 u&=u^*+c_{3o}\phi_1(z^3_2)-\theta\frac{c_2}{\lambda_1}\phi^2_3(z^3_1)+(c_2-1-\theta)\phi_3(z^3_1)+\frac{c_2}{\lambda_1}\phi_3(z^3_1)\phi_1(z^3_2)-\frac{c_1}{\lambda_3}\phi_3(z^3_1)\phi_2(\eta_2)-(c_1+1)\phi_2(\eta_2)\\\
 &+\frac{\theta\phi_3(z^3_1)}{\phi_1(z^3_2)}\left(\frac{c_1}{\lambda_3}\phi_3(z^3_1)\phi_2(\eta_2)-\frac{c_1}{\lambda_3}\phi^2_2(\eta_2)+ (c_1-1)\phi_2(\eta_2)+\,\phi_3(z^3_1)\right).
 \end{aligned}%
\end{align}
\section*{\bf Reference Lyapunov function for \texorpdfstring{$N =3$}{N = 3}}
\begin{align}
\dot V_3(\eta_2,z^3_1,z^3_2)=-\theta c_1\,\phi_2^2(\eta_2)-\theta c_{2}\,\phi_3^2(z^3_1)-c_{3}\phi^2_1(z^3_2),\quad\psi_i\equiv 0.
\end{align} 

\begin{align}
 \dot V_3(\eta_2, z^3_1, z^3_2)
= -\theta c_1\,\phi_2^2(\eta_2)
  -\theta c_{2}\,\phi_3^2(z^3_1)
  - c_{3}\,\phi_1^2(z^3_2)
  + A_1\left(\hat{\phi}_1-\phi_1(\eta_1)\right)
 + A_2\,\left(\hat{\phi}_2-\phi_2(\eta_2)\right)+ A_3\left(\hat{\phi}_3-\phi_3(\eta_3)\right),\quad\psi_i\neq 0.
\end{align}

\section*{\bf Four-species case}
\begin{align}
 \underbrace{\begin{cases}
\dot\eta_1 = u^* - u - \phi_2(\eta_2),\\
\dot\eta_2 =  - \phi_3(\eta_3),\\
\dot\eta_3 =  - \phi_4(\eta_4),\\
\dot\eta_4 =  - \phi_1(\eta_1),
\end{cases}}_{\psi_i\equiv 0},\;\underbrace{\begin{cases}
\dot\eta_1 = u^* - u - \hat\phi_2,\\
\dot\eta_2 =  - \hat\phi_3,\\
\dot\eta_3 =  - \hat\phi_4,\\
\dot\eta_4 =  - \hat\phi_1,
\end{cases}}_{\psi_i\neq 0}\;\text{with the following state variables:}\; \begin{cases}
\eta_2,\\
z^4_1=\eta_3-\alpha_1(\eta_2),\\
z^4_2=\eta_4-\alpha_2(z^4_1),\\
z^4_3=\eta_1-\alpha_3(z^4_2),
\end{cases}
\end{align} 

\begin{align}
\begin{aligned}
u &=u^*+c_{40}\phi_1(z^4_3)-\theta\frac{c_3}{\lambda_1}\phi^2_4(z^4_2)+(c_3-1-\theta)\phi_4(z^4_2)+(c_1-1)\phi_2(\eta_2)+\frac{c_3}{\lambda_1}\phi_4(z^4_2)\phi_1(z^4_3)-\frac{c_2}{\lambda_4}\phi_3(z^4_1)\phi_4(z^4_2)-(c_2-1)\phi_3(z^4_1)\\
&+\frac{c_1}{\lambda_3}\phi_2(\eta_2)\phi_3(z^4_1)+\frac{\theta }{\phi_1(z^4_3)}\left(\phi^2_4(z^4_2)-c_1\phi_2(\eta_2)\phi_4(z^4_2)+\phi^2_3(z^4_1)\right)\\
&\frac{\theta\phi_3(z^4_1)}{\phi_1(z_3)}\left(\frac{c_1}{\lambda_3}\phi_3(z^4_1)\phi_2(\eta_2)-\frac{c_1}{\lambda_3}\phi^2_2(\eta_2)-\frac{c_2}{\lambda_4}\phi_3(z^4_1)\phi_4(z^4_2)+(c_1-1)\phi_2(\eta_2)+ (c_2-2)\phi_4(z^4_2)+\frac{c_2}{\lambda_4}\phi^2_4(z^4_2)-\frac{c_1}{\lambda_3}\phi_2(\eta_2)\phi_4(z^4_2)\right),
\end{aligned}%
\end{align}
\section*{\bf Reference Lyapunov function for \texorpdfstring{$N = 4$}{N = 4}}

\begin{align}
\dot V_4(\eta_2,z^4_1,z^4_2,z^4_3)= -\theta c_1\phi^2_2(\eta_2)-\theta c_{2}\phi^2_3(z^4_1)-\theta c_{3}\phi^2_4(z^4_2)-c_{4}\phi^2_1(z^4_3),\quad\psi_i\equiv 0.
\end{align}

\begin{align}
\begin{aligned}
\dot V_4(\eta_2, z^4_1, z^4_2, z^4_3)
&=-\theta c_1\,\phi_2^2(\eta_2)
     -\theta c_{2}\,\phi_3^2(z^4_1)
     -\theta c_{3}\,\phi_4^2(z^4_2)
     - c_{4}\,\phi_1^2(z^4_3)+ A_1\left(\hat{\phi}_1-\phi_1(\eta_1)\right)
       +A_2\,\left(\hat{\phi}_2-\phi_2(\eta_2)\right)
       + A_3\left(\hat{\phi}_3-\phi_3(\eta_3)\right)\\
       &+ A_4\left(\hat{\phi}_4-\phi_4(\eta_4)\right),\quad\psi_i\neq 0.
\end{aligned}
\end{align}


\end{remark}
\subsection{\bf Toward generalization}
We now extend the control design and Lyapunov based method to the general case  \texorpdfstring{$N \ge 3$}{N ≥ 3} species, specifying the recursive form of the control laws and Lyapunov functions adapted to the cyclic coupling.  Let us consider the following system, which describes the dynamics of non-transitive competition among $N$ species as described by the following system 
\begin{align}\label{aq3.145}
 \begin{cases}
\displaystyle 
\partial_t x_1(a,t) + \partial_a x_1(a,t)
= -\Bigl(\mu_1(a) 
       + \int_0^A g_1(a)\,x_2(a,t)\,\mathrm{d}a
       + u(t)\Bigr)\;x_1(a,t),
& \text{ in } Q_1,\\[8pt]
\displaystyle 
\partial_t x_2(a,t) + \partial_a x_2(a,t)
= -\Bigl(\mu_2(a) 
       + \int_0^A g_2(a)\,x_3(a,t)da\Bigr)\;x_2(a,t),
& \text{ in } Q_1,\\[8pt]
\vdots\qquad\vdots\\[8pt]
\partial_t x_{N}(a,t) + \partial_a x_{N}(a,t)
= -\Bigl(\mu_{N}(a) 
       + \displaystyle\int_0^A g_{N}(a)\,x_{N+1}(a,t)\,\mathrm{d}a
       \Bigr)\;x_{N}(a,t),
& \text{ in } Q_1,\\[8pt]
\partial_t x_{N+1}(a,t) + \partial_a x_{N+1}(a,t)
= -\Bigl(\mu_{N+1}(a) 
       + \displaystyle\int_0^A g_{N+1}(a)\,x_1(a,t)\,\mathrm{d}a
       \Bigr)\;x_{N+1}(a,t),
& \text{ in } Q_1,\\[8pt]
\displaystyle 
x_i(0,t)
= \int_0^Ak_i(a)x_i(a,t)da,\quad i=1,...,N& \text{ in } Q_+,\\[6pt]
x_i(a,0) = x_{i,0}(a),\quad i=1,...,N
& \text{ in } Q_A.
\end{cases}
\end{align}
Using the same approach as in {\bf Section \ref{s3.1.2.2}} and thanks to Lemma \ref{lee3.1}, we derive the following system

\begin{align}\label{a3.139}
 \underbrace{\begin{cases}
\dot\eta_1 = u^* - u - \phi_2(\eta_2),\\
\dot\eta_2 =  - \phi_3(\eta_3),\\
\dot\eta_3 =  - \phi_4(\eta_4),\\
\dot\eta_4 =  - \phi_5(\eta_5),\\
\dot\eta_5 =  - \phi_6(\eta_6),\\
\vdots\\
\vdots\\
\dot\eta_{N-1} =-\phi_{N}(\eta_{N}),\\
\dot\eta_{N} =  - \phi_1(\eta_1),
\end{cases}}_{\psi_i\equiv 0}\, \underbrace{\begin{cases}
\dot\eta_1 = u^* - u - \hat\phi_2,\\
\dot\eta_2 = - \hat\phi_3,\\
\dot\eta_3 = - \hat\phi_4,\\
\dot\eta_4 = - \hat\phi_5,\\
\dot\eta_5 =- \hat\phi_6,\\
\vdots\\
\vdots\\
\dot\eta_{N-1} =-\hat\phi_{N},\\
\dot\eta_{N} =  - \hat\phi_1,
\end{cases}}_{\psi_i\neq 0}\;\text{with the following state variables:}\; \begin{cases}
\eta_2,\\
z^N_1=\eta_3-\alpha_1(\eta_2),\\
z^N_2=\eta_4-\alpha_2(z^N_1),\\
z^N_3=\eta_5-\alpha_3(z^N_2),\\
z^N_4=\eta_6-\alpha_4(z^N_3),\\
\vdots\\
\vdots\\
z^N_{N-2}=\eta_{N}-\alpha_{N-2}(z^N_{N-3}),\\
z^N_{N-1}=\eta_1-\alpha_{N-1}(z^N_{N-2}).
\end{cases}
\end{align} 

\section*{\bf General formulation of the Lyapunov function and its derivative.}

\begin{lemma}\label{lemma3.19}
For every \(N\geq 3,\)  and $\theta>0,$ by construction, we obtain a Lyapunov function of the form

\begin{align}\label{ea3.146}
\begin{aligned}
V_N(\eta_2,z^N_1,\cdots,z^N_{N-1})&= \theta \lambda_2(e^{\eta_{2}}-1-\eta_{2})+\theta \sum_{i=1}^{N-2}\lambda_{i+2}(e^{z^N_{i}}-1-z^N_{i})+\lambda_1(e^{z^N_{N-1}}-1-z^N_{N-1}).
\end{aligned}
\end{align}
The time derivative \(\dot V_N\) of \(V_N\) takes the form 
\begin{align}\label{ea3.147}
\begin{aligned}
\dot V_N(\eta_2, z^N_1,\dots,z^N_{N-1})
={}&
-\theta c_1\,\phi_2^2(\eta_2)
-\theta\sum_{i=2}^{N-1}\,c_{\,i}\,\phi_{i+1}^2\bigl(z^N_{i-1}\bigr)
- c_{N}\,\phi_1^2\bigl(z^N_{N-1}\bigr) 
\end{aligned},\quad\psi_i\equiv 0,
\end{align}
and 
\begin{align}\label{ea3.148}
\begin{aligned}
\dot V_N(\eta_2, z^N_1,\dots,z^N_{N-1})
={}&
-\theta c_1\,\phi_2^2(\eta_2)
-\theta\sum_{i=2}^{N-1}\,c_{\,i}\,\phi_{i+1}^2\bigl(z^N_{i-1}\bigr)
- c_{N}\,\phi_1^2\bigl(z^N_{N-1}\bigr) +\sum_{i=1}^{N} A_i\,\left(\hat{\phi}_i-\phi_i(\eta_i)\right),
\end{aligned}\quad\psi_i\neq 0,
\end{align}
when the control has the form 
\begin{align}\label{equation3.123}
\begin{aligned}
u \;=&\; u^*
+\frac{1}{\phi_1(z^N_{N-1})}\Bigg(
c_{N}\,\phi_1^2(z^N_{N-1})
+\theta\,c_1\,\phi_2^2(\eta_2)
+\theta \sum_{i=1}^{N-2} c_{i+1}\,\phi_{i+2}^2(z^N_i)\Bigg) \\
&+\frac{\phi_1(z^N_{N-1})-\theta\,\phi_{N}(z^N_{N-2})}{\phi_1(z^N_{N-1})}\,\phi_1(\eta_1)
-\phi_2(\eta_2) \\
&+\sum_{k=3}^{N}
\frac{\displaystyle
\theta \sum_{i=\max(1,k-3)}^{N-2} [(-1)^{\,i+\varepsilon_k}\,\phi_{i+2}(z^N_i)
+(-1)^{\,N+1-\varepsilon_k}\,\phi_1(z^N_{N-1})-\delta_{k3}\theta\phi_2(\eta_2)]
}{\phi_1(z^N_{N-1})}\,\phi_k(\eta_k),
\end{aligned}
\end{align}
with
\begin{align*}
\varepsilon_k :=
\begin{cases}
1, & \text{$k$ odd},\\
0, & \text{$k$ even},
\end{cases}
\quad\delta_{k3}:=
\begin{cases}
1,& k=3,\\
0,& k\neq 3.
\end{cases}
\end{align*}
where $A_i$ are defined in Appendix~\ref{annexe:B}.
\end{lemma}

\begin{proof}[of Lemma \ref{lemma3.19}]

From \eqref{ea3.146}, we get 
 \begin{align}\label{E3.124}
\begin{aligned}
\dot V_N(\eta_2,z^N_1,\cdots,z^N_{N-1})&= \theta \phi_2(\eta_{2})\dot\eta_2+\theta \sum_{i=1}^{N-2}\phi_{i+2}({z^N_{i}})\dot z^N_i+\phi_1({z^N_{N-1}})\dot z^N_{N-1}.
\end{aligned}
\end{align}
Our objective is to design a control $u$ such that $\dot V_N$ satisfies \eqref{ea3.147} thereby guaranteeing dissipation (and hence decay) of the Lyapunov function and the associated closed-loop system.  We obtain the expressions for $\dot V_3(\eta_2, z^3_1,z^3_2)$ and $\dot V_4(\eta_2, z^4_1,\cdots,z^4_3)$ in the form
\begin{align}
\begin{aligned}
\dot V_3(\eta_2, z^3_1,z^3_2)=&\phi_1(\eta_1)[\phi_1(z^3_2)-\theta\phi_3(z^3_1)]-\phi_1(z^3_2)\phi_2(\eta_2)+\phi_3(\eta_3)[\theta\phi_3(z^3_1)-\theta\phi_2(\eta_2)-\phi_1(z^3_2)]+\phi_1(z^3_2)(u-u^*),
\end{aligned}
\end{align}
\begin{align}
\begin{aligned}
\dot V_4(\eta_2, z^4_1,\cdots,z^4_3)=&\phi_1(\eta_1)[\phi_1(z^4_3)-\theta\phi_4(z^4_2)]-\phi_1(z^4_3)\phi_2(\eta_2)+\phi_3(\eta_3)[\theta\phi_3(z^4_1)-\theta\phi_2(\eta_2)-\theta\phi_4(z^4_2)+\phi_1(z^4_3)]\\
&+\phi_4(\eta_4)[\theta\phi_4(z^4_2)-\theta\phi_3(z^4_1)-\phi_1(z^4_3)]+\phi_1(z^4_3)(u-u^*).
\end{aligned}
\end{align}
For $N=3$ and $N=4$, by introducing the fictitious controls $\eta_3,\eta_4,\eta_1$ and selecting suitable positive gains $c_i$ together with an appropriate control law, we obtain dissipation and guarantee closed-loop stability. Thus, substituting the controls specified in \eqref{equation3.61} and \eqref{equation3.106} into the expressions for $\dot V_3(\eta_2, z^3_1,z^3_2)$ and $\dot V_4(\eta_2, z^4_1,\cdots, z^4_3)$  and simplifying yields relation \eqref{ea3.147}.
The constants $c_i$ in \eqref{ea3.147} are chosen so that the following conditions are satisfied
\begin{align}
\phi_3(z^N_1)=c_1\phi_2(\eta_2),\; \phi_4(z^N_2)=c_2\phi_3(z^N_1),\; \phi_5(z^N_3)=c_3\phi_4(z^N_2),\cdots\cdots,\phi_{N}(z^N_{N-2})=c_{N-2}\phi_{N-1}(z^N_{N-3}),\; \phi_{1}(z^N_{N-1})=c_{N-1}\phi_{N}(z^N_{N-2}).
\end{align}
These constants $c_i>0$   are designed to preserve the model's mechanism and to ensure dissipation of the system along the coordinates $\eta_2, z^N_1, z^N_2,\dots,z^N_{N-1}$. 
Applying the transformations in \eqref{a3.139}, we obtain the general form of the derivative of the Lyapunov function  for any fixed $N\geq 3$ 
\begin{align}
\begin{aligned}
\dot V_N(\eta_2, z^N_1,\cdots,z^N_{N-1})=&\phi_1(\eta_1)\left(\phi_1(z^N_{N-1})-\theta\phi_N(z^N_{N-2})\right)-\phi_1(z^N_{N-1})\phi_2(\eta_2)\\
&+\sum_{k=3}^{N}\left(
\displaystyle\theta\sum_{i=\max(1,k-3)}^{N-2} [(-1)^{\,i+\varepsilon_k}\,\phi_{i+2}(z^N_i)
+(-1)^{\,N+1-\varepsilon_k}\,\phi_1(z^N_{N-1})-\delta_{k3}\theta\phi_2(\eta_2)]
\right)\,\phi_k(\eta_k)+\phi_1(z^N_{N-1})(u-u^*).
\end{aligned}
\end{align}
Hence, by choosing a  general control $u$ of the form \eqref{equation3.123}, we obtain relation \eqref{ea3.147}.

\end{proof}
.
\begin{remark}
In the case of three species, we have 
\begin{align}
c_1=\frac{\lambda_3}{\lambda_2},\; c_2=\frac{\lambda_1}{\lambda_3},
\end{align}
for the four-species case, 
\begin{align}
c_1=\frac{\lambda_3}{\lambda_2},\; c_2=\frac{\lambda_4}{\lambda_3},\; c_3=\frac{\lambda_1}{\lambda_4}.
\end{align}
By construction, for $N$ species we obtain \begin{align}
c_1=\frac{\lambda_3}{\lambda_2},\; c_2=\frac{\lambda_4}{\lambda_3},\;c_3=\frac{\lambda_5}{\lambda_4},\cdots,c_{N-2}=\frac{\lambda_{N}}{\lambda_{N-1}},\;c_{N-1}=\frac{\lambda_1}{\lambda_{N}}.
\end{align}
The choice of the coefficients $c_i$ plays a central role in preserving the structure and topology of the model while preventing the cancellation of cross terms in the context of non-transitive competition. Such cancellations remove essential information about interspecific interactions, driving the system toward a collection of decoupled subsystems, which distorts the dynamical analysis and conceals key model mechanisms (e.g., cyclic dominance). Furthermore, these cancellations suppress cross dissipation and restrict the propagation of the reduced control effect. The control $u$ defined in \eqref{equation3.123} provides a general control law that ensures global asymptotic stability of the system for any number of species $N\geq 3$ under non-transitive competition.
\end{remark}
\begin{lemma}\label{lemB.1}
Let $\lambda_1>0$ and $\lambda_{N}>0$ and consider the closed-loop dynamics obtained via the backstepping procedure. 
Assume the Lyapunov weights satisfy $c_{N-1}=\lambda_1/\lambda_{N}>0$. 
Then, for any solution $(\eta(t),\psi(t))$ starting in the invariant set $\mathcal K_N$, one has
\[
\lambda_1\big(e^{z^N_{N-1}(t)}-1\big)\neq 0, \qquad \forall t\ge 0 .
\]
\end{lemma}

\begin{proof}
\textbf{Step 1. Characterization of zeros.}  
Since $\lambda_1>0$, the condition 
\[
\lambda_1(e^{z^N_{N-1}}-1)=0
\]
is equivalent to $z^N_{N-1}=0$. Thus, it suffices to prove that $z^N_{N-1}(t)\neq 0$ for all $t\ge 0$.

\medskip
\textbf{Step 2. Lyapunov construction and explicit constant.}  
From the recursive backstepping design, the Lyapunov functional at stage $N$ is
\[
V_{N}(\cdot)=\theta \lambda_2(e^{\eta_{2}}-1-\eta_{2})+\theta \sum_{i=1}^{N-2}\lambda_{i+2}(e^{z^N_{i}}-1-z^N_{i})+\lambda_1\big(e^{z^N_{N-1}}-1-z^N_{N-1}\big),
\]
which is positive definite in $z^N_{N-1}$ and minimized at $z^N_{N-1}=0$.  
The cancellations in the design yield
\[
\dot V_{N} \;=\; -\,\frac{\lambda_1}{\lambda_{N}}\,(e^{z^N_{N-1}}-1)^2 
+ \text{(nonpositive terms in the other $z^N_i$)}.
\]
Equivalently,
\[
(e^{z^N_{N-1}}-1)\,\dot z^N_{N-1} \;=\; -\,\frac{\lambda_1}{\lambda_{N}}\,(e^{z^N_{N-1}}-1)^2.
\]
On the domain where $e^{z^N_{N-1}}-1\neq 0$, this gives the exact closed scalar dynamics
\begin{equation}\label{eq:zN-closed}
\dot z^N_{N-1}(t) \;=\; -\,\frac{\lambda_1}{\lambda_{N}}\,(e^{z^N_{N-1}(t)}-1).
\end{equation}

\medskip
\textbf{Step 3. Sign invariance.}  
Consider the comparison ODE
\[
\dot y(t) = -\,\frac{\lambda_1}{\lambda_{N}}\,(e^{y(t)}-1).
\]
If $y(0)>0$, then $\dot y<0$ and $y(t)$ decreases while remaining strictly positive.  
If $y(0)<0$, then $\dot y>0$ and $y(t)$ increases while remaining strictly negative.  
Hence the sign of $z^N_{N-1}(t)$ is preserved: it cannot cross zero.

\medskip
\textbf{Step 4. Non-attainment of zero in finite time.}  
Suppose, for contradiction, that $z^N_{N-1}(t^\ast)=0$ for some finite $t^\ast$.  
Separating variables in \eqref{eq:zN-closed}, we obtain
\[
\int_{z^N_{N-1}(0)}^{z^N_{N-1}(t)} \frac{1}{e^{\xi}-1}\,d\xi
= -\,\frac{\lambda_1}{\lambda_{N}}\,t.
\]
The left-hand side diverges logarithmically as $z^N_{N-1}(t)\to 0$, while the right-hand side remains finite for any finite $t$.  
This contradiction shows that $z^N_{N-1}(t)$ never reaches $0$ in finite time.

\medskip  
Starting from any $z^N_{N-1}(0)\neq 0$, the trajectory satisfies $z^N_{N-1}(t)\neq 0$ for all $t\ge 0$, and tends to $0$ only asymptotically.  
Therefore $\lambda_1(e^{z^N_{N-1}(t)}-1)\neq 0$ for all $t\ge 0$.
\end{proof}
Let us the following result.
\begin{theorem}\label{theorem3.20}
There exists a feedback control law under which the general non-transitive competition system \eqref{ee3.11} is globally asymptotically stable. Furthermore, the feedback control constructed remains uniformly bounded. In particular, the control $u$ defined in \eqref{equation3.123} and the Lyapunov function $V_{N}$ in \eqref{ea3.146} ensures the global asymptotic stability of the $N-$species system \eqref{a3.139}. Moreover, the control satisfies $u(t)>0$ for every $t>0,$ for every $\eta_i(0)$ belonging to the largest
level set of $V_{N}(\eta_2, z^N_1, \cdots, z^N_{N-1})$ within the set 
\begin{align}
\mathcal{K}'_{N}= 
\left\{ \eta\in\mathbb{R}^{N}\ \middle|\ 
\begin{aligned}
&u(t)>0\quad\text{in}\;\eqref{equation3.123}
\end{aligned}
\right\}.
\end{align}
\end{theorem}
\begin{proof}[of Theorem \ref{theorem3.20}]
From Lemma \ref{lem3.19}, we have 
\begin{align}
\begin{aligned}
V_N(\eta_2,z^N_1,\cdots,z^N_{N-1})&= \theta \lambda_2(e^{\eta_{2}}-1-\eta_{2})+\theta \sum_{i=1}^{N-2}\lambda_{i+2}(e^{z^N_{i}}-1-z^N_{i})+\lambda_1(e^{z^N_{N-1}}-1-z^N_{N-1}).
\end{aligned}
\end{align}
The time derivative $\dot V_{N}$ of the Lyapunov function $V_{N}$ satisfies
\begin{align}
\begin{aligned}
\dot V_N(\eta_2, z^N_1,\dots,z^N_{N-1})_{\mid \psi_i\equiv 0}
={}&
-\theta c_1\,\phi_2^2(\eta_2)
-\theta\sum_{i=2}^{N-1}\,c_{\,i}\,\phi_{i+1}^2\bigl(z^N_{i-1}\bigr)
- c_{N}\,\phi_1^2\bigl(z^N_{N-1}\bigr) 
\end{aligned},
\end{align}
with a control of the form \eqref{equation3.123}.
\end{proof}
\begin{lemma}\label{lem3.19}
Under the model \eqref{aq3.145} assumptions, there exists $o_i>0$ such that for all $t\in \mathbb{R}_+$ 
\begin{align}
\vert\eta_i(t)\vert\leq o_i\Longrightarrow\vert A_i\vert\leq C_i,\quad i\in\lbrace 1,...,N\rbrace.
\end{align}
\end{lemma}
\begin{remark}
The control $u$ defined in \eqref{equation3.123} is uniformly bounded on $\mathbb{R}_+$. The claim follows directly from Lemma \ref{lem3.19}.
\end{remark}
Under {\bf Assumption H6}, by applying the same control of the form \eqref{equation3.123}, the $N-$ species system \eqref{a3.139} is globally asymptotically stabilizable for $\psi_i \neq 0$, with a Lyapunov function of the form 
\begin{align}\label{a3.160}
V_G(\eta,\psi)
= V_{{N}}(\eta_2,z^N_1,\cdots,z^N_{N-1})\;+\;\frac{\gamma_1}{\sigma_1}h(G_1(\psi_1))\;+\;\frac{\gamma_2}{\sigma_2}h(G_2(\psi_2))\;+\;\cdots+\;\frac{\gamma_{N}}{\sigma_{N}}h(G_{N}(\psi_{N})).
\end{align}

Thus, we obtain the general stabilization theorem for non-transitive  competition models, with the single control applied to one species synthesized by the backstepping method.
\begin{theorem}\label{theorem3.23}
Under {\bf Assumption H6}, there exists a feedback control law under which the general non-transitive competition system \eqref{ee3.11} is globally asymptotically stable. Furthermore, the feedback control constructed remains uniformly bounded.  In particular, the same control $u$  in \eqref{equation3.123} and the Lyapunov function $V_{G}$ in \eqref{a3.160} ensures the global asymptotic stability of the $N-$species system \eqref{a3.139}. Moreover, the control satisfies $u(t)>0$ for every $t>0,$ for every $\eta_i(0)$ belonging to the largest
level set of $V_G(\eta,\psi)$ within the set
\begin{equation}
\mathcal{K}_{N}=
\left\{ (\eta,\psi)\in \mathbb{R}^{N}\times\mathcal{S}^{N}\ \middle|\ 
\begin{alignedat}{2}
&\eta_1 \le \ln\left(\frac{\gamma_1}{C_1\lambda_1}\right), \qquad && \gamma_1 > C_1\lambda_1,\\
&\eta_2 \le \ln\left(\frac{\gamma_2}{C_2\lambda_2}\right), \qquad && \gamma_2 > C_2\lambda_2,\\
&\vdots && \vdots\\
&\eta_{N} \le \ln\left(\frac{\gamma_{N}}{C_{N}\lambda_{N}}\right), \qquad && \gamma_{N} > C_{N}\lambda_{N},\\[4pt]
&&& u(t)>0\quad\text{in }\eqref{equation3.123}
\end{alignedat}
\right\}.
\end{equation}
\end{theorem}
\begin{proof}[of the Theorem \ref{theorem3.23}]
From Lemma \ref{lemma3.19}, we have 
\begin{align}\label{}
\begin{aligned}
\dot V_{N}(\eta_2, z^N_1,\dots,z^N_{N-1})_{\mid \psi_i\neq 0}
={}&
-\theta c_1\,\phi_2^2(\eta_2)
-\theta\sum_{i=2}^{N-1}\,c_{\,i}\,\phi_{i+1}^2\bigl(z^N_{i-1}\bigr)
- c_{N}\,\phi_1^2\bigl(z^N_{N-1}\bigr) + \sum_{i=1}^{N} A_{i}\,\left(\hat{\phi}_{i}-\phi_{i}(\eta_{i})\right);
\end{aligned}
\end{align}
applying \eqref{a3.160} subsequently gives 
\begin{align}
V_G(\eta,\psi)
\leq &{}
-\theta c_1\,\phi_2^2(\eta_2)
-\theta\sum_{i=2}^{N-1}\,c_{\,i}\,\phi_{i+1}^2\bigl(z^N_{i-1}\bigr)
- c_{N}\,\phi_1^2\bigl(z^N_{N-1}\bigr) + \sum_{i=1}^{N} \left(A_{i}\,\left(\phi_{i}(\eta_{i})+\lambda_i\right)-\gamma_i\right)\left(e^{G_i}-1\right).
\end{align}

From Lemma \ref{lem3.19}, we get 
\begin{align}
V_G(\eta,\psi)
\leq &\ -\theta c_1\,\phi_2^2(\eta_2)
-\theta\sum_{i=2}^{N-1}\,c_{\,i}\,\phi_{i+1}^2\bigl(z^N_{i-1}\bigr)
- c_{N}\,\phi_1^2\bigl(z^N_{N-1}\bigr)  -\theta\,c_{N}\,\phi_{N+1}^2\bigl(z^N_{N-1}\bigr)
   -c_{N+1}\phi_1^2\bigl(z^N_{N}\bigr)\\
&\  + \sum_{i=1}^{N} \left(C{i}\,\vert\phi_{i}(\eta_{i})+\lambda_i\vert-\gamma_i\right)\left(e^{G_i}-1\right).
\end{align}
establishing the result.
\end{proof}

\begin{proposition}\label{pro3.13}
There exists $c>0$ such that a connected component of the sublevel set
\[
\mathcal L_c := \{(\eta,\psi)\mid V_G(\eta,\psi)\le c\}
\]

is contained in $\mathcal{K}_{N}$.
\end{proposition}
\begin{proof}[of Proposition \ref{pro3.13}]
We select an initial vector

$$
(\eta^0,\psi^0), \qquad \text{with } \eta^0=(\varepsilon,0,0\cdots,0), \;\; \psi^0=(0,0,0,\cdots,0),
$$

where $\varepsilon\neq0$ is a sufficiently small real number, chosen such that $
0<\varepsilon<\min\Big\{\ln\tfrac{\gamma_1}{C_1\lambda_1},\;\ln\tfrac{\gamma_2}{C_2\lambda_2},\;\ln\tfrac{\gamma_3}{C_3\lambda_3},\cdots,\ln\frac{\gamma_{N}}{C_{N}\lambda_{N}}\Big\}.$ Since $\phi_1$ is continuous with $\phi_1(0)=0$, one can further choose $\varepsilon>0$ small enough to guarantee $|c_{No}\phi_1(\varepsilon)|<\tfrac{u^*}{2}.$ Therefore,
$$
u^*+c_{No}\phi_1(\varepsilon)>\tfrac{u^*}{2}>0,
$$
which shows that $(\eta^0,\psi^0)\in \mathcal K_{N}$. In particular, this implies $\mathcal K_{N}\neq\varnothing$. Next, by continuity we have

$$
V_G(\eta^0,\psi^0)=V_{N}(0,\cdots,0,\varepsilon)+\sum_{i=1}^{N} \frac{\gamma_i}{\sigma_i}\, h(G_i(\psi_i^0)).
$$

Since $V_{N}(0,\cdots,0,\varepsilon)\to 0$ as $\varepsilon\to 0$, and $h(0)=0$ whenever $G_i(\psi_i^0)=0$, it follows that $V_G(\eta^0,\psi^0)$ can be made arbitrarily small.

Hence, for any given $c>0$, there exists $\varepsilon>0$ sufficiently small such that $V_G(\eta^0,\psi^0)\leq c.$ Consequently, $(\eta^0,\psi^0)\in \mathcal L_c\cap \mathcal K_{N},$ which proves that $\mathcal L_c\cap \mathcal K_{N} \neq \varnothing.$
\end{proof}

\begin{remark}
In the study of the stability of the non-transitive competition model with three species, four species, and its generalization, the state $\eta_2$ is used as a reference in the backstepping stabilization approach, depending on the localization of the control. For instance, in the case of three species, the control $u$ is applied to species $x_1$, while the newborns $x_2^*(0)$ in \eqref{e3.6} are affected by the control $u^*$.  For example, when the control $u$ is applied directly to $x_2$, the signal $x_3$ is chosen as the reference state (or tracking reference) and is used to define the fictitious/intermediate control in the backstepping procedure.
\end{remark}

After the general analysis presented above, we turn our attention to a particular case: mosquito dynamics. While the traditional literature predominantly favors unstructured models, the explicit inclusion of the age variable in mosquito models remains relatively understudied, particularly with respect to control strategies applied to aquatic and adult stages.

\subsection{Control of malaria‑vector mosquitoes}
Many insect species, particularly mosquitoes, serve as vectors for numerous life-threatening diseases, including malaria, Zika virus, dengue fever, chikungunya, schistosomiasis, human African trypanosomiasis (sleeping sickness), yellow fever, and onchocerciasis. Globally, several thousand mosquito species have been identified, with a subset implicated in disease transmission. Notably, Aedes aegypti is a primary vector of chikungunya, Aedes albopictus of dengue and yellow fever, and Anopheles gambiae of malaria.

According to the World Health Organization's 2022 report, approximately 247 million cases of malaria were recorded worldwide, with 96\% of malaria, related deaths occurring in Africa, predominantly among children under the age of five. Vector-borne diseases exert a substantial burden on both human and animal health, and significantly affect socioeconomic development. As such, the implementation of effective vector control and disease management strategies remains a global health priority.

\subsubsection{\bf Control strategy}\label{s31}
Mosquitoes require access to water, typically stagnant or slow‐flowing habitats, to complete their holometabolous life cycle. Following oviposition at the water’s edge, eggs hatch into larvae, which undergo four instars before pupating and emerging as winged adults. Consequently, the vector population is naturally partitioned into an aquatic stage (eggs, larvae, pupae) and an aerial adult stage (males and females). Within the aquatic compartment, mortality comprises a density‐independent component (e.g., predation, adverse climatic conditions) and a density‐dependent component, reflecting competition among larvae for limited breeding sites. Upon emergence, females require mating and a blood meal, typically within 3–4 days, before initiating gonotrophic cycles of approximately 4–5 days, each yielding 100–150 eggs deposited in 10–15 distinct microhabitats. In \cite{ref41}, a four‐compartment model was proposed, tracking the dynamics of the aquatic population, nulliparous females, gravid egg‐laying females, and males, and incorporating both forms of larval mortality and the transition delays associated with mating and first blood meal.

Traditional vector‐control strategies combine indoor adulticiding, egg‐destruction measures, and larval habitat management; chemical insecticides have dominated these efforts for decades but face challenges from resistance and ecological impact. To address these limitations, complementary genetic approaches have been developed. The Sterile Insect Technique (SIT), pioneered by E. Knipling and collaborators and famously used to eradicate screwworms in 1950s Florida, involves mass‐releases of radiation‐sterilized males to suppress wild populations. Building on this concept, the Target Malaria project in West Africa has trialed releases of genetically modified sterile Anopheles males. The Incompatible Insect Technique (IIT) exploits Wolbachia, a maternally transmitted endosymbiont, to induce cytoplasmic incompatibility: released infected males render eggs inviable when mating with uninfected females, and releases that include infected females can replace wild populations with Wolbachia‐carrying lines that also exhibit reduced competence for dengue, Zika, and chikungunya viruses. Such environmentally benign, species‐specific methods aim to drive mosquito densities below the critical threshold for disease transmission, a concept first articulated by Ronald Ross, thereby achieving sustainable malaria control.
  
\subsubsection{\bf Modeling}\label{s32}
We develop an age‑structured dynamic model that simultaneously tracks wild and genetically modified mosquitoes with the goal of reducing disease transmission risk. Building on the framework of prior age‑structured studies, particularly the analysis of blood‑feeding plasticity in natural environments presented in \cite{ref44}, our formulation distinguishes four population compartments:
\begin{itemize}
\item $I$, the aquatic (immature) stage;
\item $F_j$, newly emerged (nulliparous) females;
\item $F_a$, fertilized (egg‑laying) adult females;
\item $M$, adult males.
\end{itemize}

We then introduce two vector‑control strategies. The first deploys a predator targeting the aquatic larvae, while the second releases genetically modified male mosquitoes to suppress wild populations. Through numerical simulations, we will assess how predation at the larval stage alters overall mosquito dynamics and, separately, how releases of modified males impact adult population structure. Ultimately, this study aims to elucidate the interactions between wild and engineered mosquitoes and to quantify the consequent reduction in vector‑borne disease transmission.

Although our compartmental models draw on the mosquito dynamics frameworks of \cite{ref38, ref42, ref44, ref43, ref41}, we extend these formulations by incorporating explicit age structure and a logistic term that captures environmental carrying capacity.

\section*{\bf Introduction of a predator}\label{s321}
Biological control exploits the deliberate introduction of natural enemies to suppress pest populations, especially when such pests expand unchecked in the absence of their usual predators (the ecological‐release paradigm). One classic example is the use of the mosquitofish  Gambusia affinis, introduced into Algeria in 1928 (and earlier in Europe, circa 1921) to prey upon anopheline larvae and curb malaria transmission. Native to Central America and Florida,  G. affinis thrives in diverse freshwater habitats and remains one of the most effective biological control agents against mosquitoes, readily integrating with existing vector‐management strategies.

In our age‐structured logistic model, we therefore include a predator compartment $P(t)$ that feeds exclusively on the aquatic mosquito cohort. Specifically, in system \eqref{eq3.2}, the term $
\mathcal{P}(t)
=\frac{\displaystyle\int_{0}^{A}M(a,t)\,da}{1+\displaystyle\int_{0}^{A}M(a,t)\,da+\displaystyle\int_{0}^{A}M_s(a,t)\,da}$ captures an Allee‐type effect, representing the probability of male–female encounters in the adult population. The aquatic stage spans ages $a\in(0,\tau)$. Upon emergence, adults are allocated to females and males according to a fixed sex ratio $r\in(0,1)$.  To account for the emergence of the aquatic population over the interval $(0,\tau)$, we introduce the function $w(a),$ which represents the age‑dependent emergence rate.

 System $\eqref{eq3.1}$ describes the dynamics of an Anopheles mosquito population by distinguishing three age‑structured cohorts $a$: the aquatic population $I(a,t)$, adult females $F(a,t)$, and adult males $M(a,t)$.

 In the aquatic phase, each cohort experiences natural mortality $\mu\bigl(a,p(t)\bigr)\,I(a,t),$
 where $p(t)=\int_{0}^{\tau}I(a,t)\,\mathrm{d}a$ represents the concentration of the aquatic population subject to predation or other stressors. The logistic growth term $ \Gamma(t)\,I(a,t)\Bigl(1-\tfrac{\gamma(t)}{K(t)}\!\int_{0}^{A}I(a,t)\,\mathrm{d}a\Bigr)$
 limits aquatic population development as the total cohort $p(t)$ approaches the environmental carrying capacity $K(t)$.

 To capture the impact of aquatic control campaigns, we include an exogenous mortality term $
 -\,I(a,t)\,P(t),$
 where $P(t)$ aggregates human interventions (drainage, introduction of larvivorous fish, environmental management, etc.) applied uniformly across the aquatic cohort.

 The boundary conditions then link the aquatic and adult stages. Newly hatched individuals ($a=0$) derive from eggs laid by adult females: $
I(0,t)=\int_{0}^{A}\beta\bigl(a,m(t)\bigr)\,F(a,t)\,\mathrm{d}a,$  where $\beta(a,m)$ is the fecundity rate, potentially modulated by male availability $m(t)=\displaystyle\int_{0}^A\lambda(a)M(a,t)da$.
 Upon maturation, the aquatic cohort gives rise to adult females or males according to $
 F(0,t)=r\displaystyle\int_{0}^{A}w(a)\,I(a,t)\,\mathrm{d}a,
 \quad
 M(0,t)=(1-r)\displaystyle\int_{0}^{A}w(a)\,I(a,t)\,\mathrm{d}a$ with $A=\max\lbrace\tau,A^*\rbrace.$

 Finally, the adult equations incorporate intra‑sex competition for food and shelter via terms of the form $
 -\,\gamma(t)\,F(a,t)\int_{0}^{A}F(a,t)\,\mathrm{d}a
 \quad\text{and}\quad
 -\,\gamma(t)\,M(a,t)\int_{0}^{A}M(a,t)\,\mathrm{d}a,$ which respectively constrain female and male densities.

 System \eqref{eq3.1} with control $P(t)$ can thus be studied for global asymptotic stability, while a distributed control acting across all ages of the aquatic cohort, akin to the framework in \cite{ref61}, raises natural questions of controllability under predation pressure. 

\begin{equation}\label{eq3.1}
\left\lbrace 
\begin{array}{ll}
\partial_tI(a,t)+\partial_aI(a,t)+\mu(a,p(t))I(a,t)=\Gamma(t) I(a,t)\left(1-\dfrac{\gamma(t)}{K(t)}\displaystyle\int_{0}^AI(a,t)da\right)-I(a,t)P(t) &\text{ in }Q,  
\\\partial_tF(a,t)+\partial_aF(a,t)+\mu_{F}(a)F(a,t)=-\gamma(t)F(a,t)\displaystyle\int_{0}^A  F(a,t)da &\text{ in }Q,
\\\partial_tM(a,t)+\partial_aM(a,t)+\mu_M(a)M(a,t)=-\gamma(t) M(a,t)\displaystyle\int_{0}^A M(a,t)da &\text{ in }Q,
\\I(0,t)=\displaystyle\int_{0}^A\beta(a,m)F(a,t)da,\; F(0,t)=r\displaystyle\int_0^Aw(a)I(a,t)da,\; M(0,t)=(1-r)\displaystyle\int_0^Aw(a)I(a,t)da &\text{ in }Q_T,
\\I(a,0)\geq  0,\quad F(a,0)\geq  0,\quad M(a,0)\geq  0, &\text{ in }Q_A,\\
\\P(t)\ge 0,\;K(t)\ge \epsilon>0,\;\Gamma(t)\ge 0,\gamma(t)\ge 0, &\text{ in }Q_T.
\end{array}
\right.
\end{equation} 
   In our framework, the time‐dependent function $P(t)$ is interpreted not as the intrinsic dynamics of a Gambusia affinis population, but as a unified control parameter representing all human‐driven interventions against the aquatic mosquito stage, whether by fish releases, habitat drainage, larviciding, or other larval‐reduction measures.  In practice, these activities are planned and scheduled by antimalarial programs according to predetermined frequencies, dosages, and target areas; accordingly, $P(t)$ appears in the model as an exogenous mortality rate term, $-I(a,t)\,P(t)$, applied uniformly across the aquatic cohort.  This aggregation of disparate control actions into a single, time‐varying parameter greatly simplifies the system by obviating the need for an extra differential equation for the predator, while still capturing the combined ecological and operational constraints of vector‐control campaigns.

Furthermore, although one could introduce interspecific terms in the adult mortality rates to reflect resource competition (e.g. for nectar or resting sites), we assume that adult female and male death rates depend solely on age.  This assumption aligns with the natural separation of feeding niches (blood meals for females versus nectar for males) and allows us to concentrate the mathematical analysis on the stability effects of the aquatic‐stage control $P(t)$.  In particular, the multiplicative form of the control enables a clear investigation of global asymptotic stability in Section \ref{s34}. 

Since model \eqref{eq3.1} is a non-autonomous logistic model, to carry out its qualitative analysis we may likewise replace the time-dependent functions $K(t)$, $\Gamma(t)$, and $\gamma(t)$ (see Section \ref{s34}) by their respective mean values $K^*$, $\Gamma^*$, and $\gamma^*$. In the dynamic case, the functions $K(t)$, $\Gamma(t)$, and $\gamma(t)$ are assumed to be continuous and bounded on the interval $(0,T),$ namely assuming hypothesis

\begin{align*}
(HH) : \begin{cases}
K(t),\; \Gamma(t),\; \gamma(t) \in L^{\infty}(0,T),\\
\\K(t)\ge \epsilon>0,\;\Gamma(t)\ge 0,\gamma(t)\ge 0, &\text{ in } (0,T).
\end{cases}
\end{align*}
Moreover, we adopt the following standing hypotheses (unless otherwise stated):
\begin{align*}
\textbf{(H33)}\begin{cases}
\mu_i(a),\; \mu(a,p)\geq 0 \quad \text{a.e. on }(0,A), \\[0.3em]
\mu_i \in L^1_{\rm loc}(0,A), \quad \displaystyle\int_0^{A}\mu_i(a)\,da = +\infty,
\end{cases}
\qquad
\textbf{(H44)}\begin{cases}
\beta(.,m),\; w(.) \in L^\infty(0,A),  \\[0.3em]
\beta(a,m),\; w(a)\ge 0 \ \text{a.e. on } (0,A).
\end{cases}
\end{align*}
\section*{\bf Genetic Control}\label{s322} 
Introduce sterile male mosquitoes \( M_S \) into the adult population. This genetic control strategy disrupts reproduction, effectively limiting population growth.  In this model, we incorporate an interspecific interaction between fertile males $M$ and genetically modified sterile males $M_S$, represented through cross terms in their respective equations. This interaction reflects a competitive dynamic, whereby each male type disrupts the reproductive contribution of the other, especially via indirect effects on survival and mating outcomes. In addition, we define the mating probability $\mathcal{P}(t)$, which expresses the likelihood that a female encounters a fertile male and thus produces viable offspring. Although sterile males $M_s$ do not contribute to reproduction, they compete for mating opportunities. As their density increases, $\mathcal{P}(t)$ declines, reducing the effective recruitment of adult females, impeding the renewal of the aquatic population, and ultimately limiting overall population growth. To quantify the impact of sterile males on fecundity, we introduce a modulated fertility rate $\beta(a, m, m_s)$ of the form $ \beta(a, m, m_s) = \beta_0(a)\, \frac{m}{m + \iota}\, e^{-\delta\, m_s},$ where $\beta_0(a)$ denotes the age-dependent baseline fertility,  $e^{-\delta\, m_s}$, ($\delta > 0$), models the global inhibitory effect of sterile males on oviposition, $\frac{m}{m + \iota}$, ($\iota > 0$), reflects the proportion of fertile males among the total male population. Finally, the control intervention is modeled by an impulsive function $\Lambda(t)=\displaystyle\sum_{k=1}^n\alpha_k\delta_{\lbrace t_k\rbrace}(t),\quad t_0=0<t_1<...<t_n<T,$ which represents the periodic release of sterile male cohorts at discrete times $t_k$.

\begin{equation}\label{eq3.2}
\left\lbrace 
\begin{array}{ll}
\partial_tI(a,t)+\partial_aI(a,t)+\mu(a,p(t))I(a,t)=\Gamma(t) I(a,t)\left(1-\dfrac{\displaystyle\int_{0}^A\beta(a,m,m_s) F_a(a,t)da}{K(t)}\right), &\text{ in }Q, 
\\\partial_tF_j(a,t)+\partial_aF_j(a,t)+\mu_{F_j}(a)F_j(a,t)=-F_j(a,t)\;\displaystyle\int_{0}^A\gamma(t) F_j(a,t)da &\text{ in }Q, 
\\\partial_tF_a(a,t)+\partial_aF(a,t)+\mu_{F_a}(a)F_a(a,t)=-F_a(a,t)\displaystyle\int_{0}^A\gamma (t)F_a(a,t)da, &\text{ in }Q,
\\\partial_tM(a,t)+\partial_aM(a,t)+\mu_M(a)M(a,t)=-M(a,t)\displaystyle\int_{0}^A\gamma(t) M_S(a,t)da &\text{ in }Q,
\\\partial_tM_s(a,t)+\partial_aM_s(a,t)+\mu_{M_s}(a)M_s(a,t)=-M_s(a,t)\displaystyle\int_{0}^A\gamma(t) M(a,t)da +M_s(a,t)\Lambda(t) &\text{ in }Q,
\\I(0,t)=\displaystyle\int_{0}^A\beta(a,m,m_s)F_a(a,t)da,\; F_j(0,t)=r\displaystyle\int_{0}^Aw(a)I(a,t)da,&\text{ in }Q_T,
\\M(0,t)=(1-r)\displaystyle\int_{0}^Aw(a)I(a,t)da,\;F_a(0,t)=\mathcal{P}(t)\displaystyle\int_0^AF_j(a,t)da,\;M_s(0,t)=0 &\text{ in }Q_T,
\end{array}
\right.
\end{equation}

Another model has already been studied in \cite{ref38,ref42,ref43}, without considering age, through genetically modified mosquitoes, a Sterile Insect Technique control strategies with constant or variable number of sterile males to be released that drive the wild population of mosquitoes towards elimination. The mortality of the sterile males is usually larger than that of wild males \cite{ref73}, i.e. $\mu_{M_S}\geq\mu_M.$  The description of the parameters is given in Table \ref{t1} below.
\begin{table}[ht]
\centering
\begin{tabular}{c|p{10cm}}
\hline
Parameter & Description \\
\hline
$\lambda$ & Fertility function of male individuals. \\
$r \in (0,1)$ & Primary sex ratio in offspring. \\
$\beta$ & Mean number of eggs that a single female can deposit on average per day. \\
$w$ & the age‑dependent emergence rate. \\
$\mu(a,p(t)),\,\mu_{F_j},\,\mu_{F_a},\,\mu_M,\,\mu_{M_s}$ & Mean death rates of immature individuals (density‐dependent and independent), young females, fertilized females, males and sterile males, respectively. \\
$\gamma(t)$ (or $\gamma^*$) & Competition parameter. \\
$K(t)$ (or $K^*$) & Carrying capacity related to the amount of available nutrients and space. \\
$\Gamma(t)$ (or $\Gamma^*$) & Growth rate. \\
\hline
\end{tabular}
\caption{Description of the parameters}
\label{t1}
\end{table}
\subsubsection*{\bf Well-posedness}
  We establish the well‑posedness of the time‑evolution problem by means of the semigroup approach.  To this end, let $\mathcal H_2^5 = \bigl(L^2(0,A)\bigr)^5,$
  and define the linear operator

 $$
   \mathcal A_m : D(\mathcal A_m)\subset\mathcal H_2^5\;\longrightarrow\;\mathcal H_2^5,
   \quad
   \mathcal A_m\varphi
   = -\partial_a\varphi \;-\; D(a,p)\,\varphi,\quad\text{where}\; \varphi=(\varphi_I,\varphi_{F_j},\varphi_{F_a},\varphi_M,\varphi_{M_s})
 $$
 with

\begin{align*}
D(\mathcal A_m)
= \Bigl\{\varphi\in\mathcal H_2^5 :\;
\varphi\text{ is a.c. on }[0,A],\;
&\varphi_I(0)=\int_{0}^{A}\beta(a,m,m_s)\,\varphi_{F_a}(a)\,da,
&\varphi_{F_j}(0)=r\int_{0}^{A}w(a)\,\varphi_I(a)\,da,\\
&\varphi_M(0)=(1-r)\int_{0}^{A}w(a)\,\varphi_I(a)\,da,
&\varphi_{F_a}(0)=\mathcal P\,\int_{0}^{A}\varphi_{F_j}(a)\,da,\\
&\varphi_{M_s}(0)=\Lambda\,\int_{0}^{A}\varphi_{M_s}(a)\,da,
&-\partial_a\varphi - D(a,p)\,\varphi\in\mathcal H_2^5
\Bigr\}.
\end{align*}

 In block‑diagonal notation,

 $$
   \mathcal A_m
   = \mathrm{diag}\bigl(-\partial_a-\mu_I,\;-\partial_a-\mu_{F_j},\;
   -\partial_a-\mu_{F_a},\;-\partial_a-\mu_M,\;-\partial_a-\mu_{M_s}\bigr).
 $$



 Finally, the nonlinear fonction $f:\mathcal H_2^5\to\mathcal H_2^5$ is defined component‑wise by

 $$
  f\bigl(I,F_j,F_a,M,M_s\bigr)
  = \bigl(I\,f_1,\,F_j\,f_2,\,\,F_a\,f_3,\,M\,f_4,\,M_s\,f_5\bigr)^\top
$$
with 
\begin{align}\label{e3.40}
f_1=\Gamma(t)\left(1-\dfrac{\displaystyle\int_{0}^A\beta(a,m,m_s) F_a(a,t)da}{K(t)}\right),\; f_2=-\displaystyle\int_{0}^A\gamma(t) F_j(a,t)da
\end{align}
\begin{align}\label{e3.41}
f_3=-\displaystyle\int_{0}^A\gamma(t)F_a(a,t)da,\;f_4=-\displaystyle\int_{0}^A\gamma(t) M_S(a,t)da,\; f_5=-\displaystyle\int_{0}^A\gamma(t) M(a,t)da.
\end{align}
Let 
\begin{align}\label{e3.42}
    Y(t)=\left( I(a,t), F_j(a,t), F_a(a,t),  M(a,t), M_s(a,t)\right)\in D(\mathcal{A}_m)
\end{align}
thus, we can rewrite the system \eqref{eq3.2} as an abstract Cauchy problem
\begin{equation}\label{e3.43}
\left\lbrace 
\begin{array}{ll}
\partial_tY(t)=\mathcal{A}_mY(t)+f(Y(t)),&\text{ in}\;Q_T\\
Y(0)=Y_0
\end{array}
\right.
\end{equation} 
where
\begin{align}
    Y_0=\left(I(a,0), F_j(a,0), F_a(a,0),  M(a,0), M_S(a,0)\right).
\end{align}
\begin{remark}
The mortality $\mu_i,$ fertility $\beta$ functions  and $w$ the age‑dependent emergence rate  satisfy hypotheses $(H33)$ and $(H44)$, and the function $f$ meets condition $(H3)$.
\end{remark}
Thus, investigating the well-posedness of system \eqref{eq3.2} reduces to studying equation \eqref{e3.43} along with its initial.  Hence, by applying Theorem \ref{th2.7}, we obtain well‑posedness in \( \mathcal{H}^5_2 \). By applying the same strategy, we easily show that system \eqref{eq3.1} is well-posed.
\begin{remark}
By applying the method of characteristics to the system \eqref{eq3.1}, one finds that, for every $(a,t)\;\in Q$, the solutions of \eqref{eq3.1} can be written as follows:

\begin{align}\label{e3.45}
\begin{cases}
I=I\bigl(0,\,t - a\bigr)\,e^{-\!\displaystyle\int_{0}^{\,a}\mu_{I}\bigl(\alpha,\,p\bigl(\alpha - (a - t)\bigr)\bigr)\,d\alpha 
\;+\;\int_{\,t - a}^{\,t}R_{I}(s)\,ds},
&  R_{I}(s)
=
\Gamma(t)\left(1-\dfrac{\gamma(t)}{K(t)}\displaystyle\int_{0}^AI(x,s)dx\right)
\;-\;P(s),\\
 F=  F(0,t-a)\,e^{-\!\displaystyle\int_{0}^{\,a}\mu_{F}(\alpha)\,d\alpha 
   \;+\;\int_{\,t - a}^{\,t}R_{F}(s)\,ds},
   &   R_{F}(s)=-\displaystyle\int_{0}^{A}\gamma(t)\,F(x,s)\,dx,\\
M= M(0,t-a)\,e^{-\!\displaystyle\int_{0}^{\,a}\mu_{M}(\alpha)\,d\alpha 
   \;+\;\int_{\,t - a}^{\,t}R_{M}(s)\,ds},
   &   R_{M}(s)
   =
   -\,\int_{0}^{A}\gamma(t)\,M(x,s)\,dx.
\end{cases}
\end{align}
\end{remark}


\subsubsection{\bf Stability analysis}\label{s34}
 This step focuses on the mathematical analysis of the stability of the model \eqref{eq3.1}. The objective is to examine how the biological control $P$, when applied to the aquatic population, influences the overall dynamics of the system. A steady‐state formulation of \eqref{eq3.1} takes the form

\begin{equation}\label{eq3.11}
\left\lbrace 
\begin{array}{ll}
\partial_a\,I^*(a)\;+\;\bigl(\mu\bigl(a,p^*\bigr)\;+\;\zeta_I\bigr)\,I^*(a)\;=\;0, 
& \text{ in}\;Q_A,\\[1em]
\partial_a\,F^*(a)\;+\;\bigl(\mu_F(a)\;+\;\zeta_F\bigr)\,F^*(a)\;=\;0, 
& \text{ in}\;Q_A,\\[1em]
\partial_a\,M^*(a)\;+\;\bigl(\mu_M(a)\;+\;\zeta_M\bigr)\,M^*(a)\;=\;0, 
& \text{ in}\;Q_A,\\[1em]
I^*(0)=\displaystyle \int_{0}^{A} \beta\bigl(a,m^*\bigr)F^*(a)da,\;F^*(0)=r\displaystyle\int_0^Aw(a)I^*(a)da,\;M^*(0)=(1-r)\,\displaystyle\int_0^Aw(a)I^*(a)da.
\end{array}
\right.
\end{equation}
where

\begin{align}\label{e3.47}
    \zeta_I 
\;=\; \frac{\Gamma^*\gamma^*}{K^*}\,\displaystyle \int_{0}^{A} \,I^*(a)\,da \;+\; P^*-\Gamma^*,
\quad
\zeta_F 
\;=\gamma^*\;\displaystyle \int_{0}^{A} \,F^*(a)\,da,
\quad
\zeta_M 
\;=\gamma^*\; \displaystyle \int_{0}^{A} \,M^*(a)\,da.
\end{align}

The corresponding solutions are given by

\begin{align}\label{e3.48}
    I^*(a)= I^*(0)\,\underbrace{e^{-\displaystyle\int_{0}^{a}\bigl[\mu_I(s,p^*) + \zeta_I\bigr]\,ds}}_{\Tilde{I}^*(a)},\;F^*(a) 
= F^*(0)\,\underbrace{e^{-\displaystyle\int_{0}^{a}\bigl[\mu_F(s) + \zeta_F\bigr]\,ds}}_{\Tilde{F}^*(a)},\;M^*(a) 
&= M^*(0)\,\underbrace{e^{-\displaystyle\int_{0}^{a}\bigl[\mu_M(s) + \zeta_M\bigr]\,ds}}_{\Tilde{M}^*(a)},
\end{align}
where $\zeta_I$ and $\zeta_F$ are solutions of 
\begin{align}\label{e3.49}
r\displaystyle\int_0^Aw(a)\tilde{I}^*(a)da\displaystyle\int_{0}^A\beta(a,m)\tilde{F}^*(a)da=1.
\end{align}
We rewrite these solutions of the form
\begin{align}\label{e3.50}
    F^*(a)=rI^*(0)\displaystyle\int_0^Aw(a)\tilde{I}^*(a)da\tilde{F}^*(a),\;M^*(a)=(1-r)I^*(0)\displaystyle\int_0^Aw(a)\tilde{I}^*(a)da\tilde{M}^*(a).
\end{align}
\begin{remark}
 Ensuring the stability of $I$ automatically ensures the stability of both $M$ and $F$. Indeed, we have from \eqref{eq3.11} 
\begin{align}\label{e3.51}
  I^*(0)=\dfrac{K^*}{\Gamma^*\gamma^*}\dfrac{(\left[\zeta_I+\Gamma^*\right]-P^*)}{\displaystyle\int_0^A\tilde{I}^*(a)da}>0,\quad\; P^*\in (0,\zeta_I+\Gamma^*).
\end{align}
It is noteworthy that, according to this expression, increasing the equilibrium control $P^*$ leads to a pronounced reduction in the steady-state abundance of both male and female mosquitoes. In other words, bolstering the predator population has a directly dampening effect on mosquito dynamics, underscoring the decisive influence of controlling the aquatic phase on the system’s overall behavior. Thus, our analysis of the stability of system \eqref{eq3.1} reduces to the stabilization of the aquatic population.
\end{remark}

\begin{lemma}\label{le3.4}
Consider the following transformation
\begin{align}\label{e3.52}
 \left[\begin{array}{c}
\eta_I(t) \\ 
\psi_I(t-a)\\
\psi_F(t-a)
\end{array}  \right]=\left[\begin{array}{c}
\ln[\Pi_I(I(t))] \\ 
\dfrac{I(a,t)}{I^*(a)\Pi_I(I(t))}-1\\
\dfrac{F(a,t)}{F^*(a)\Pi_I(I(t))}-1
\end{array} \right],
\end{align}
where
\begin{align}\label{e3.53}
\Pi_I(I(t))=\dfrac{\langle \pi_{0,I}, I(t)\rangle_{L^2(0,A)}}{\langle\pi_{0,I}, I^*\rangle_{L^2(0,A)}},
\end{align}

with  $\pi_{0,I},\; \pi_{0,j}$ are continuous functions of the form
\begin{align}\label{e3.54}
   \pi_{0,I}(a)= \displaystyle\int_a^{A}\beta(s,m)e^{\int_s^a(\zeta_I+\mu_I(l,p)dl}ds.
\end{align}
Moreover, the variables $\psi_i$ and $\eta_I$ satisfy:

\begin{align}\label{e3.55}
\begin{cases}
    \partial_t\eta_I(t)=\zeta_I-P(t)+\Gamma(t)-\frac{\Gamma(t)\gamma(t)}{K(t)}e^{\eta_I}\displaystyle\int_{0}^A (1+\psi_I(t-a))I^*(a)da,\\
\eta_{I}(0)
= \ln\!\bigl(\Pi[I_{0}]\bigr)
= \eta_{I,0},
\end{cases}
\end{align}


\begin{align}\label{e3.56}
\begin{cases}
\psi_I(t)=\displaystyle\int_0^Ag_F(a)\psi_F(t-a)F^*(a)da,\\
\psi_F(t)=\displaystyle\int_0^Ag_I(a)\psi_I(t-a)I^*(a)da,\\
 \psi_{i}(-a)
=\frac{i_0(a)}{i^{*}(a)\,\Pi[i_0]}\;-\;1
= \psi_{i,0}(a).
\end{cases}
\end{align}

with 
\begin{align}\label{e3.57}
g_F(a)
=\frac{\beta(a,m)F^{*}(a)}
{\displaystyle\int_{0}^{A}\beta(a,m)F^{*}(a)\,\mathrm{d}a},\; g_I(a)
=\frac{w(a)I^{*}(a)}
{\displaystyle\int_{0}^{A}w(a)I^{*}(a)\,\mathrm{d}a},
\;\text{and}
\;\int_{0}^{A}g_F(a)\,\mathrm{d}a=1,\;\int_{0}^{A}g_I(a)\,\mathrm{d}a=1.
\end{align}
The unique solutions are then given by:

\begin{align}\label{e3.58}
I(a,t)
={I}^{*}(a)\,\bigl(1+\psi_{I}(t-a)\bigr)\,e^{\eta_{I}(t)},\; F(a,t)
={F}^{*}(a)\,\bigl(1+\psi_{F}(t-a)\bigr)\,e^{\eta_{I}(t)},\; M(a,t)
={M}^{*}(a)\,\bigl(1+\psi_{F}(t-a)\bigr)\,e^{\eta_{I}(t)}.
\end{align}
\end{lemma}
\begin{remark}
The densities $I, F, M$ adopt here the form given in \eqref{e3.58}, which distinguishes them from the expressions used in \cite{ref45}, notably due to the relationships established in \eqref{e3.50}-\eqref{e3.51}. Indeed, only the aquatic population is subject to control in this model, which implies that any modification of its dynamics has a significant impact on the adult population. The diagram below illustrates the structure of our global stability proof.

\begin{center}
\resizebox{\textwidth}{!}{
\begin{tikzpicture}[node distance=2.5cm, >=Stealth,
  block/.style={draw, rectangle, rounded corners, align=center, minimum width=3cm, minimum height=1cm}]
  
  \node[block] (dynamics) {Resulting dynamics:\\
    \(I(a,t),\,F(a,t),\,M(a,t)\)};
    
  \node[block, right=of dynamics] (composition) {Compose:$j\in \lbrace I, F \rbrace,\; X^*\in \lbrace  I^*,F^*,M^*\rbrace$\\
    \(\bigl(1+\psi_j(t-a)\bigr)\,X^*(a)\,e^{\eta_I(t)}\)};

     \node[block, right=of composition] (condition) {stability of $\eta_I,\;\psi_j$:\\
    \(\eta_I\to 0,\;\psi_j\to 0\)};
    
  \node[block, right=of condition] (equilibrium) {Equilibrium :\\
    \(I^*(a),\;F^*(a),\;M^*(a)\)};

  \draw[->] (dynamics) -- (composition);
  \draw[->] (composition) -- (condition);
  \draw[->] (condition) -- (equilibrium);

\end{tikzpicture}
}
\end{center}
\end{remark}
\begin{proof}[of Lemma \ref{le3.4}]




By multiplying the equations of system \eqref{eq3.1} respectively by the functions $\pi_{0,I}$, $\pi_{0,F}$, and $\pi_{0,M}$, and then integrating by parts over the interval $(0, A)$, we obtain 
\begin{align}\label{e3.59}
     \left\langle \pi_{0,F}(a), \partial_ tF(a,t)\right\rangle= r\left\langle \pi_{0,F}(0)w(a),I(a,t)\right\rangle+\left\langle\partial_a\pi_{0,F}(a)-\pi_{0,F}(a)(\mu_F(a)+\zeta_F), F(a,t)\right\rangle
\end{align}
\begin{align*}
    +\left\langle\pi_{0,F}(a),(\zeta_F-\displaystyle\int_{0}^A \gamma(t)F(a,t)da)F(a,t)\right\rangle,
\end{align*}

\begin{align}\label{e3.60}
     \left\langle \pi_{0,M}(a), \partial_ tM(a,t)\right\rangle= (1-r)\left\langle \pi_{0,M}(0)w(a),I(a,t)\right\rangle+\left\langle\partial_a\pi_{0,M}(a)-\pi_{0,M}(a)(\mu_M(a)+\zeta_M), M(a,t)\right\rangle
\end{align}
\begin{align*}
    +\left\langle\pi_{0,M}(a),(\zeta_M-\displaystyle\int_{0}^A \gamma (t)M(a,t)da)M(a,t)\right\rangle,
\end{align*}

\begin{align}\label{e3.61}
    \left\langle \pi_{0,I}(a),\partial_ tI(a,t)\right\rangle=\left\langle \pi_{0,I}(0)\beta(a,m),F(a,t)\right\rangle+\left\langle \partial_a\pi_{0,I}(a)-\pi_{0,I}(a)(\mu(a,p(t))+\zeta_I), I(a,t)\right\rangle+\left\langle\pi_{0,I}(a), (\zeta_I-P(t))I(a,t)\right\rangle
\end{align}

\begin{align*}
    +\left\langle\pi_{0,I}(a),(\Gamma(t)-\dfrac{\Gamma(t)\gamma(t)}{K(t)}\displaystyle\int_{0}^AI(a,t)da)I(a,t)\right\rangle.
\end{align*}
By summing them, we obtain

\begin{align}\label{e3.62}
    \left\langle \pi_{0,F}(a), \partial_ tF(a,t)-(\zeta_F-\displaystyle\int_{0}^A \gamma (t)F(a,t)da)F(a,t)\right\rangle+ \left\langle \pi_{0,I}(a),\partial_ tI(a,t)-(\zeta_I-P(t)+\Gamma(t)-\dfrac{\Gamma(t)\gamma(t)}{K(t)}\displaystyle\int_{0}^AI(a,t)da)I(a,t)\right\rangle+
\end{align}
\begin{align*}
    \left\langle \pi_{0,M}(a), \partial_ tM(a,t)-(\zeta_M-\displaystyle\int_{0}^A \gamma(t) M(a,t)da)M(a,t)\right\rangle=0,
\end{align*}
with 

\begin{align}\label{e3.63}
\begin{cases}
\mathcal{D}^*\pi_{0,I}(a)=\partial_a\pi_{0,I}(a)-\pi_{0,I}(a)(\mu(a,p(t))+\zeta_I)+r\pi_{0,F}(0)w(a)+(1-r)\pi_{0,M}(0)w(a),\quad \pi_{0,I}(A)=0,\\
\\\mathcal{D}^*\pi_{0,F}(a)=\partial_a\pi_{0,F}(a)-\pi_{0,F}(a)(\mu(a)+\zeta_F)+\pi_{0,I}(0) \beta(a,m),\qquad    \pi_{0,F}(A)=0,\\
\\\mathcal{D}^*\pi_{0,M}(a)=\partial_a\pi_{0,M}(a)-\pi_{0,M}(a)(\mu_M(a)+\zeta_M),\qquad \pi_{0,M}(A)=0.
\end{cases}
\end{align}

For all  functions $\pi_{0,F},\;\pi_{0,I},\; \pi_{0,M}$ in $L^2(0,A)$ implies that

\begin{align}\label{e3.64}
    \partial_t F(a,t)\;=\;\Bigl(\zeta_F-\!\!\int_{0}^{A}\gamma(t)\,F(a,t)\,da\Bigr)F(a,t),\;\quad
\partial_t M(a,t)\;=\;\Bigl(\zeta_M-\!\!\int_{0}^{A}\gamma(t)\,M(a,t)\,da\Bigr)M(a,t),
\end{align}
and 
\begin{align}\label{e3.65}
\partial_t I(a,t)\;=\;\Bigl(\zeta_I - P(t) + \Gamma(t) - \tfrac{\Gamma(t)\gamma(t)}{K(t)}\!\!\int_{0}^{A}I(a,t)\,da\Bigr)I(a,t),
\end{align}

almost everywhere. Consequently
\begin{align*}
    \partial_t\eta_I(t)=\zeta_I-P(t)+\Gamma(t)-\dfrac{\Gamma(t)\gamma(t)}{K(t)}\displaystyle\int_{0}^AI(a,t)da=\zeta_I+R_I(t),
\end{align*}
and from transformation \eqref{e3.52}, we get 
\begin{align}\label{e3.66}
\partial_t\eta_I(t)=\zeta_I-P(t)+\Gamma(t)-\dfrac{\Gamma(t)\gamma(t)}{K(t)}e^{\eta_I}\displaystyle\int_{0}^A (1+\psi_I(t-a))I^*(a)da.
\end{align}
On the other hand, by definition
\begin{align}\label{e3.67}
    \psi_I(t)
 = \frac{I(0,t)\,e^{-\eta_I(t)}}{I^*(0)} \;-\; 1\Longrightarrow\psi_I(t)=\displaystyle\int_0^Ag_F(a)\psi_F(t-a)F^*(a)da.
\end{align}
By analogy, we obtain
\begin{align}\label{e3.68}
\begin{cases}
\psi_F(t)=\displaystyle\int_0^Ag_I(a)\psi_I(t-a)I^*(a)da,\\
\psi_M(t)= \displaystyle\int_0^Ag_I(a)\psi_I(t-a)I^*(a)da.
\end{cases}
\end{align}

 By applying transformation \eqref{e3.52}, we obtain equation \eqref{e3.58}.


\end{proof}
 
 For the stability analysis of Theorems \ref{th3.6} and \ref{th3.8}, we consider the following assumptions:
 \begin{enumerate}
     \item[]  {\bf Assumption H7:  $\psi_I\equiv 0$}
     \item[]  {\bf Assumption H8:  $\psi_I\neq 0.$}
 \end{enumerate}
 \subsubsection*{\bf 3.4.4.1. Stability in the absence of a delay term}\label{s341}
We arrive, under {\bf Assumption H7}, at the following system from \eqref{e3.55} :
\begin{align}\label{e3.69}
   \partial_t\eta_I(t)=\zeta_I-P(t)+\Gamma(t)-\dfrac{\Gamma(t)\gamma(t)}{K(t)}e^{\eta_I}\displaystyle\int_{0}^A I^*(a)da.
\end{align}

From equation \eqref{e3.47}, by setting 
\begin{align}\label{e3.70}
    k_I=\displaystyle\int_{0}^A I^*(a)da,\;\phi_I(\eta_I)=k_I(e^{\eta_I}-1),\; 
\end{align}

we obtain
\begin{align}\label{e3.71}
  \partial_t\eta_I=P^*-P(t)-k_I\Gamma(t)(\frac{\gamma(t)}{K(t)}-\frac{1}{k_I})+k_I\Gamma^*(\frac{\gamma^*}{K^*}-\frac{1}{k_I})-\frac{\Gamma(t)\gamma(t)}{K(t)}\phi_I(\eta_I),
\end{align}

In the stability analysis of $\eta_I,$ we introduce the following Lyapunov candidate function:

 \begin{align}\label{e3.72}
     V_{I}(\eta_{I})=\;\int_{0}^{\eta_{I}}\phi_{I}(\alpha)\,\mathrm{d}\alpha
 \;
 \end{align}
 \begin{align*}
     =\;k_I\bigl(e^{\eta_{I}}-\eta_{I}-1\bigr)
 \end{align*}
 \begin{align*}
     =\phi_I(\eta_I)-k_I\eta_I.
 \end{align*}
 
 This function satisfies the Lyapunov conditions:
\begin{itemize}
 \item  $V_{I}(0)=0$,
 \item   for all $\alpha\neq0,\;V_{I}(\alpha)>0$ and $\lim_{\alpha\to\infty}V_{I}(\alpha)=+\infty$.
\end{itemize}
\begin{theorem}\label{th3.6}
 Under {\bf Assumption H7}, the system \eqref{eq3.1} is globally asymptotically stabilizable.
\end{theorem}
\begin{proof}[of Theorem \ref{th3.6}]
Using the Lyapunov candidate $V_{I}$, we get

\begin{align}\label{e3.73}
    \dot V_{I}
=\phi_{I}(\eta_{I}(t))\;\dot\eta_{I}(t).
\end{align}

From the equation \eqref{e3.71}, we substitute:

\begin{align}\label{e3.74}
    \dot V_{I}
= \phi_{I}(\eta_{I})\,\bigl(P^*-P(t)-k_I\Gamma(t)(\frac{\gamma(t)}{K(t)}-\frac{1}{k_I})+k_I\Gamma^*(\frac{\gamma^*}{K^*}-\frac{1}{k_I}) - \frac{\Gamma(t)\gamma(t)}{K(t)}\phi_{I}(\eta_{I})\bigr).
\end{align}
By choosing a control of the form
\begin{align}\label{e3.75}
    P(t)=P^*+k_I\left[\frac{\Gamma^*\gamma^*}{K^*}-\frac{\Gamma(t)\gamma(t)}{K(t)}\right]-k_I\left[\frac{\Gamma^*}{k_I}-\frac{\Gamma(t)}{k_I}\right],
\end{align}
and thanks to {\bf Assumption (HH)}, it follows that
\begin{align}\label{e3.76}
    \dot V_{I}
= -\frac{\Gamma(t)\gamma(t)}{K(t)}\phi_{I}(\eta_{I})^{2}
\;\le\;0,
\end{align}
holds; consequently, system \eqref{eq3.1}  is asymptotically stable.
\end{proof}

\begin{remark}
The time derivative of the Lyapunov function $V_I$ can be written equivalently in quadratic form as

 \begin{align}\label{e3.77}
     \dot V_I(\eta)
 \;=\;
 -\,\bigl[\phi_{I}\;\;\phi_{I}\bigr]
  \;Q(t)\;
  \begin{bmatrix}\phi_{I} \\[4pt]\phi_{I}\end{bmatrix},\;\text{where}\;  Q(t) \;=
 \begin{pmatrix}
 \frac{\Gamma(t)^2\gamma(t)}{K(t)(\Gamma(t)+\gamma(t)-2K(t))} & -\; \frac{\Gamma(t)\gamma(t)}{\Gamma(t)+\gamma(t)-2K(t)}\\[4pt]
 -\; \frac{\Gamma(t)\gamma(t)}{\Gamma(t)+\gamma(t)-2K(t)}                & \frac{\Gamma(t)\gamma(t)^2}{K(t)(\Gamma(t)+\gamma(t)-2K(t))}
 \end{pmatrix}.
 \end{align}

 For $Q(t)$ to be positive definite, one requires

 \begin{align}\label{e3.78}
\frac{K(t)^2}{\gamma(t)}<\Gamma(t),\quad\text{a.e.}\;t\in(0,T).
\end{align}
Consequently, its smallest eigenvalue is
\begin{align}\label{e3.79}
    \lambda_{\min}(Q(t))
 =\frac{2\Gamma(t)\gamma(t)}{K(t)}\frac{\Gamma(t)\gamma(t)-K(t)^2}{(\Gamma(t)+\gamma(t)-2K(t))(\Gamma(t)+\gamma(t))+\sqrt{(\Gamma(t)+\gamma(t)-2K(t))^2\left[(\Gamma(t)-\gamma(t))^2+4K(t)^2\right]}}.
\end{align}
\end{remark}

\subsubsection*{\bf 3.4.4.2 Stability in the presence of a delay term}\label{s342}
In the context of mosquito dynamics, we analyze the system’s stability when the delay kernel $\psi_I$, associated with the aquatic population $I$, is nonzero. This assumption accounts for the developmental delays in the larval and pupal stages and requires a dedicated investigation of their influence on convergence to the global equilibrium. Under {\bf Assumption H8}, we have from \eqref{e3.47}-\eqref{e3.55} 
\begin{align}\label{3.80}
   \partial_t\eta_I(t)=P^*-P(t)-k_I\Gamma(t)(\frac{\gamma(t)}{K(t)}-\frac{1}{k_I})+k_I\Gamma^*(\frac{\gamma^*}{K^*}-\frac{1}{k_I})
\;-\frac{\Gamma(t)\gamma(t)}{K(t)}\;k_I\Biggl(\frac{e^{\eta_{I}}}{k_I}
\int_{0}^{A}I^{*}(a)\bigl(1+\psi_{I}(t-a)\bigr)\,\mathrm{d}a
-1\Biggr)
\end{align}

Define the normalized kernel

\begin{align}\label{e3.81}
g(a)
=\frac{I^{*}(a)}
{\displaystyle\int_{0}^{A}I^{*}(a)\,\mathrm{d}a},\; 
\quad
\int_{0}^{A}g(a)\,\mathrm{d}a=1,
\end{align}

Then
\begin{align}\label{e3.82}
    \frac{1}{k_I}\int_{0}^{A}I^{*}(a)\bigl(1+\psi_{I}(t-a)\bigr)\,\mathrm{d}a
=(1+\int_{0}^{A}g(a)\,\psi_{I}(t-a)\,\mathrm{d}a),
\end{align}
substituting gives

\begin{align}\label{e3.83}
    \partial_{t}\eta_{I}(t)
= P^*-P(t)-k_I\Gamma(t)(\frac{\gamma(t)}{K(t)}-\frac{1}{k_I})+k_I\Gamma^*(\frac{\gamma^*}{K^*}-\frac{1}{k_I})
\;-\;\frac{\Gamma(t)\gamma(t)}{K(t)}k_I\Bigl(e^{\eta_{I}}
\bigl[1+\!\int_{0}^{A}g(a)\psi_{I}(t-a)\,\mathrm{d}a\bigr]
-1\Bigr).
\end{align}
We therefore introduce the function
\begin{align}\label{e3.84}
\hat{\phi}_1= k\Bigl(
e^{\displaystyle\eta_{I}+\displaystyle\ln\bigl(1+\displaystyle\int_{0}^{A}g(a)\,\psi_{I}(t-a)\,\mathrm{d}a\bigr)}
-1\Bigr),
\end{align}
so that

\begin{align}\label{e3.85}
    \partial_{t}\eta_{I}
= P^*-P(t)-k_I\Gamma(t)(\frac{\gamma(t)}{K(t)}-\frac{1}{k_I})+k_I\Gamma^*(\frac{\gamma^*}{K^*}-\frac{1}{k_I}) \;-\frac{\Gamma(t)\gamma(t)}{K(t)}\;\hat\phi_{1}.
\end{align}

Finally, choosing the control \eqref{e3.75} yields the simple form

\begin{align}\label{e3.86}
    \partial_{t}\eta_{I}= \;-\frac{\Gamma(t)\gamma(t)}{K(t)}\;\hat\phi_{1},
\end{align}
with 
\begin{align}\label{e3.87}
\begin{cases}
\psi_I(t)=\displaystyle\int_{0}^Ag_F(a)\psi_F(t-a)da,\\
\psi_F(t)=\displaystyle\int_{0}^Ag_I(a)\psi_I(t-a)da.   
\end{cases}
\end{align}

Thanks to \eqref{e3.57}, 
we then make the following hypothesis \cite{ref69} :
\begin{enumerate}
     \item[]  {\bf Assumption H9: }
    There exist constants $\kappa_I,\;\kappa_F>0$ such that%

$\int_{0}^{A}\Bigl|\,
 g_F(a)\,
 \;-\;z_I\kappa_I\!\int_a^A g_F(s)\,\,ds\Bigr|\;da
 \;<\;1,\; \int_{0}^{A}\Bigl|\,g_I(a)\,\;-\;z_F\kappa_F\!\int_a^Ag_I(s)\,ds\Bigr|\;da\;<\;1$

 where $ z_I = \Bigl(\int_{0}^{A}a\,g_F(a)\,da\Bigr)^{-1},\; z_F = \Bigl(\int_{0}^{A}a\,g_I(a)\,da\Bigr)^{-1}.$  Let $\sigma>0$  be a sufficiently small constant that satisfies the inequality 
 $\int_{0}^{A}\Bigl|\,g_I(a)\,\;-\;z_F\kappa_F\!\int_a^A g_I(s)\,ds\Bigr|\;e^{\sigma a}da<\;1,\; 
 \int_{0}^{A}\Bigl|\,
  g_F(a)\,
 \;-\;z_I\kappa_I\!\int_a^A  g_F(s)\,ds\Bigr|
 \;e^{\sigma a}da<\;1.$
 \end{enumerate}
 \begin{remark}
It was proved in \cite{ref69} that the state $\psi_I$ of the internal
dynamics are restricted to the sets 
\begin{align*}
\mathcal{S}=\left\lbrace \psi_i\in\; C^0((-A,0);(-1,\infty)) : P(\psi_i)=0\wedge\psi_I(0)=\displaystyle\int_0^Ag_F(a)\psi_I(-a)da\right\rbrace,
\end{align*}
where
\begin{align*}
P(\psi_I)=\dfrac{\displaystyle\int_0^A\psi_I(-a)\displaystyle\int_a^Ag_F(s)dsda}{\displaystyle\int_a^Aag_F(a)da},
\end{align*}
and that the state $\psi_I$ is globally
exponentially stable in $\mathcal{L}^{\infty}$ norm, which means that there exist $M_i>1,\sigma_i>0$ such that 
\begin{align*}
    \Vert\psi_i(t-a) \Vert \leq M_ie^{-\sigma_i t} \Vert\psi_{i,0} \Vert_{\infty},\; i\in \lbrace I, F\rbrace.
\end{align*}
\end{remark}

Before stating the main result of this section, we define the following functions. Let the functional
\begin{align}\label{e3.89}
    G_I(\psi_I)
= \dfrac{\max_{a\in(0,A)}\bigl|\psi_I(t-a)\bigr|\,e^{-a\sigma}}{1+\max(0,\min_{a\in(0,A)}\psi_I(t-a))},
\end{align}
whose Dini derivative satisfies (see \cite{ref69})
\begin{align}\label{e3.90}
    D^+\bigl(G_I(\psi_{I,t})\bigr)
\;\le\;
-\sigma\,G_I(\psi_{I,t})
\end{align}

We then define the following Lyapunov function
\begin{align}\label{e3.91}
    V(\eta_I,\psi_I)
= V_I\;+\;\frac{\gamma_1}{\sigma}h(G_I(\psi_I)).
\end{align}
with
\begin{align}\label{e3.92}
    h(p)=\displaystyle\int_0^p\frac{1}{z}(e^z-1)^2dz.
\end{align}

\begin{theorem}\label{th3.8}
Under Assumption H9, system \eqref{eq3.1} is globally asymptotically stabilizable, and the control remains uniformly bounded. Moreover,  the control satisfies \(P(t)>0\) for all \(t>0\), for every initial condition \(\eta_I(0)\) belonging to the largest level set of $V(\eta_I,\psi_I)$ within the set

 \begin{align}\label{e3.98}
\mathcal{A} &= 
\left\{ (\eta,\psi)\in\mathbb{R}\times\mathcal{S}\ \middle|\ 
\begin{aligned}
&\eta \le \ln\!\left(\sqrt{\frac{2\gamma_1 K(t)}{\Gamma(t)\gamma(t)+K(t)\,\lambda_{\min}(Q(t))}}\right),\\[4pt]
&\gamma_1 > \frac{\Gamma(t)\gamma(t)+K(t)\,\lambda_{\min}(Q(t))}{2K(t)},\\[4pt]
& P^*+k_I\!\left[\frac{\Gamma^*\gamma^*}{K^*}-\frac{\Gamma(t)\gamma(t)}{K(t)}\right]
 -k_I\!\left[\frac{\Gamma^*}{k_I}-\frac{\Gamma(t)}{k_I}\right]>0
\end{aligned}
\right\}.
\end{align}
\end{theorem}
 \begin{proof}[of Theorem \ref{th3.8}]
Recall that

\begin{align*}
\dot V_{I}(\eta)
= -\frac{1}{2}\left(\bigl[\phi_{I}\;\;\phi_{I}\bigr]
  \;Q(t)\;
  \begin{bmatrix}\phi_{I} \\[4pt]\phi_{I}\end{bmatrix}+\bigl[\hat{\phi}_{1}\;\;\hat{\phi}_{1}\bigr]
  \;Q(t)\;
  \begin{bmatrix}\hat{\phi}_{1} \\[4pt]\hat{\phi}_{1}\end{bmatrix}\right)+\dfrac{\Gamma(t)\gamma(t)}{2K(t)}\Vert \hat{\phi}_{1}-\phi_I\Vert^2.
\end{align*}


Since $Q(t)$ is symmetric and positive semi-definite with strictly positive smallest eigenvalue $\lambda_{\min}(Q(t))>0$, it follows immediately that
\begin{align}\label{e93}
    \dot V_{I}(\eta)\;\le\; -\,\frac{\lambda_{\min}(Q(t))}{2}\,(\|\phi_I\|^{2}+\|\hat{\phi}_1\|^{2})+\dfrac{\Gamma(t)\gamma(t)}{2K(t)}\Vert \hat{\phi}_{1}-\phi_I\Vert^2.
\end{align}
and 
\begin{align}\label{e3.94}
    \Vert \hat{\phi}_{1}-\phi_I\Vert^2=(\phi_I+1)^2(e^{v_I}-1)^2,\;v_I=\displaystyle\ln\bigl(1+\displaystyle\int_{0}^{A}g(a)\,\psi_{I}(t-a)\,\mathrm{d}a\bigr)
\end{align}
then 
\begin{align}\label{e3.95}
\dot V_{1}(\eta)\;\le\; -\,\frac{\lambda_{\min}(Q(t))}{2}\,(\|\phi_I\|^{2}+\|\hat{\phi}_1\|^{2})+\dfrac{\Gamma(t)\gamma(t)}{2K(t)}(\phi_I+1)^2(e^{v_I}-1)^2.
\end{align}
For the second term $h(G_{I}(\psi_{I}))$, the Dini-derivative estimate \eqref{e3.90} implies
\begin{align}\label{e3.96}
    D^+h(G_I)=\frac{e^{G_I}-1}{G_I}\,D^+G_I.\leq -\sigma (e^{G_I}-1).
\end{align}



By applying Young’s inequality  thanks to $\vert v_I\vert\leq G_I(\psi_I)$, we obtain





\begin{align}\label{e3.97}
    D^{+}V(\eta,\psi)
\le\;
 -\,\frac{3\lambda_{\min}(Q(t))}{4}\,\|\phi_I\|^{2}+\left[\frac{K(t)\lambda_{\min}(Q(t))+\Gamma(t)\gamma(t)}{2K(t)}(\phi_I+1)^2-\gamma_1\right](e^{G_I}-1)
\end{align}
 Finally, with $(\eta_I,\psi_I)\in\mathcal{A}$, and thanks to assumption, we obtain the required estimate, and hence the equilibrium is globally asymptotically stable. 
\end{proof}
\begin{remark}
Since the original system is nonlinear, the stability analysis of the linearized system \eqref{eq3.11}, whose zero eigenvalue corresponds to the equilibrium profiles $(I^*,F^*, M^*)$, is not sufficient to guarantee overall stability 
  (it only guarantees stability in the local case, as in {\bf Section} \ref{s2.4}). We therefore employ a nonlinear method based on the study of an adjoint mode. To this end, we introduce an adjoint eigenfunction ($\pi_{0,F},\pi_{0,I}, \pi_{0,M}$) associated with the zero eigenvalue of the linearized operator. This eigenfunction acts as a filter: it allows us to project the nonlinear perturbations onto the critical age–time direction corresponding to the neutral spectral subspace. By projecting in $L^2(0,A)$, we exactly isolate the neutral mode of the dynamics \eqref{e3.63}, rather than applying an arbitrary projection. Projecting the full nonlinear dynamics onto this mode reduces the problem to a single ordinary differential equation (ODE) \eqref{e3.55} for the perturbation amplitude. Analyzing this ODE, then determines the asymptotic stability of the equilibrium. Hence, the adjoint eigenfunction is essential for completing the nonlinear analysis beyond what a mere linear spectral study can reveal. 
\end{remark}

\begin{remark}
By replacing the time-varying functions $K(t)$, $\Gamma(t)$, and $\gamma(t)$ with their average values $K^*$, $\Gamma^*$, and $\gamma^*$, in \eqref{e3.75} we recover identical global asymptotic stability results under the static control strategy 
\begin{align}\label{e3.99}
P(t)=P^*.
\end{align}
In the autonomous formulation, where all parameters are held constant, these fixed values provide a baseline for model analysis and equilibrium evaluation. When the parameters vary periodically, the model explicitly accounts for seasonal drivers such as temperature and precipitation. In the stochastic framework, the incorporation of random variability captures environmental uncertainty.
\end{remark}
\subsection*{\bf Discussion on the control strategy \texorpdfstring{$P(t)$}{P(t)}}
In the control law \eqref{e3.75},
\begin{align}
P(t)=P^*+(\Gamma(t)-\Gamma^*)+k_I\Big(\frac{\Gamma^*\gamma^*}{K^*}-\frac{\Gamma(t)\gamma(t)}{K(t)}\Big),
\end{align}
the term $\Gamma(t)-\Gamma^*$ plays the role of a direct feedforward correction : it instantaneously compensates any variation in the growth rate $\Gamma$ and follows its phase shifts. The proportional term is explicitly
\begin{align}
\,k_I\Big(\dfrac{\Gamma^*\gamma^*}{K^*}-\dfrac{\Gamma(t)\gamma(t)}{K(t)}\Big)\,,
\end{align}
which constitutes a feedback mechanism on the normalized demographic pressure $\dfrac{\Gamma\gamma}{K}$. This term acts as a sensor of "demographic energy" and tends to drive the current pressure back toward its nominal value. The parameter $k_I$, representing the static aquatic total population, scales the feedback response : the larger $k_I$, the stronger the control reacts to deviations in demographic pressure. Hence, $k_I$ determines the intensity of the control effort applied. Under periodic forcing, $P(t)$ both tracks and attenuates parametric variations rather than allowing them to amplify the aquatic population; this results in a reduction of the oscillation amplitude of $I(a,t)$ and in a re-centering of the dynamics around a periodic equilibrium of smaller amplitude. By construction, $P(t)$ aims to restore the normalized pressure and promotes global asymptotic stability around the target state $I^*(a)$. The rigorous proof of this stabilization relies on the invariance of an appropriate attraction region (e.g., $\mathcal A$ defined in \eqref{e3.98}) and on the proper choice of $k_I$ and $P^*$, which guarantee the positivity and effectiveness of the control.

\begin{remark}
The condition \eqref{e3.78} defines a critical threshold below which the equilibrium $(I^*,F^*,M^*)$ is globally asymptotically stable. Mathematically, it implies that the carrying capacity $K(t)$ is too small to offset the intrinsic growth rate $\Gamma(t)$ once intra‑aquatic competition $\gamma(t)$ is taken into account: the quadratic regulation term dominates, preventing sustained growth. Biologically, this means that even with a high growth rate $\Gamma(t)$, a limited number of aquatic habitats (low $K(t)$) cannot sustain the population: mortality driven by intra‑aquatic competition outpaces cohort expansion. This inequality \eqref{e3.78} thus provides a clear operational criterion: reducing the carrying capacity $K(t)$ amplifies the effect of intra‑aquatic competition ($\gamma(t)$), thereby shifting the system into the regime where the trivial equilibrium is attractive. Moreover, it captures the effectiveness of control measures modeled by $-I(a,t)\,P(t)$: any intervention that increases aquatic mortality (via $P(t)$) is equivalent to decreasing $K(t)$ or enhancing $\gamma(t)$, thereby facilitating the condition \eqref{e3.78}. In practice, maintaining a constant level of control ensures this threshold is met, guaranteeing that the trivial equilibrium remains the sole attractor and thereby effectively guiding aquatic control strategies.
\end{remark}

\begin{remark}
The global stability analysis carried out in this section rigorously confirms that biological control targeting the aquatic stages of mosquito populations can, under specific structural conditions on the system’s parameters, lead to asymptotic stabilization of the equilibrium, thereby reflecting a sustained reduction in vector dynamics.

Beyond the theoretical results, several historical and contemporary examples support the effectiveness of such control strategies. A notable illustration is the case of Mandatory Palestine in the 1920s \cite{ref72}, where malaria was eliminated not through insecticides or vaccination, but primarily via the continuous destruction of larval breeding sites through systematic management of aquatic habitats, supported by community education and involvement.

Similar outcomes have been observed in regions such as Zanzibar, southern Tanzania, and rural India, where environmental sanitation, drainage, and the introduction of natural predators such as larvivorous fish have significantly reduced malaria transmission.

These observations, combined with our mathematical framework, suggest that biological control of mosquito aquatic populations is not only ecologically sustainable but also structurally effective in reducing malaria endemicity. This strategy, often complementary to chemical or genetic approaches, offers a relevant and efficient lever in malaria control policies, particularly in rural or semi-urban settings where continuous deployment of conventional interventions may be more challenging.
\end{remark}
\section*{\bf Numerical simulation}\label{s35}
To solve system \eqref{eq3.1}, the age discretization is performed with finite difference method on $(0,A).$ For more details on the discretization, the reader is invited to consult \cite{ref58, ref57}.  
\begin{figure}[H]
  \centering
  \begin{subfigure}[b]{0.3\textwidth}
    \centering
\includegraphics[width=\textwidth]{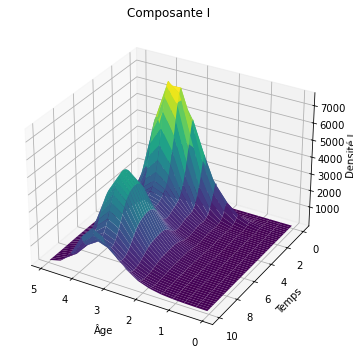}
    \caption{Cuve I}
  \end{subfigure}
  \hfill
  \begin{subfigure}[b]{0.3\textwidth}
    \centering
\includegraphics[width=\textwidth]{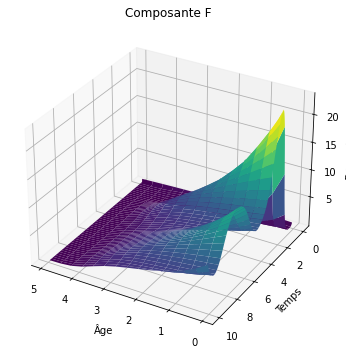}
    \caption{Cuve F}
  \end{subfigure}
  \hfill
  \begin{subfigure}[b]{0.3\textwidth}
    \centering   \includegraphics[width=\textwidth]{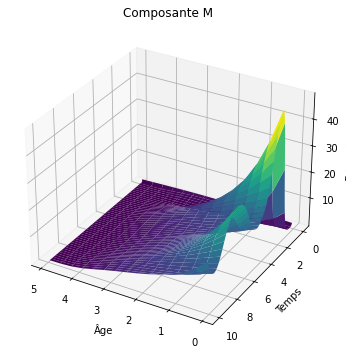}
    \caption{Cuve M}
  \end{subfigure}
  \caption{Uncontrolled density ($k(t),\gamma(t),\Gamma(t)$) : The parameters vary periodically; this corresponds to the non-autonomous logistic case, describing, from left to right, the respective dynamics of the aquatic mosquito population, adult females, and wild males over time.}
\end{figure}

\begin{figure}[H]
  \centering
  \begin{subfigure}[b]{0.3\textwidth}
    \centering
\includegraphics[width=\textwidth]{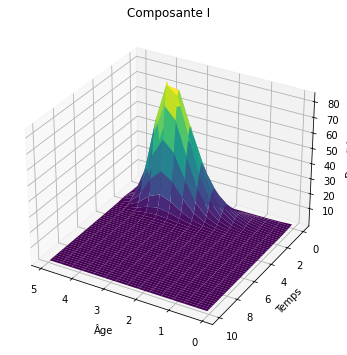}
    \caption{Cuve I}
  \end{subfigure}
  \hfill
  \begin{subfigure}[b]{0.3\textwidth}
    \centering
\includegraphics[width=\textwidth]{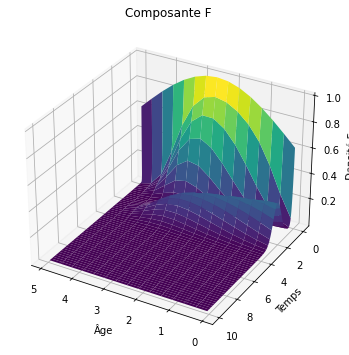}
    \caption{Cuve F}
  \end{subfigure}
  \hfill
  \begin{subfigure}[b]{0.3\textwidth}
    \centering   \includegraphics[width=\textwidth]{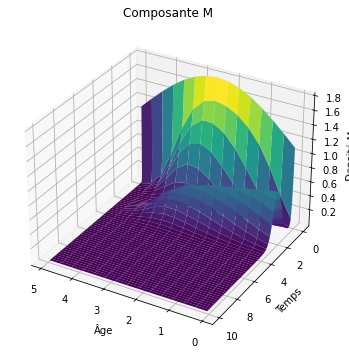}
    \caption{Cuve M}
  \end{subfigure}
  \caption{Controlled density with $k(t),\gamma(t),\Gamma(t)$ : These figures depict the temporal evolution of the system’s dynamics under the application of the control $P$ in \eqref{e3.75}.}
\end{figure}

\begin{figure}[H]
  \centering
  \begin{subfigure}[b]{0.3\textwidth}
    \centering
\includegraphics[width=\textwidth]{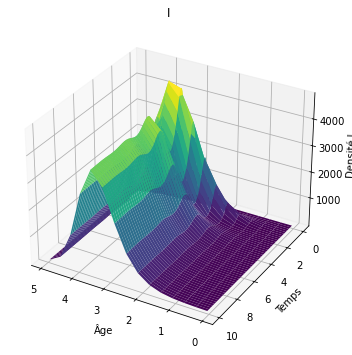}
    \caption{Cuve I}
  \end{subfigure}
  \hfill
  \begin{subfigure}[b]{0.3\textwidth}
    \centering
\includegraphics[width=\textwidth]{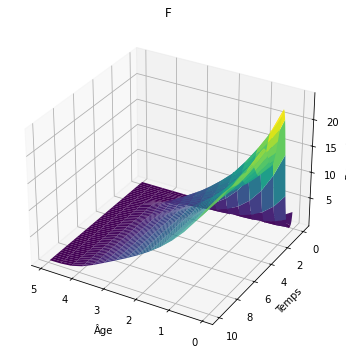}
    \caption{Cuve F}
  \end{subfigure}
  \hfill
  \begin{subfigure}[b]{0.3\textwidth}
    \centering   \includegraphics[width=\textwidth]{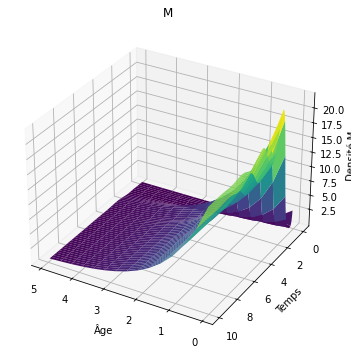}
    \caption{Cuve M}
  \end{subfigure}
  \caption{Uncontrolled density ($K^*,\Gamma^*,\gamma^*$) : These figures correspond to the classical logistic case, where the parameters $(k, \gamma, \Gamma)$ are constant.}
\end{figure}

\begin{figure}[H]
  \centering
  \begin{subfigure}[b]{0.3\textwidth}
    \centering
\includegraphics[width=\textwidth]{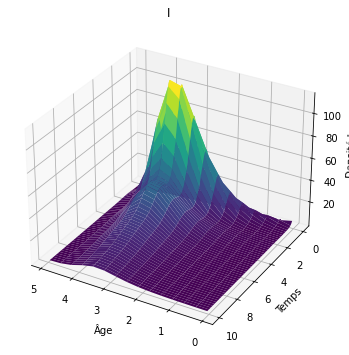}
    \caption{Cuve I}
  \end{subfigure}
  \hfill
  \begin{subfigure}[b]{0.3\textwidth}
    \centering
\includegraphics[width=\textwidth]{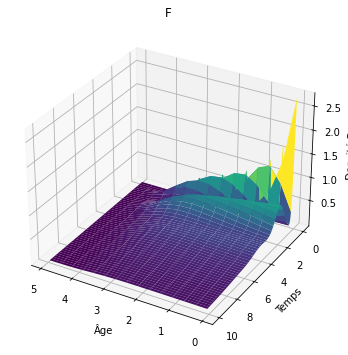}
    \caption{Cuve F}
  \end{subfigure}
  \hfill
  \begin{subfigure}[b]{0.3\textwidth}
    \centering   \includegraphics[width=\textwidth]{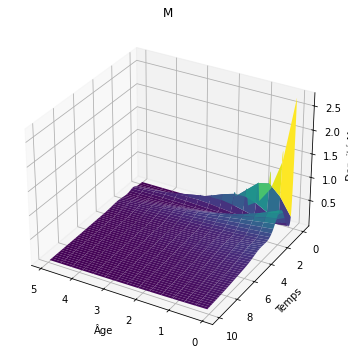}
    \caption{Cuve M}
  \end{subfigure}
  \caption{Controlled density ($K^*,\Gamma^*,\gamma^*$) :These figures depict the temporal evolution of the system’s dynamics under the application of the control $P=P^*$ in \eqref{e3.99}.}
\end{figure}
\subsubsection*{\bf Conclusion on the numerical results}
The numerical results demonstrate the impact of control $P$ on the overall mosquito population dynamics, through its effect on the aquatic stage. The control defined in $\eqref{e3.75}$, which accounts for the temporal variation of resource availability $K(t)$, shows a significant effect and robust effectiveness despite temporal variability. The control $P(t)$ of $\eqref{e3.75}$ combines an immediate correction of the growth rate with a proportional feedback on the normalized demographic pressure. By adapting the control effort to the reference aquatic population size $k_I$, it compensates for the periodic variations of the parameters and tends to stabilize the aquatic dynamics by limiting the amplitude of the forced oscillations. By contrast, in the second case the static control $P^{*}$ applied to an autonomous logistic system, a simple model typically used for analysis, is also effective but less realistic.
\section{\bf Conclusion and outlook}\label{s4}
Despite numerous advances, the study of population dynamics remains a vast field to explore. Existing predator–prey models undoubtedly have strengths, but they also exhibit significant limitations. To better understand biodiversity and investigate species persistence, we propose introducing multi‐species models that incorporate a key factor: age. For example, in epidemiological models, some diseases primarily affect the young while others manifest in older age. Ignoring age in such cases can lead to inaccurate predictions. 

Furthermore, the study of general multi-species non-transitive competition models, and in particular the three- and four-species cases addressed here, opens up vast perspectives, such as controllability, the determination of necessary and sufficient conditions for stabilization, the stabilization of multi-species systems in a non-homogeneous setting; the study of stabilization within a more general framework, extending beyond cyclic cases, and the turnpike property in optimal control, which we cite as examples among many others.  For instance, in order to generalize, the strong connectivity of the interaction graph together with the stability results obtained in the non-transitive case are valuable tools for studying global stability. In particular, if the matrix $A$ defined in \eqref{equation3.3} is strongly connected, the interaction graph contains cycles, which naturally guides the qualitative analysis of the system \eqref{a3.9}. The analysis of the stability of general multi-species competition and/or predator–prey models remains a major challenge for mastering, predicting, and understanding biodiversity

Beyond these applications, employing multi‐species models in forest dynamics is also of clear interest.
\appendix
\section{\bf Stabilization proof for four-species non-transitive competition}\label{annexe:A}
\begin{proof}[of Proposition \ref{proposition3.17}]
From \eqref{equation3.100}, we get 
\begin{align}
\dot \omega_1(\eta_2)=-\phi_2(\eta_2)\phi_3(\eta_3).
\end{align}
We set the fictitious control $\eta_3$ with $c_1>0$ so that 
\begin{align}
\phi_3(\eta_2)=c_1\phi_2(\eta_2).
\end{align}
We have 
\begin{align}\label{a3.89}
\lambda_3(e^{z^4_1+\eta_2}-1))= \phi_3(z^4_1)e^{\eta_2}+\phi_3(\eta_2),
\end{align}
 and thus 
\begin{align}
\dot \omega^1_1(\eta_2)=-c_1\phi^2_2(\eta_2)- \phi_2(\eta_2)\phi_3(z^4_1)-\frac{c_1}{\lambda_3}\phi^2_2(\eta_2)\phi_3(z^4_1)
\end{align} 
Consider the second Lyapunov function for $(z^4_1,\eta_2)$ 
\begin{align}
\omega^1_2(\eta_2,z^4_1)=\omega^1_1(\eta_2)+\lambda_3(e^{z^4_1}-1-z^4_1),
\end{align}
 and 
 \begin{align}
 \dot \omega^1_2(\eta_2,z^4_1)=-c_1\phi^2_2(\eta_2)- \phi_2(\eta_2)\phi_3(z^4_1)-\frac{c_1}{\lambda_3}\phi^2_2(\eta_2)\phi_3(z^4_1)+\phi_3(z^4_1)\left(-\phi_4(\eta_4)+\phi_3(\eta_3)\right)   
 \end{align}
 Next, in the dynamics of $z^4_1$ we consider the fictitious control $\eta_4$ with a constant $c_2$ such that 
 \begin{align}
\phi_4(z^4_1)=c_2\phi_3(z^4_1),
\end{align}
and 
\begin{align}\label{a3.95}
\lambda_4(e^{z^4_2+z^4_1}-1))= \phi_4(z^4_2)e^{z^4_1}+\phi_4(z^4_1).
\end{align}
Then,

 \begin{align}
\dot \omega^1_2(\eta_2,z^4_1)=-c_1\phi^2_2(\eta_2)-c_{2}\phi^2_3(z^4_1)+\phi^2_3(z^4_1)- \phi_2(\eta_2)\phi_3(z^4_1)-\phi_3(z^4_1)\phi_4(z^4_2)-\frac{c_1}{\lambda_3}\phi^2_2(\eta_2)\phi_3(z^4_1)-\frac{c_2}{\lambda_4}\phi^2_3(z^4_1)\phi_4(z^4_2)+\frac{c_1}{\lambda_3}\phi^2_3(z^4_1)\phi_2(\eta_2)
 \end{align}
 \begin{align*}
+c_1\phi_2(\eta_2)\phi_3(z_1).
 \end{align*}
 Let us consider the third Lyapunov function given by 
 \begin{align}\label{aa3.95}
\omega^1_3(\eta_2,z^4_1,z^4_2)=\omega^1_2(\eta_2,z^4_1)+\lambda_4(e^{z^4_2}-1-z^4_2)
 \end{align}
and and we introduce the fictitious control $\eta_1$ to stabilize the dynamics of $z^4_2,$ choosing $c_3>0$ so that condition 
\begin{align}
\phi_1(z^4_2)=c_3\phi_4(z^4_2),
\end{align}
holds.
 It follows from \eqref{aa3.45} that
\begin{align}
\dot \omega^1_3(\eta_2,z^4_1,z^4_2)=-c_1\phi^2_2(\eta_2)-c_{2}\phi^2_3(z^4_1)- \phi_2(\eta_2)\phi_3(z^4_1)-\phi_3(z^4_1)\phi_4(z^4_2)-\frac{c_1}{\lambda_3}\phi^2_2(\eta_2)\phi_3(z^4_1)-\frac{c_2}{\lambda_4}\phi^2_3(z^4_1)\phi_4(z^4_2)+\frac{c_1}{\lambda_3}\phi^2_3(z^4_1)\phi_2(\eta_2)
\end{align}
 \begin{align*}
+\phi^2_3(z^4_1)+c_1\phi_2(\eta_2)\phi_3(z^4_1)+\phi_4(z^4_2)\dot z^4_2.
 \end{align*}
We compute $\dot z^4_2$ in the form 

 \begin{align}
\dot z^4_2=-c_3\phi_4(z^4_2)-\phi_1(z^4_3)-\frac{c_3}{\lambda_1}\phi_4(z^4_2)\phi_1(z^4_3)+c_2\phi_3(z^4_1)+\phi_4(z^4_2)+\frac{c_2}{\lambda_4}\phi_3(z^4_1)\phi_4(z^4_2)-c_1\phi_2(\eta_2)-\phi_3(z^4_1) 
 -\frac{c_1}{\lambda_3}\phi_2(\eta_2)\phi_3(z^4_1)
 \end{align}
 thanks to 
 \begin{align}\label{a3.102}
\phi_1(\eta_1)=c_3\phi_4(z^4_2)+\phi_1(z^4_3)+\frac{c_3}{\lambda_1}\phi_4(z^4_2)\phi_1(z^4_3).
\end{align}
Finally, $\dot \omega^1_3(\eta_2,z^4_1,z^4_2)$ is given by 
\begin{align}
\dot \omega^1_3(\eta_2,z^4_1,z^4_2)=-c_1\phi^2_2(\eta_2)-c_{2}\phi^2_3(z^4_1)-c_{3}\phi^2_4(z^4_2)+\phi^2_3(z^4_1)+\phi^2_4(z^4_2)- \phi_2(\eta_2)\phi_3(z^4_1)-\phi_3(z^4_1)\phi_4(z^4_2)-\frac{c_1}{\lambda_3}\phi^2_2(\eta_2)\phi_3(z^4_1)
\end{align}

\begin{align*}
-\frac{c_2}{\lambda_4}\phi^2_3(z^4_1)\phi_4(z^4_2)+\frac{c_1}{\lambda_3}\phi^2_3(z^4_1)\phi_2(\eta_2)+c_1\phi_2(\eta_2)\phi_3(z^4_1)-\phi_1(z^4_3)\phi_4(z^4_2)-\frac{c_3}{\lambda_1}\phi^2_4(z^4_2)\phi_1(z^4_3)+c_2\phi_3(z^4_1)\phi_4(z^4_2)+\frac{c_2}{\lambda_4}\phi_3(z^4_1)\phi^2_4(z^4_2)
 \end{align*}
 \begin{align*}
-c_1\phi_2(\eta_2)\phi_4(z^4_2)  -\phi_3(z^4_1)\phi_4(z^4_2)-\frac{c_1}{\lambda_3}\phi_2(\eta_2)\phi_3(z^4_1)\phi_4(z^4_2).
 \end{align*}
Let the last Lyapunov control function be defined by
\begin{align}\label{a3.106}
V_4(\eta_2,z^4_1,z^4_2,z^4_3)=\theta \omega^1_3(\eta_2,z^4_1,z^4_2)+\lambda_1(e^{z^4_3}-1-z^4_3),
\end{align}
and 
\begin{align}
\dot V_4(\eta_2,z^4_1,z^4_2,z^4_3)=\theta\dot \omega^1_3(\eta_2,z^4_1,z^4_2)+\phi_1(z^4_3)\dot z^4_3.
\end{align}
We shown that  
\begin{align}
\dot z^4_3=u^*-u-\phi_2(\eta_2)+c_3\phi_4(z^4_2)+\phi_1(z^4_3)+\frac{c_3}{\lambda_1}\phi_4(z^4_2)\phi_1(z^4_3)-c_2\phi_3(z^4_1)-\phi_4(z^4_2)-\frac{c_2}{\lambda_4}\phi_3(z^4_1)\phi_4(z^4_2)+c_1\phi_2(\eta_2)+\phi_3(z^4_1) 
\end{align}
\begin{align*}
+\frac{c_1}{\lambda_3}\phi_2(\eta_2)\phi_3(z^4_1),
\end{align*}
and therefore
\begin{align}
\dot V_4(\eta_2,z^4_1,z^4_2,z^4_3)= -\theta c_1\phi^2_2(\eta_2)-\theta c_{2}\phi^2_3(z^4_1)-\theta c_{3}\phi^2_4(z^4_2)  
\end{align}
\begin{align*}
+\theta[- \phi_2(\eta_2)\phi_3(z^4_1)-\phi_3(z^4_1)\phi_4(z^4_2)-\frac{c_1}{\lambda_3}\phi^2_2(\eta_2)\phi_3(z^4_1)-\frac{c_2}{\lambda_4}\phi^2_3(z^4_1)\phi_4(z^4_2)+\frac{c_1}{\lambda_3}\phi^2_3(z^4_1)\phi_2(\eta_2)+c_1\phi_2(\eta_2)\phi_3(z^4_1)+\phi^2_3(z^4_1)+\phi^2_4(z^4_2)
\end{align*}
\begin{align*}
+c_2\phi_3(z^4_1)\phi_4(z^4_2)+\frac{c_2}{\lambda_4}\phi_3(z^4_1)\phi^2_4(z^4_2)-c_1\phi_2(\eta_2)\phi_4(z^4_2)-\phi_3(z^4_1)\phi_4(z^4_2)-\frac{c_1}{\lambda_3}\phi_2(\eta_2)\phi_3(z^4_1)\phi_4(z^4_2)]
 \end{align*}
 \begin{align*}
\phi_1(z^4_3)[(c_1-1)\phi_2(\eta_2)+(c_3-1-\theta)\phi_4(z^4_2)-(c_2-1)\phi_3(z^4_1)+\phi_1(z^4_3)+\frac{c_3}{\lambda_1}\phi_4(z^4_2)\phi_1(z^4_3)-\frac{c_2}{\lambda_4}\phi_3(z^4_1)\phi_4(z^4_2)+\frac{c_1}{\lambda_3}\phi_2(\eta_2)\phi_3(z^4_1)]
 \end{align*}
 \begin{align*}
+\phi_1(z^4_3)[u^*-u-\theta\frac{c_3}{\lambda_1}\phi^2_4(z^4_2)].
 \end{align*}
 
Then, with a control of the form 

\begin{align}\label{a3.112}
\begin{aligned}
u &=u^*+c_{40}\phi_1(z^4_3)-\theta\frac{c_3}{\lambda_1}\phi^2_4(z^4_2)+(c_3-1-\theta)\phi_4(z^4_2)+(c_1-1)\phi_2(\eta_2)+\frac{c_3}{\lambda_1}\phi_4(z^4_2)\phi_1(z^4_3)-\frac{c_2}{\lambda_4}\phi_3(z^4_1)\phi_4(z^4_2)-(c_2-1)\phi_3(z^4_1)\\
&+\frac{c_1}{\lambda_3}\phi_2(\eta_2)\phi_3(z^4_1)+\frac{\theta }{\phi_1(z^4_3)}\left(\phi^2_4(z^4_2)-c_1\phi_2(\eta_2)\phi_4(z^4_2)+\phi^2_3(z^4_1)\right)\\
&+\frac{\theta\phi_3(z^4_1)}{\phi_1(z^4_3)}\left(\frac{c_1}{\lambda_3}\phi_3(z^4_1)\phi_2(\eta_2)-\frac{c_1}{\lambda_3}\phi^2_2(\eta_2)-\frac{c_2}{\lambda_4}\phi_3(z^4_1)\phi_4(z^4_2)+(c_1-1)\phi_2(\eta_2)+ (c_2-2)\phi_4(z^4_2)+\frac{c_2}{\lambda_4}\phi^2_4(z^4_2)-\frac{c_1}{\lambda_3}\phi_2(\eta_2)\phi_4(z^4_2)\right),
\end{aligned}%
\end{align}

we finally obtain 
\begin{align}
\dot V_4(\eta_2,z^4_1,z^4_2,z^4_3)= -\theta c_1\phi^2_2(\eta_2)-\theta c_{2}\phi^2_3(z^4_1)-\theta c_{3}\phi^2_4(z^4_2)-c_{4}\phi^2_1(z^4_3).
\end{align}
\end{proof}
\begin{proof}[of Proposition \ref{proposition3.18}]
Considering the Lyapunov control function \eqref{a3.106}, and by following the approach of Section \ref{s3.1.2.2} with control \eqref{a3.112} applied to system \eqref{a3.115}, we obtain 

 \begin{align}
\dot V_4(\eta_2,z^4_1,z^4_2,z^4_3)= -\theta c_1\phi^2_2(\eta_2)-\theta c_{2}\phi^2_3(z^4_1)-\theta c_{3}\phi^2_4(z^4_2)-c_{4}\phi^2_1(z^4_3)
\end{align}
\begin{align*}
\underbrace{\left(\phi_1(z^4_3)-\theta\phi_4(z^4_2)\right)}_{A_1}\left(\hat{\phi}_1-\phi_1(\eta_1)\right) \underbrace{-\phi_1(z^4_3)}_{A_2}\left(\hat{\phi}_2-\phi_2(\eta_2)\right)+\underbrace{\left(\phi_1(z^4_3)-\theta\phi_4(z^4_2)+\theta\phi_3(z^4_1)-\theta\phi_2(\eta_2)\right)}_{A_3}\left(\hat{\phi}_3-\phi_3(\eta_3)\right)\\+\underbrace{\left(\theta\phi_4(z^4_2)-\phi_1(z^4_3)-\theta\phi_3(z^4_1)\right)}_{A_4}\left(\hat{\phi}_4-\phi_4(\eta_4)\right)+\mathcal{R}_1
\end{align*}
where 
\begin{align}
    \mathcal{R}_1=\phi_1(z^4_3)(-\phi_1(z^4_3)+\theta\frac{c_3}{\lambda_1}\phi^2_4(z^4_2)-(c_3-1-\theta)\phi_4(z^4_2)-(c_1-1)\phi_2(\eta_2)-\frac{c_3}{\lambda_1}\phi_4(z^4_2)\phi_1(z^4_3)+\frac{c_2}{\lambda_4}\phi_3(z^4_1)\phi_4(z^4_2)+(c_2-1)\phi_3(z^4_1)) 
\end{align}
\begin{align*}
-\frac{c_1}{\lambda_3}\phi_2(\eta_2)\phi_3(z^4_1))+\theta c_1\phi_2(\eta_2)\phi_4(z^4_2) +\theta c_1\phi^2_2(\eta_2)+\theta c_{2}\phi^2_3(z^4_1)-\theta\phi^2_3(z^4_1)+\theta c_{3}\phi^2_4(z^4_2)-\theta\phi^2_4(z^4_2)
\end{align*}
\begin{align*}
-\theta\phi_3(z^4_1)\left(\frac{c_1}{\lambda_3}\phi_3(z^4_1)\phi_2(\eta_2)-\frac{c_1}{\lambda_3}\phi^2_2(\eta_2)-\frac{c_2}{\lambda_4}\phi_3(z^4_1)\phi_4(z^4_2)+(c_1-1)\phi_2(\eta_2)+ (c_2-2)\phi_4(z^4_2)+\frac{c_2}{\lambda_4}\phi^2_4(z^4_2)-\frac{c_1}{\lambda_3}\phi_2(\eta_2)\phi_4(z^4_2)\right)
\end{align*}

\begin{align*}
+\left(\phi_1(z^4_3)-\theta\phi_4(z^4_2)\right)\phi_1(\eta_1) -\phi_1(z^4_3)\phi_2(\eta_2)+\left(\phi_1(z^4_3)-\theta\phi_4(z^4_2)+\theta\phi_3(z^4_1)-\theta\phi_2(\eta_2)\right)\phi_3(\eta_3)+\left(\theta\phi_4(z^4_2)-\phi_1(z^4_3)-\theta\phi_3(z^4_1)\right)\phi_4(\eta_4).
\end{align*}
With 
\begin{align}\label{ea3.130}
\begin{cases}
\phi_1(\eta_1)=c_3\phi_4(z^4_2)+\phi_1(z^4_3)+\frac{c_3}{\lambda_1}\phi_4(z^4_2)\phi_1(z^4_3),\\
\phi_3(\eta_3)=c_1\phi_2(\eta_2)+\phi_3(z^4_1)+\frac{c_1}{\lambda_3}\phi_3(z^4_1)\phi_2(\eta_2),\\
\phi_4(\eta_4)=c_2\phi_3(z^4_1)+\phi_4(z^4_2)+\frac{c_2}{\lambda_4}\phi_3(z^4_1)\phi_4(z^4_2),
\end{cases}
\end{align}
from \eqref{a3.89}-\eqref{a3.95}-\eqref{a3.102}, we get
\begin{align}
\mathcal{R}_1=0.
\end{align}
Then 
 \begin{align}
\dot V_4(\eta_2,z^4_1,z^4_2,z^4_3)= -\theta c_1\phi^2_2(\eta_2)-\theta c_{2}\phi^2_3(z^4_1)-\theta c_{3}\phi^2_4(z^4_2)-c_{4}\phi^2_1(z^4_3)
\end{align}
\begin{align*}
+A_1\left(\hat{\phi}_1-\phi_1(\eta_1)\right) +A_2\left(\hat{\phi}_2-\phi_2(\eta_2)\right)+A_3\left(\hat{\phi}_3-\phi_3(\eta_3)\right)+A_4\left(\hat{\phi}_4-\phi_4(\eta_4)\right).
\end{align*}

Thanks to \eqref{a3.63} and from Lemma \ref{lem3.11}, the derivative of the general Lyapunov function \eqref{a3.114} is given by

\begin{align}
\dot V_G(\eta,\psi)
\leq-\theta c_1\phi^2_2(\eta_2)-\theta c_{2}\phi^2_3(z^4_1)-\theta c_{3}\phi^2_4(z^4_2)-c_{4}\phi^2_1(z^4_3)+\left(C_1\vert\phi_1(\eta_1)+\lambda_1\vert-\gamma_1\right)(e^{G_1}-1)
\end{align}
\begin{align*}
+\left(C_2\vert\phi_2(\eta_2)+\lambda_2\vert-\gamma_2\right)(e^{G_2}-1)+\left(C_3\vert\phi_3(\eta_3)+\lambda_3\vert-\gamma_3\right)(e^{G_3}-1)+\left(C_4\vert\phi_4(\eta_4)+\lambda_4\vert-\gamma_4\right)(e^{G_4}-1).
\end{align*}

Selecting  $\eta$ as follows
\begin{align}
\begin{cases}
\eta_1\leq \ln\left(\frac{\gamma_1}{C_1\lambda_1}\right),\\
\\\eta_2\leq \ln\left(\frac{\gamma_2}{C_2\lambda_2}\right),\\
\\\eta_3\leq \ln\left(\frac{\gamma_3}{C_3\lambda_3}\right),\\
\\\eta_4\leq \ln\left(\frac{\gamma_4}{C_4\lambda_4}\right),
\end{cases}\qquad\text{with}\; \begin{cases}
\gamma_1>C_1\lambda_1,\\
\\\gamma_2>C_2\lambda_2,\\
\\\gamma_3>C_3\lambda_3,\\
\\\gamma_4>C_4\lambda_4,
\end{cases}
\end{align}
yields the desired solution. 
\end{proof}
\section{\bf Explicit form of the control-Lyapunov function for \texorpdfstring{$N$}{N} species}\label{annexe:B}
\section*{Control}
Using \eqref{equation3.123}-\eqref{ea3.130}, equation
\begin{align}\label{ea3.163}
\begin{cases}
\phi_1(\eta_1)=c_N\phi_N(z^N_{N-2})+\phi_1(z^N_{N-1})+\frac{c_N}{\lambda_1}\phi_1(z^N_{N-1})\phi_N(z^N_{N-2}),\\
\\\phi_3(\eta_3)=c_1\phi_2(\eta_2)+\phi_3(z^N_1)+\frac{c_1}{\lambda_3}\phi_3(z^N_1)\phi_2(\eta_2),\\
\\\phi_4(\eta_4)=c_2\phi_3(z^N_1)+\phi_4(z^N_2)+\frac{c_2}{\lambda_4}\phi_3(z^N_1)\phi_4(z^N_2)\\
\\\phi_5(\eta_5)=c_3\phi_4(z^N_2)+\phi_5(z^N_3)+\frac{c_3}{\lambda_5}\phi_5(z^N_3)\phi_4(z^N_2)\\
\vdots\\
\vdots
\\\phi_N(\eta_N)=c_{N-2}\phi_{N-1}(z^N_{N-3})+\phi_N(z^N_{N-2})+\frac{c_{N-2}}{\lambda_N}\phi_N(z^N_{N-2})\phi_{N-1}(z^N_{N-3}),
\end{cases}
\end{align}
follows at the $N$ step by induction. Substituting $\eqref{ea3.163}$ into $u$ yields 
\begin{align}
\begin{aligned}
u \;=&\; u^*
+\frac{1}{\phi_1(z^N_{N-1})}\Bigg(
c_{N}\,\phi_1^2(z^N_{N-1})
+\theta\,c_1\,\phi_2^2(\eta_2)
+\theta \sum_{i=1}^{N-2} c_{i+1}\,\phi_{i+2}^2(z^N_i)\Bigg) \\
&+\frac{\overbrace{\phi_1(z^N_{N-1})-\theta\,\phi_{N}(z^N_{N-2})}^{A_1}}{\phi_1(z^N_{N-1})}\,(c_N\phi_N(z^N_{N-2})+\phi_1(z^N_{N-1})+\frac{c_N}{\lambda_1}\phi_1(z^N_{N-1})\phi_N(z^N_{N-2}))
-\frac{\overbrace{\phi_1(z^N_{N-1})}^{A_2}}{\phi_1(z^N_{N-1})}\phi_2(\eta_2) \\
&+
\frac{\overbrace{\displaystyle
\theta \sum_{i=\max(1,k-3)}^{N-2} [(-1)^{\,i+\varepsilon_k}\,\phi_{i+2}(z^N_i)
+(-1)^{\,N+1-\varepsilon_k}\,\phi_1(z^N_{N-1})-\theta\phi_2(\eta_2)]}^{A_3}
}{\phi_1(z^N_{N-1})}\,(c_1\phi_2(\eta_2)+\phi_3(z^N_1)+\frac{c_1}{\lambda_3}\phi_3(z^N_1)\phi_2(\eta_2))\\
&+\sum_{k=4}^{N}
\frac{\overbrace{\displaystyle
\theta \sum_{i=\max(1,k-3)}^{N-2} [(-1)^{\,i+\varepsilon_k}\,\phi_{i+2}(z^N_i)
+(-1)^{\,N+1-\varepsilon_k}\,\phi_1(z^N_{N-1})]}^{A_k}
}{\phi_1(z^N_{N-1})}\,(c_{k-2}\phi_{k-1}(z_{k-3})+\phi_k(z^N_{k-2})+\frac{c_{k-2}}{\lambda_k}\phi_k(z^N_{k-2})\phi_{k-1}(z^N_{k-3})).
\end{aligned}
\end{align}
\section*{Lyapunov function} 
We set, by definition
\begin{align}
\begin{aligned}
V_N(\eta_2,z^N_1,\cdots,z^N_{N-1})&= \theta \lambda_2(e^{\eta_{2}}-1-\eta_{2})+\theta \sum_{i=1}^{N-2}\lambda_{i+2}(e^{z^N_{i}}-1-z^N_{i})+\lambda_1(e^{z^N_{N-1}}-1-z^N_{N-1}).
\end{aligned}
\end{align}
For every $N>4$, the following form holds: 
\begin{align}
\begin{aligned}
V_N(\eta_2,z^{N}_1,\cdots,z^{N}_{N-1})=&V_{N-1}(\eta_2,z^{N-1}_1,\cdots,z^{N-1}_{N-2})-\theta\lambda_{N-1}(e^{z^{N-1}_{N-3}}-1-z^{N-1}_{N-3})+\theta\lambda_{N-1}(e^{z^{N}_{N-3}}-1-z^{N}_{N-3})\\
&-\lambda_{1}(e^{z^{N-1}_{N-2}}-1-z^{N-1}_{N-2})+\lambda_{1}(e^{z^{N}_{N-1}}-1-z^{N}_{N-1})+\theta\lambda_{N}(e^{z^{N}_{N-2}}-1-z^{N}_{N-2}).
\end{aligned}
\end{align}
We define the general explicit form of the derivative of the Lyapunov function for any  $N$:
\begin{itemize}
\item[Case 1:] if $N>3$ is odd then
\begin{align}
\begin{aligned}
\dot V_N(\eta_2, z^N_1,\cdots,z^N_{N-1})=
&\left(\phi_1(z^N_{N-1})-\theta\phi_N(z^N_{N-2})\right)\phi_1(\eta_1)-\phi_2(\eta_2)+\left(-\theta\phi_2(\eta_{2})+\theta\displaystyle\sum_{i=1}^{N-2}(-1)^{i+1}\phi_{i+2}(z^N_i)-\phi_1(z^N_{N-1})\right)\phi_3(\eta_3)\\
&+\left(\theta\displaystyle\sum_{i=1}^{N-2}(-1)^{i}\phi_{i+2}(z^N_i)+\phi_1(z^N_{N-1})\right)\phi_4(\eta_4)+\left(\theta\displaystyle\sum_{i=2}^{N-2}(-1)^{i+1}\phi_{i+2}(z^N_i)-\phi_1(z^N_{N-1})\right)\phi_5(\eta_5)+\cdots\cdots\\
&+\left(\theta\displaystyle\sum_{i=N-4}^{N-2}(-1)^{i}\phi_{i+2}(z_i)+\phi_1(z^N_{N-1})\right)\phi_{N-1}(\eta_{N-1})+\left(\theta\displaystyle\sum_{i=N-3}^{N-2}(-1)^{i+1}\phi_{i+2}(z^N_i)-\phi_1(z^N_{N-1})\right)\phi_{N}(\eta_{N})\\
&+\phi_1(z^N_{N-1})(u-u^*).
\end{aligned}
\end{align}
\item[Case 2:] if $N>4$ is even then
\begin{align}
\begin{aligned}
\dot V_N(\eta_2, z^N_1,\cdots,z^N_{N-1})=
&\left(\phi_1(z^N_{N-1})-\theta\phi_N(z^N_{N-2})\right)\phi_1(\eta_1)-\phi_2(\eta_2)+\left(-\theta\phi_2(\eta_{2})+\theta\displaystyle\sum_{i=1}^{N-2}(-1)^{i+1}\phi_{i+2}(z^N_i)+\phi_1(z^N_{N-1})\right)\phi_3(\eta_3)\\
&+\left(\theta\displaystyle\sum_{i=1}^{N-2}(-1)^{i}\phi_{i+2}(z^N_i)-\phi_1(z^N_{N-1})\right)\phi_4(\eta_4)+\left(\theta\displaystyle\sum_{i=2}^{N-2}(-1)^{i+1}\phi_{i+2}(z^N_i)+\phi_1(z^N_{N-1})\right)\phi_5(\eta_5)+\cdots\cdots\\
&+\left(\theta\displaystyle\sum_{i=N-4}^{N-2}(-1)^{i+1}\phi_{i+2}(z^N_i)+\phi_1(z_{N-1})\right)\phi_{N-1}(\eta_{N-1})+\left(\theta\displaystyle\sum_{i=N-3}^{N-2}(-1)^{i}\phi_{i+2}(z^N_i)-\phi_1(z^N_{N-1})\right)\phi_{N}(\eta_{N})\\
&+\phi_1(z^N_{N-1})(u-u^*).
\end{aligned}
\end{align}   
\end{itemize}
\paragraph{\bf Acknowledgement}
\textbf{The authors wish to thank Prof. Enrique Zuazua  for his comments, suggestions and for fruitful discussions.}\\
\textbf{The author Yacouba Simpore is supported by the Alexander von Humboldt Foundation through an Alexander von Humboldt research fellowship.}
\bibliographystyle{plain}
 \bibliography{biblio_1}
 \vfill 
\end{document}